\documentclass[british,a4paper,11pt,final]{article}
\makeatletter

\PassOptionsToPackage{pdftex,dvipsnames}{xcolor}
\RequirePackage{xcolor}

\usepackage[utf8]{inputenc}
\usepackage{cmap}
\usepackage[T1]{fontenc}

\iftrue
\usepackage[rm=lining,sf=lining,tt={lining,tabular}]{cfr-lm}
\AtBeginDocument{%
  \DeclareFontShape{T1}{clm\cfrlm@rmpt\cfrlm@rmol}{bx}{sc}{<->ssub*fcm/b/sc}{}}

\else\iftrue
\usepackage[proportional,semibold,type1]{libertineRoman}
\usepackage[proportional,type1]{biolinum}
\usepackage[semibold,type1,scale=.88888]{sourcecodepro}

\else
\usepackage{tgtermes,tgheros,tgcursor}
\usepackage{fontaxes,textcomp}
\fi\fi

\usepackage[final]{microtype}
\usepackage[shorthands=off,safe=none,math=normal]{babel}

\usepackage{mathtools}
\mathtoolsset{mathic=true}

\usepackage[notext,lcgreekalpha,upint]{stix}
\usepackage[cal=cm,scr=boondoxo]{mathalfa}

\let\circ\vysmwhtcircle

\chardef\fml@textampersand\the\&
\mathchardef\fml@mathampersand\mathcode`\&
\protected\def\&{%
  \ifmmode
    \expandafter\fml@mathampersand
  \else
    \expandafter\fml@textampersand
  \fi}

\newcommand*\obeyampersands{\catcode`\&=\active}
\newcommand*\texampersand{&}
{\obeyampersands\global\let&=\texampersand}
\mathcode`\&="8000\relax
\AtBeginDocument{\obeyampersands}

\protected\def\fml@dollar{\csname\ifmmode)\else(\fi\space\endcsname}
\newcommand*\obeydollars{\catcode`\$=\active}
{\obeydollars\global\let$=\fml@dollar}%
\mathcode`\$="8000\relax
\AtBeginDocument{\obeydollars}

\newcommand*\noic{\sb{}\kern-\scriptspace\relax}

\usepackage{ltxcmds}
\usepackage{etoolbox}

\protected\def\fml@opt@sb{%
  \ltx@ifnextchar@nospace[\fml@opt@sb@{}}%
\def\fml@opt@sb@[#1]{_{#1}}

\def\do#1{%
  \protected\csdef{#1s}{#1\fml@opt@sb}%
  \protected\csdef{#1rm}{\mathrm#1\fml@opt@sb}%
  \protected\csdef{#1bf}{\mathbf#1\fml@opt@sb}%
  \protected\csdef{#1sf}{\mathsfit#1\fml@opt@sb}%
  \protected\csdef{#1bb}{\mathbb#1\fml@opt@sb}%
  \protected\csdef{#1cal}{\mathcal#1\fml@opt@sb}%
  \protected\csdef{#1scr}{\mathscr#1\fml@opt@sb}%
  \protected\csdef{#1frak}{\mathfrak#1\fml@opt@sb}%
}
\forcsvlist\csundef{ss,SS}%
\docsvlist{ A,B,C,D,E,F,G,H,I,J,K,L,M,N,O,P,Q,R,S,T,U,V,W,X,Y,Z,
            a,b,c,d,e,f,g,h,i,j,k,l,m,n,o,p,q,r,s,t,u,v,w,x,y,z, }

\def\fml@tmp#1=#2{\ifdef#1{}{\let#1#2}}
\def\do#1{\fml@tmp#1}
\docsvlist{ \Alpha=A, \Beta=B, \Epsilon=E, \Zeta=Z, \Eta=H, \Iota=I, \Kappa=K,
            \Mu=M, \Nu=N, \Omicron=O, \Rho=P, \Tau=T, \Chi=X, \omicron=o, }
\def\do#1{%
  \letcs\fml@tmp{#1}%
  \csletcs{#1}{var#1}%
  \cslet{var#1}\fml@tmp}
\docsvlist{ epsilon, phi, }%

\def\fml@tmp#1{\ltx@LocToksA\expandafter{\ltx@gobble#1}}
\def\do#1{%
  \expandafter\fml@tmp\expandafter{\string#1}%
  \protected\csdef{\the\ltx@LocToksA s}{\mathit#1\fml@opt@sb}%
  \protected\csdef{\the\ltx@LocToksA rm}{\mathrm#1\fml@opt@sb}%
  \protected\csdef{\the\ltx@LocToksA bf}{\mathbf#1\fml@opt@sb}%
  \protected\csdef{\the\ltx@LocToksA sf}{\mathsfit#1\fml@opt@sb}%
}
\docsvlist{ \Alpha, \Beta,    \Gamma,   \Delta, \Epsilon, \Zeta,
            \Eta,   \Theta,   \Iota,    \Kappa, \Lambda,  \Mu,
            \Nu,    \Xi,      \Omicron, \Pi,    \Rho,     \Sigma,
            \Tau,   \Upsilon, \Phi,     \Chi,   \Psi,     \Omega,
            \alpha, \beta,    \gamma,   \delta, \epsilon, \zeta,
            \eta,   \theta,   \iota,    \kappa, \lambda,  \mu,
            \nu,    \xi,      \omicron, \pi,    \rho,     \sigma,
            \tau,   \upsilon, \phi,     \chi,   \psi,     \omega, }
\docsvlist{ \nabla, \partial, \digamma,
            \varepsilon, \vartheta, \varpi, \varrho, \varsigma, \varphi, }

\newmuskip\verythinmuskip
\AtBeginDocument{%
  \thickmuskip=5mu plus 3mu minus 1mu\relax%
  \medmuskip=4mu plus 2mu minus 3mu\relax%
  \thinmuskip=3mu plus 1mu minus 1mu\relax
  \verythinmuskip=1.5mu plus 1mu minus 1mu\relax}

\newrobustcmd*\tab{\mathchoice
  {\mskip\thickmuskip}%
  {\mskip\medmuskip}%
  {\mskip\thinmuskip}%
  {\mskip\verythinmuskip}}%

\usepackage{mleftright}

\newrobustcmd*\st[1][\middle]{%
  \nonscript\tab#1\vert\allowbreak\nonscript\tab\mathopen{}}

\usepackage{tikz-cd}
\usepackage{relsize}

\let\tikz@ensure@dollar@catcode\relax

\usetikzlibrary{arrows}
\usetikzlibrary{shapes.geometric}

\tikzset{%
  loop/.style={%
    every curve to/.append style={every loop},%
    to path={\pgfextra{\let\tikztotarget=\tikztostart}\tikz@to@curve@path},
  },
  every loop/.append style={looseness=6,min distance=1ex},
  loop right/.append style={out=30,in=-30},
  loop above/.append style={out=120,in=60},
  loop left/.append style={out=210,in=150},
  loop below/.append style={out=300,in=240},
}

\tikzcdset{%
  every label/.style={%
    /tikz/auto,
    /tikz/font={\relsize{-2}},%
    /tikz/inner sep=+0.5ex},
}

\tikzcdset{%
  arrow style=tikz,
  tikzcd to/.tip={to},
  tikzcd bar/.tip={|},%
  tikzcd left hook/.tip={left hook},
  tikzcd implies/.tip={implies},
}

\newdimen\fmlcd@dim
\tikzset{%
  lax/.default={45}{3ex},
  lax/.style 2 args={%
    draw,thin,double,arrows={-implies},
    to path={%
      \pgfextra{%
        \pgfmathparse{#1}\let\fmlcd@buf\pgfmathresult
        \pgfmathsetlength\fmlcd@dim{(#2)/2}}%
      (barycentric cs:\tikztostart=1,\tikztotarget=1)
      +(\fmlcd@buf:\fmlcd@dim)
      -- +(\fmlcd@buf:-\fmlcd@dim) \tikztonodes},
  },
  oplax/.default={45}{3ex},
  oplax/.style 2 args={lax={#1}{#2},arrows={implies-}},
  bilax/.default={45}{3ex},
  bilax/.style 2 args={lax={#1}{#2},arrows={implies-implies}},
}

\usepackage{url}
\usepackage{xargs}

\usepackage[framemethod=TikZ]{mdframed}

\usepackage{algorithmicx}

\usepackage{subcaption} %
\usepackage{bussproofs} %

\usepackage{rotating}

\PassOptionsToPackage{pdfpagelabels,colorlinks,hypertexnames=false}{hyperref}
\RequirePackage{hyperref}

\hypersetup{
    colorlinks,
    urlcolor={black},
    citecolor={black},
    linkcolor={black}
}

\usepackage{amsthm}
\usepackage[nosort,capitalise]{cleveref}
\crefname{figure}{Figure}{Figures}
\Crefname{figure}{Figure}{Figures}

\renewcommand\labelenumi{(\roman{enumi})}
\renewcommand\theenumi\labelenumi

\newcommand{\N}{\mathord{\mathbb N}}

\newcommand{\pow}[1]{\mathcal{P}#1}

\newcommand{\suchthat}{\mathrel{\mid\,}}

\newcommand{\isom}{\cong}

\newcommand{\One}{\mathord\mathbf {1}}

\newcommand{\after}{\mathbin{\circ}}

\newcommand{\opcat}[1]{{#1}^{\sf op}}

\newcommand{\tensor}{\mathbin{\otimes}}
\newcommand{\limp}{\mathbin{\multimap}}
\newcommand{\llpar}{\upand}

\newlength{\lenpa}

\newtheorem{theorem}{Theorem}[section] 

\newtheorem{lemma}[theorem]{Lemma}
\newtheorem{proposition}[theorem]{Proposition}
\newtheorem{corollary}[theorem]{Corollary}

\theoremstyle{definition}
\newtheorem{definition}{Definition}

\newtheorem{remark}{Remark}

\makeatletter \renewenvironment{proof}[1][\proofname]
{\par\pushQED{\qed}\normalfont\topsep6\p@\@plus6\p@\relax
\trivlist\item[\hskip\labelsep\bfseries#1\@addpunct{.}]\ignorespaces}
{\popQED\endtrivlist\@endpefalse}
\makeatother

\newcommand{\fctl}[3]{\begin{tikzcd}[ampersand
    replacement=\&, cramped, sep=small]#1\colon #2 \arrow[r] \& #3 \end{tikzcd}}
\newcommand{\fromto}[2]{\begin{tikzcd}[ampersand
    replacement=\&, cramped, sep=small
]#1\arrow[r]\& #2\end{tikzcd}}

\tikzstyle{st2b}=[rectangle, minimum height=6pt, minimum width=-0.1pt,inner sep=0pt,draw]
\newcommand{\smllm}{\mllm}

\newcommand{\ProofLabel}[1]{\LeftLabel{\small{#1}}}

\newcommand{\parr}{\llpar}

\newcommand{\otimesT}{\(\otimes\)}
\newcommand{\parT}{\(\llpar\)}
\newcommand{\exchT}{\textsf{exch}}
\newcommand{\idT}{\textsf{id}}
\newcommand{\axT}{\textsf{ax}}
\newcommand{\cutT}{\textsf{cut}}
\newcommand{\iUpT}{\textsf{i}\(\uparrow\)}
\newcommand{\iDownT}{\textsf{i}\(\downarrow\)}
\newcommand{\aiUpT}{\textsf{ai}\(\uparrow\)}
\newcommand{\aiDownT}{\textsf{ai}\(\downarrow\)}
\newcommand{\switchT}{\textsf{s}}
\newcommand{\sigmaSwitchT}{\(\sigma\textsf{s}\)}
\newcommand{\involT}{\(\equiv_{\textsf{DeMor}}\)}
\newcommand{\tinvolT}{\(\scriptscriptstyle\equiv_{\textsf{\tiny DeMor}}\)}

\newcommand{\sigmaUpT}{\(\sigma\!\!\uparrow\)}
\newcommand{\sigmaDownT}{\(\sigma\!\!\downarrow\)}
\newcommand{\alphaUpT}{\(\alpha\!\!\uparrow\)}
\newcommand{\alphaDownT}{\(\alpha\!\!\downarrow\)}
\newcommand{\otimesL}{\ProofLabel{\otimesT}}
\newcommand{\parL}{\ProofLabel{\parT}}
\newcommand{\exchL}{\ProofLabel{\exchT}}
\newcommand{\idL}{\ProofLabel{\idT}}
\newcommand{\axL}{\ProofLabel{\axT}}
\newcommand{\cutL}{\ProofLabel{\cutT}}
\newcommand{\iUpL}{\ProofLabel{\iUpT}}
\newcommand{\iDownL}{\ProofLabel{\iDownT}}
\newcommand{\aiUpL}{\ProofLabel{\aiUpT}}
\newcommand{\aiDownL}{\ProofLabel{\aiDownT}}
\newcommand{\switchL}{\ProofLabel{\switchT}}
\newcommand{\sigmaSwitchL}{\ProofLabel{\sigmaSwitchT}}
\newcommand{\involL}{\ProofLabel{\involT}}
\newcommand{\demorL}{\involL{}}
\newcommand{\sigmaUpL}{\ProofLabel{\sigmaUpT}}
\newcommand{\sigmaDownL}{\ProofLabel{\sigmaDownT}}
\newcommand{\alphaUpL}{\ProofLabel{\alphaUpT}}
\newcommand{\alphaDownL}{\ProofLabel{\alphaDownT}}
\newcommand{\totimesL}{\ProofLabel{\tiny\otimesT}}
\newcommand{\tparL}{\ProofLabel{\tiny\parT}}
\newcommand{\texchL}{\ProofLabel{\tiny\exchT}}
\newcommand{\tidL}{\ProofLabel{\tiny\idT}}
\newcommand{\taxL}{\ProofLabel{\tiny\axT}}
\newcommand{\tcutL}{\ProofLabel{\tiny\cutT}}

\newcommand{\taiUpL}{\ProofLabel{\tiny\aiUpT}}
\newcommand{\taiDownL}{\ProofLabel{\tiny\aiDownT}}

\newcommand{\tsigmaSwitchL}{\ProofLabel{\tiny\sigmaSwitchT}}

\newcommand{\tinvolL}{\ProofLabel{\tinvolT}}
\newcommand{\tdemorL}{\tinvolL{}}
\newcommand{\tsigmaUpL}{\ProofLabel{\tiny\sigmaUpT}}
\newcommand{\tsigmaDownL}{\ProofLabel{\tiny\sigmaDownT}}

\newcommand{\tsigmaalphaL}{\ProofLabel{$\scriptscriptstyle\alpha ,\sigma$}}
\newcommand{\fDown}[1]{f^{\downarrow}_{#1}}
\newcommand{\fDownT}[1]{\(\fDown{#1}\)}
\newcommand{\fiDown}[1]{f^{i\downarrow}_{#1}}
\newcommand{\fiDownT}[1]{\(\fiDown{#1}\)}
\newcommand{\fUp}[1]{f^{\uparrow}_{#1}}
\newcommand{\fUpT}[1]{\(\fUp{#1}\)}
\newcommand{\fiUp}[1]{f^{i\uparrow}_{#1}}
\newcommand{\fiUpT}[1]{\(\fiUp{#1}\)}

\newcommand{\fSwitch}[1]{f^{s}_{#1}}
\newcommand{\fSwitchT}[1]{\(\fSwitch{#1}\)}
\newcommand{\fSigmaSwitch}[1]{f^{\sigma s}_{#1}}
\newcommand{\fSigmaSwitchT}[1]{\(\fSigmaSwitch{#1}\)}

\newcommand{\fSigmaUp}[1]{f^{\sigma\!\uparrow}_{#1}}
\newcommand{\fSigmaDown}[1]{f^{\sigma\!\downarrow}_{#1}}
\newcommand{\fAlphaUp}[1]{f^{\alpha\!\uparrow}_{#1}}
\newcommand{\fAlphaDown}[1]{f^{\alpha\!\downarrow}_{#1}}
\newcommand{\fSigmaUpT}[1]{\(\fSigmaUp{#1}\)}
\newcommand{\fSigmaDownT}[1]{\(\fSigmaDown{#1}\)}
\newcommand{\fAlphaUpT}[1]{\(\fAlphaUp{#1}\)}
\newcommand{\fAlphaDownT}[1]{\(\fAlphaDown{#1}\)}

\newcommand{\fAlphaPar}[1]{\fAlphaUp{#1}}

\newcommand{\fSigmaParT}[1]{\fSigmaUpT{#1}}
\newcommand{\fSigmaTensorT}[1]{\fSigmaDownT{#1}}
\newcommand{\fAlphaParT}[1]{\fAlphaUpT{#1}}
\newcommand{\fAlphaTensorT}[1]{\fAlphaDownT{#1}}

\newcommand{\rhoTensor}[1]{\rho^{\tensor}_{#1}}
\newcommand{\rhoTensorT}[1]{\(\rhoTensor{#1}\)}

\newcommand{\lambdaPar}[1]{\lambda^{\parr}_{#1}}

\newcommand{\fExch}[1]{f^{\textsf{exch}}_{#1}}
\newcommand{\fExchT}[1]{\(\fExch{#1}\)}
\newcommand{\fPar}[1]{f^{\parr}_{#1}}
\newcommand{\fParT}[1]{\(\fPar{#1}\)}
\newcommand{\fTensor}[1]{f^{\tensor}_{#1}}
\newcommand{\fTensorT}[1]{\(\fTensor{#1}\)}
\newcommand{\fCut}[1]{f^{\textsf{cut}}_{#1}}
\newcommand{\fCutT}[1]{\(\fCut{#1}\)}

\newcommand{\axD}{
  \begin{prooftree}
    \AxiomC{}
    \axL\UnaryInfC{\(\ttile \dual{\alpha}, \alpha\)}
  \end{prooftree}
}
\newcommand{\idD}{
  \begin{prooftree}
    \AxiomC{}
    \idL\UnaryInfC{\(\ttile \dual{A}, A\)}
  \end{prooftree}
}
\newcommand{\exchD}{
  \begin{prooftree}
    \AxiomC{\(\ttile \Gamma, A, B, \Delta\)}
    \exchL\UnaryInfC{\(\ttile \Gamma, B, A, \Delta\)}
  \end{prooftree}
}
\newcommand{\parD}{
  \begin{prooftree}
    \AxiomC{\(\ttile \Gamma, A, B, \Delta\)}
    \parL\UnaryInfC{\(\ttile \Gamma, A \parr B, \Delta\)}
  \end{prooftree}
}
\newcommand{\otimesD}{
  \begin{prooftree}
    \AxiomC{\(\ttile \Gamma, A\)} \AxiomC{\(\ttile B, \Delta\)}
    \otimesL\BinaryInfC{\(\ttile \Gamma, A \otimes B, \Delta\)}
  \end{prooftree}
}
\newcommand{\cutD}{
  \begin{prooftree}
    \AxiomC{\(\ttile \Gamma, A\)} \AxiomC{\(\ttile \dual{A}, \Delta\)}
    \cutL\BinaryInfC{\(\ttile \Gamma, \Delta\)}
  \end{prooftree}
}
\newcommand{\aiDownAxiomD}{
  \begin{prooftree}
    \AxiomC{}
    \aiDownL\UnaryInfC{\(\dual{\alpha} \parr \alpha\)}
  \end{prooftree}
}
\newcommand{\aiDownD}{
  \begin{prooftree}
    \AxiomC{\(S\{B\}\)}
    \aiDownL\UnaryInfC{\(S\{B \otimes (\dual{\alpha} \parr \alpha)\}\)}
  \end{prooftree}
}

\newcommand{\iDownAxiomD}{
  \begin{prooftree}
    \AxiomC{}
    \iDownL\UnaryInfC{\(\dual{A} \parr A\)}
  \end{prooftree}
}
\newcommand{\iDownD}{
  \begin{prooftree}
    \AxiomC{\(S\{B\}\)}
    \iDownL\UnaryInfC{\(S\{B \otimes (\dual{A} \parr A)\}\)}
  \end{prooftree}
}
\newcommand{\iUpAxiomD}{
  \begin{prooftree}
    \AxiomC{\(A \otimes \dual{A}\)}
    \iUpL\UnaryInfC{\phantom{\(\{\}\)}}
  \end{prooftree}
}
\newcommand{\iUpD}{
  \begin{prooftree}
    \AxiomC{\(S\{(A \otimes \dual{A}) \parr B\}\)}
    \iUpL\UnaryInfC{\(S\{B\}\)}
  \end{prooftree}
}
\newcommand{\switchD}{
  \begin{prooftree}
    \AxiomC{\(S\{A \otimes (B \parr C)\}\)}
    \switchL\UnaryInfC{\(S\{(A \otimes B) \parr C\}\)}
  \end{prooftree}
}

\newcommand{\sigmaUpD}{
  \begin{prooftree}
   \AxiomC{\(S\{A\parr B\}\)}
   \sigmaUpL\UnaryInfC{\(S\{B\parr A\}\)}
  \end{prooftree}
}
\newcommand{\sigmaDownD}{
  \begin{prooftree}
   \AxiomC{\(S\{A\tensor B\}\)}
   \sigmaDownL\UnaryInfC{\(S\{B\tensor A\}\)}
  \end{prooftree}
}
\newcommand{\alphaUpD}{
  \begin{prooftree}
   \AxiomC{\(S\{A\parr (B\parr C)\}\)}
   \alphaUpL\UnaryInfC{\(S\{(A\parr B)\parr C\}\)}
  \end{prooftree}
}
\newcommand{\alphaDownD}{
  \begin{prooftree}
   \AxiomC{\(S\{(A\tensor B)\tensor C\}\)}
   \alphaDownL\UnaryInfC{\(S\{A\tensor (B\tensor C)\}\)}
  \end{prooftree}
}

\newcommand{\DIAxioms}[2][1]{
  {
    \begin{tikzpicture}[baseline={(0,0)}]
      \draw (-0.05,0) -- node[left]{#2}            (-0.05,#1);
      \draw ( 0.05,0) -- node[right]{\phantom{#2}} (0.05,#1);
      \draw (-0.2,#1) --                           (0.2,#1);
    \end{tikzpicture}
  }
}
\newcommand{\DIRules}[2][1]{
  {
    \begin{tikzpicture}[baseline={(0,0)}]
      \draw (-0.05,0) -- node[left]{#2}            (-0.05,#1);
      \draw ( 0.05,0) -- node[right]{\phantom{#2}} (0.05,#1);
    \end{tikzpicture}
  }
}
\newcommand{\arbitraryDIAxioms}[2][1]{\noLine \AxiomC{\DIAxioms[#1]{#2}} \noLine}
\newcommand{\arbitraryDIRules}[2][1]{\noLine \UnaryInfC{\DIRules[#1]{#2}} \noLine}
\newcommand{\SCRules}[2][1]{
  \begin{tikzpicture}[baseline={(0,0)}]
    \draw
         (-1*#1, {0.08*#1*sin(0) + #1})
      -- (0,0)
      -- (1*#1, {0.08*#1*sin(6 r) + #1});
    \draw[domain=-1:1] plot (\x*#1, {0.08*#1*sin(3*(\x+1) r) + #1});
    \node[draw=none] at (0, {(0.08*#1*sin(3 r) + #1)/2 + 0.05}) {#2};
  \end{tikzpicture}
}
\newcommand{\arbitrarySCAxioms}[2][1]{
  \noLine
  \AxiomC{\SCRules[#1]{#2}}
  \noLine
}
\newcommand{\arbitrarySCRules}[2][1]{
  \noLine
  \UnaryInfC{\SCRules[#1]{#2}}
  \noLine
}

\newcommand{\dual}[1]{\overline{#1}}

\newcommand{\ttile}{\vdash}

\newcommand{\coheq}[1]{{\overset{\lower.8ex\hbox{$\scriptstyle\frown$}}{\lower.2ex\hbox{$\scriptstyle\smile$}}}_{#1}}
\newcommand{\ncoheq}[1]{{\overset{\lower.9ex\hbox{$\scriptstyle\smile$}}{\lower.2ex\hbox{$\scriptstyle\frown$}}}_{#1}}

\newcommand{\deriv}[1]{\Pi_{#1}}

\newenvironment{bprooftree}
  {\leavevmode\hbox\bgroup}
  {\DisplayProof\egroup}

\newcommand{\permutation}{permutation}

\newcommand{\permutations}{permutations}

\newcommand{\proofax}[1]{#1\{\text{\axT{}}\}}
\newcommand{\proofid}[1]{#1\{\text{\idT{}}\}}
\newcommand{\proofcut}[1]{#1\{\text{\cutT{}}\}}
\newcommand{\proofswitch}[1]{#1\{\text{\switchT{}}\}}
\newcommand{\proofotimes}[1]{#1\{\text{\otimesT{}}\}}
\newcommand{\proofaiDown}[1]{#1\{\text{\aiDownT{}}\}}
\newcommand{\proofiUp}[1]{#1\{\text{\iUpT{}}\}}
\newcommand{\proofiDown}[1]{#1\{\text{\iDownT{}}\}}
\newcommand{\proofsize}[1]{|#1|}

\newcommand{\transcut}{\mathop{T^{\mathalpha{SC}}_{\mathalpha{SC}\setminus\proofcut{}}}}

\newcommand{\dontbreak}[1]{\vbox{#1}}

\newcommand{\mllm}[1]{\mathalpha\textbf{MLL}^-_{\mathalpha #1}}

\newenvironment{floatframe}[1]{\begin{mdframed}[style=FFrame,align=center,userdefinedwidth=#1]}{\end{mdframed}}

\mdfdefinestyle{FFrame}{%
    outerlinewidth=.4pt,
    roundcorner=10pt,
   innertopmargin=15pt,
   innerbottommargin=15pt}

\newcommand{\num}{\operatorname*{n}}

\newcommand{\ntp}{\num\nolimits^+_{\tensor}}
\newcommand{\ntm}{\num\nolimits^-_{\tensor}}
\newcommand{\npp}{\num\nolimits^+_{\parr}}
\newcommand{\npm}{\num\nolimits^-_{\parr}}
\newcommand{\nc}{\num\nolimits_{\mbox{\tiny{co}}}}
\newcommand{\np}{\num\nolimits_{p}}
\newcommand{\nt}{\num\nolimits_{t}}

\newcommand{\orth}[1]{{#1}^{\circ}}
\newcommand{\dorth}[1]{{#1}^{\circ\circ}}

\newcommand{\cohcat}{\mathord{\mathsf{Coh}}}
\newcommand{\concat}{\mathord{\mathsf{Con}}}

\newcommand{\prodor}{\mathbin{\,|\,}}

\newcommand{\lb}{\mathopen{\lBrack}}
\newcommand{\rb}{\mathclose{\rBrack}}

\newcommand{\intp}[1]{\lb #1\rb}

\newcommand{\transdisc}{\mathop{T^{\mathalpha{DI}}_{\mathalpha{SC}}}}
\newcommand{\transdiscalt}{\mathop{T'^{\mathalpha{DI}}_{\mathalpha{SC}}}}
\newcommand{\transscdi}{\mathop{T^{\mathalpha{SC}}_{\mathalpha{DI}}}}

\title{Modelling Multiplicative Linear Logic via Deep Inference}
\author{Tomer Galor\thanks{\href{mailto:publications@work.galor.email}{publications@work.galor.email}}
  \qquad\qquad Andrea Schalk\thanks{	
        \href{mailto:andrea.schalk@manchester.ac.uk}{andrea.schalk@manchester.ac.uk}}\\
	School of Computer Science\\
		University of Manchester\\
		Manchester, UK\\
              }

\makeatletter

\begin{document}
\maketitle

\begin{abstract} 
Multiplicative linear logic is a very well studied formal system, and
most such studies are concerned with the one-sided sequent
calculus. In this paper we look in detail at
existing translations between a deep inference system and the standard sequent calculus one,
provide a simplified translation, and provide a formal
proof that a standard approach to modelling is indeed invariant to all
these translations. En route we establish a necessary condition for
provable sequents related to the number of pars and tensors in a formula
that seems to be missing from the literature.
\end{abstract}
\section{Introduction}\label{intro}

Linear logic is a constructive logic created by Girard \cite{girard}
that has gained increased interest since its inception. It has shown
potential applications from theorem provers and optimising and
compiling concurrent programming languages as a result of the way in
which linear logic is able to encode state transitions \cite{girard,
  taste_of_ll}, to providing an embedding for classical logic and a
categorical interpretation for proof equivalences. Our main focus,
being a body of work that arose from a master's dissertation
\cite{tomers_msc}, is aiming to shed a different light on the area by
using \emph{deep inference}\/ systems \cite{deepinf} as an alternative
mode of description. One of our motivations was to explore to what
extent the different system affects the formulation of the properties
required of a model, and we use Girard's \emph{coherence
  spaces}\/~\cite{girard,lazy} to make the comparison. 

Deep inference systems for multiplicative linear logic share a close
relationship to the work done by Retor{\'e} on pomset logic \cite{retore2021}
which influenced the work on deep inference \cite{guglielmi_deep_inference,retore2021}.
Indeed, Retor{\'e} gave an account of coherence semantics applied to
the deep pomset system \cite{retore2021} and showed equivalence to the
SBV system given by Guglielmi and Stra{\ss}burger \cite{mell_calculus_structures}.
Our work is quite different from this body of work however because our
primary focus is to explore how coherence and categorical semantics hold
invariance by transformation procedures between the sequent calculus and
deep inference systems.

The Equivalence of Proofs is a subcategory of Proof Theory that asks
the question of how one is able to determine if two derivations are
equivalent, or semantically ``equal''. Multiplicative linear logic, $\mllm{}$,
(unlike more powerful systems) has a well-established answer to this
question: At least for the sequent calculus formulation we have a
notion of \emph{proof net}\/~\cite{proof_nets,correctness_criterion},
which identifies two derivations provided they only differ trivially
(at least in some sense) regarding the order in which the inference
rules have been applied.

A proof net boils down a derivation to the following ingredients:
\begin{itemize}
\item The syntactic structure of the formula or sequent at issue; and
  \item a bijection between positive and negative occurrences of
    atoms which preserves names which captures which pairs were
    created by instances of the only axiom of the system, known as
    the \emph{axiom links}.
\end{itemize}

Based on these ideas Stra{\ss}burger~\cite{proof_nets} argues that
one may readily transfer this notion to deep inference systems, and
that since the translation process does not make any changes to the
pairs of atoms created together, the notion of proof net stays the
same. However, the translation process is not straightforward,
introduces a large number of cuts and increases the size of an
interpretation significantly, and for that reason we wanted to
develop a formal argument that a standard way of modelling
derivations in coherence spaces is indeed not affected by the
translation process. We additionally provide an alternative to
Stra{\ss}burger's translation which keeps the size of derivations
smaller by making more assiduous choices regarding when to employ the
cut rule.

In our studies of $\mllm{}$ in a deep inference formulation we noted
that there is what seems to be an unpublished necessary condition for
a formula to be provable, and while one may readily establish the
result for the sequent calculus version, it is rather more obvious in
the former. We do not expect the result to be a surprise to those
familiar with the proof theory of $\mllm{}$.

There is a significant body of literature (for example 
~\cite{seelycat_mod,gangoffour,haghverdi,mellies}) concerned with what is
required for a category to provide a model of linear logic. We should
note here that our aim was not to use \emph{dinatural
  transformations}\/~\cite{blutedinat,blutescottdinat,tanthesis,schalksteele}
to model derivations, but instead to pick morphisms dependent on a
valuation for atoms. To obtain a case study it is sufficient to look at one
such model, for which we chose Girard's coherence
space~\cite{girard}---the reader interested in generalisations is
invited to look at our description of the deep inference model which
readily suggests a generalisation to categories with suitable
infrastructure.

We find it more convenient to work with an alternative axiomatization
of coherence spaces, and we introduce the isomorphic category of
\emph{concordance space}\/~below.

We were surprised to be unable to track down a published description
of how to model a derivation using coherence spaces;
Girard's description ~\cite{girard} merely defines the interpretation
of various connectives on objects. The only source we were able to
find is~\cite{lazy}, which has some shortcomings discussed below. 

Hughes observes~\cite{hughesdeepinf} that deep inference systems are
much closer to a categorical setting, and indeed one may think of
such systems as describing various free categories. Indeed, modelling
a sequent calculus system requires significant overhead to deal with
the infrastructure required, and in our case study we find it rather
easier to describe that model. Most studies of sequent calculus
systems of linear logic remove the cut rule, and since there is a
cut elimination result~\cite{pfenningcut_elim} this can be justified.
However, in order to translate from a deep inference system to a
sequent calculus one the cut rule has to be present. From a
categorical point of view the cut rule is ultimately concerned with
composition, but due to the overheads required to deal with the
infrastructure this is not very clear when looking at the categorical
interpretation of a sequent calculus derivation. In contrast the way
deep inference derivations are jointed together \emph{provides}\/ a
composition (up to some mild quotienting), and so it should come as
no surprise that it gives us a syntactic system that is much closer to
the categorical model.

We hope that this paper is helpful to newcomers to linear logic by
aiming to be largely self-contained while also encouraging others to
explore this space further. For this reason we have kept the use of
abstract category theory to a minimum, and have decided instead to
provide a fairly concrete description of our model. One may think of
this as a study into interpretations of proofs without requiring the
need to talk about proof nets. The literature on deep inference
systems for linear logic is scarce, and we hope that this paper
encourages others to further study these.

We aim to give a summary of our starting point in \cref{back}
which includes a description of the systems under consideration,
before establishing some results regarding their properties in
\cref{props}. We then turn to the question of modelling
derivations and prove that this model is preserved by translation
between the two systems.

\section{Background}\label{back}

In the interest of keeping this paper largely self-contained and to aid readability we
give a brief summary of various notions and definitions we use.

\subsection{Derivation systems under consideration}\label{systems}

Linear logic~\cite{girard} is a very well studied formal system, and
much attention has been given to its multiplicative fragment without
units. The standard presentation of linear logic is as a single-sided
sequent calculus from which the cut rule has been eliminated. This
simplifies some considerations but it also obscures other matters. In
some sense the sequent calculus formulation is merely an example of
Gentzen style systems \cite{gentzen1} one might describe here, and it
is chosen here because it is very familiar in the context of linear
logic. The following system is standard for Multiplicative Linear
Logic without units, known as $\mllm{SC}$ hereafter.

\begin{figure}[!htbp] 
\begin{floatframe}{.8\textwidth}
  \centering
  \newcommand{\SUBFIGUREWIDTH}{0.45\linewidth}
  \newcommand{\VERTICALSPACE}{\vspace{2em}}
  \begin{subfigure}[b]{\SUBFIGUREWIDTH}
    \idD{}
  \end{subfigure}
  \VERTICALSPACE
  \begin{subfigure}[b]{\SUBFIGUREWIDTH}
    \exchD{}
  \end{subfigure}
  \begin{subfigure}[b]{\SUBFIGUREWIDTH}
    \parD{}
  \end{subfigure}
  \VERTICALSPACE
  \begin{subfigure}[b]{\SUBFIGUREWIDTH}
    \otimesD{}
  \end{subfigure}
  \begin{subfigure}[b]{\SUBFIGUREWIDTH}
    \cutD{}
  \end{subfigure}
\end{floatframe}
\end{figure}

This system enjoys a cut elimination property and so the system most
widely used in the literature is the above without the cut rule.
Since we wish to compare this inference system with another which
follows quite a different philosophy, we will make a few comments on
the choices made here which are usually taken for granted.

While there are introduction rules for the two binary connectives,
$\tensor$ and $\llpar$, there are no such rules for negation, and the
only way of obtaining a negated formula is the use of the axiom rule.
In order to treat negation as a proper connective in its own right a
two-sided system is needed, and then applying negation moves the
formula in question onto the other side of the turnstile.

As a consequence we view negation as a derived rule. We have positive
and negative occurrences for each atom, and we use the following
\emph{DeMorgan rules}\/ to define negation recursively for more complex
formulae.

\[\dual{A\tensor
    B}=\dual{A}\llpar\dual{B}\qquad\text{and}\qquad\dual{A\llpar
    B}=\dual{A}\tensor\dual{B}.\]

Note that in this system a formula is equal to its double negation.
This is the standard treatment for such systems. One could, in
principle, decide that one should allow negation to appear anywhere
within a well-formed formula.  In that situation one needs to have a
notion of \emph{DeMorgan equivalence}, which is generated by the
above DeMorgan rules. Then one either needs to allow the cut rule to
be applied to pairs of formulae which are negations of each other up
to DeMorgan equivalence, or one needs a `DeMorgan derivation rule'
which allows the replacement of a formula in the sequent by a
DeMorgan equivalent formula. In either of these systems a formula is
equivalent to its double negation.

Our aim is to contrast $\mllm{SC}$ with a very different style of
formal derivation setting, known as \emph{deep inference}\/
\cite{deepinf}, or \emph{the calculus of structures}.
In the former, formulae are derived step-by-step, and
connectives have to be created in the order in which they appear in
the parse tree. In the latter, one may apply an instance of a rule
\emph{at arbitrary depth inside a formula}. Whereas, viewed from
bottom to top, derivations in a sequent calculus spread out in a
tree-like manner, deep inference derivations are linear. There is a
natural connection between deep inference systems and categorical
logic as observed by Hughes~\cite{hughesdeepinf}. Indeed, one may
think of deep inference as providing the syntax required for
describing a free category with some structure (corresponding to the
logical connectives), but describing the identifications required is
non-trivial once one has two connectives that interact with each
other. In this publication we concentrate on looking at more concrete
interpretations. Most deep inference systems studied in the
literature are concerned with classical logic, and not much has been
written about such systems that describe (fragments of) linear logic.
In \cite{local_system_for_ll} the author gives a system for full
linear logic, while in \cite{proof_nets} we find a system
for~$\mllm{}$, and that is the system we use.

We assume that readers will be familiar with derivations in a sequent
calculus, but not necessarily with deep inference systems. A number
of issues are worth commenting on:
\begin{itemize}
\item It is customary in the deep inference community to consider
  formulae \emph{equal}\/ if they are equal up to associativity and
  symmetry. This is a convention we do not adopt since it does not
  match well with the usual treatment of formulae in linear logic,
  and indeed the system given in \cite{proof_nets} does not adopt
  this convention either.
\item Once again we adopt a notion of \textbf{well-formedness for formulae}
  which allows negation on atoms only. Again this is the approach
  chosen in \cite{proof_nets}.
\item In order to indicate how we may change a formula `on the
  inside' we need some kind of convention, and different conventions
  are adopted in the literature. Here we use $S\{B\}$ to indicate a
  formula of which one particular instance of the formula $B$ has
  been identified. The inference rules allow us to produce a formula
  on the next step where that instance of $B$ has been replaced by
  the formula that appears in the braces in the conclusion.
\end{itemize}

We use $\mllm{DI}$ to refer to the following system.\footnote{Note
  that in the presence of a `deep' symmetry it is sufficient to have one
  direction only for associativity rules.}

\begin{figure}[!htbp] 
\begin{floatframe}{.8\textwidth}
  \centering
  \newcommand{\SUBFIGUREWIDTH}{0.48\linewidth}
  \newcommand{\VERTICALSPACE}{\vspace{1em}}
  \begin{subfigure}[b]{\SUBFIGUREWIDTH}
    \iDownAxiomD{}
  \end{subfigure}
  \VERTICALSPACE
  \begin{subfigure}[b]{\SUBFIGUREWIDTH}
    \iDownD{}
  \end{subfigure}
  \begin{subfigure}[b]{\SUBFIGUREWIDTH}
    \iUpAxiomD{}
  \end{subfigure}
  \VERTICALSPACE
  \begin{subfigure}[b]{\SUBFIGUREWIDTH}
    \iUpD{}
  \end{subfigure}
  \begin{subfigure}[b]{\SUBFIGUREWIDTH}
    \sigmaUpD{}
  \end{subfigure}
  \VERTICALSPACE
  \begin{subfigure}[b]{\SUBFIGUREWIDTH}
    \sigmaDownD{}
  \end{subfigure}
  \begin{subfigure}[b]{\SUBFIGUREWIDTH}
    \alphaUpD{}
  \end{subfigure}
  \VERTICALSPACE
  \begin{subfigure}[b]{\SUBFIGUREWIDTH}
    \alphaDownD{}
  \end{subfigure}
  \VERTICALSPACE
  \begin{subfigure}[b]{\SUBFIGUREWIDTH}
    \switchD{}
  \end{subfigure}
\end{floatframe}
\end{figure}

Note that overhead is created by no longer allowing derivations to
have a tree-like structure. It is easy to see that if in $\mllm{SC}$
we may derive a formula of, say, the form $A\tensor B$ then we may
also derive the formula $B\tensor A$ by swapping the two branches
that establish $A$ and $B$ respectively. In the deep inference
system, on the other hand, this requires the use of the rule
rule~\sigmaDownT{}.
It is trivial to see that the switch rule \switchT{} has a corresponding symmetric rule
from ${S\{(A \parr B) \otimes C\}}$ to ${S\{A \parr (B \otimes C)\}}$
that is derivable from the \sigmaUpT{}, \sigmaDownT{} and \switchT{} rules.
We use \sigmaSwitchT{} as notational shorthand for this.

We obtain a result similar to that for the sequent calculus
formulation in that it is possible to replace the two introduction
rules \iDownT{} by the corresponding versions for atoms below.

\begin{figure}[h]
  \newcommand{\SUBFIGUREWIDTH}{0.48\linewidth}
  \begin{floatframe}{.8\textwidth}
    \centering
    \begin{subfigure}[b]{\SUBFIGUREWIDTH}
    \aiDownAxiomD{}
    \end{subfigure}
    \begin{subfigure}[b]{\SUBFIGUREWIDTH}
    \aiDownD{}
    \end{subfigure}
  \end{floatframe}
\end{figure}

The following result is a corollary of Proposition~2.10
in~\cite{local_system_for_ll} when applied to our system for $\mllm{DI}$.

\begin{proposition}
  \label{diatomsonly}
  In \(\mllm{DI}\), a formula is derivable with the following axioms:
  \\
  \begin{minipage}[b]{0.49\textwidth}
    \iDownAxiomD{}
  \end{minipage}
  \begin{minipage}[b]{0.49\textwidth}
    \iDownD{}
  \end{minipage}
  \\
  if and only if it is derivable with the following axioms:
  \\
  \begin{minipage}[b]{0.49\textwidth}
    \aiDownAxiomD{}
  \end{minipage}
  \begin{minipage}[b]{0.49\textwidth}
    \aiDownD{}
  \end{minipage}
  \\
  \begin{proof}
    The proof in the \emph{if}\/ direction is done by trivially only using \(A\) as an atom in
      \iDownT{}.
    The \emph{only if}\/ part is more interesting
      and is an inductive proof on the size of the formula \(A\).
    \begin{enumerate}
      \item
        If \(A\) is an atom \(\alpha\)
        then the derivation is simply the \aiDownT{} rule.
      \item
        If \(A\) is the negation of some formula \(F\) such that \(A = \dual{F}\),
        where we have a proof of \(\dual{F} \parr F\),
        then the formula to derive is: \(F \parr \dual{F}.\)
        Since negation is an involution,
          the sequent can be derived after a simple use of the \sigmaUpT{} rule.
      \item
        If \(A\) is a formula of the form \(B \otimes C\),
          where we have a proof of the formula \(\dual{B} \parr B\) called
          \(\deriv{B}\) and a proof of \(\dual{C} \parr C\)
          of the form:
          \begin{prooftree}
            \AxiomC{}
            \aiDownL\UnaryInfC{\(\dual{c_0} \parr c_0\)}
            \arbitraryDIRules{\(\deriv{C}\)}
            \UnaryInfC{\(\dual{C} \parr C\)}
          \end{prooftree}
          then the formula to derive is \(\dual{(B \otimes C)} \parr (B \otimes
          C) \equiv (\dual{B} \parr \dual{C}) \parr (B \otimes C)\)
            with derivation:
          \begin{prooftree}
            \arbitraryDIAxioms{\(\deriv{B}\)}
            \UnaryInfC{\(\dual{B} \parr B\)}
            \alwaysSingleLine
            \aiDownL\UnaryInfC{\(\dual{B} \parr (B \otimes (\dual{c_0} \parr c_0))\)}
            \arbitraryDIRules{\(\dual{B} \parr (B \otimes \deriv{C})\)}
            \UnaryInfC{\(\dual{B} \parr (B \otimes (\dual{C} \parr C))\)}
            \sigmaUpL\UnaryInfC{\(\dual{B} \parr (B \otimes (C \parr \dual{C}))\)}
            \switchL\UnaryInfC{\(\dual{B} \parr ((B \otimes C) \parr \dual{C})\)}
            \sigmaUpL\UnaryInfC{\(\dual{B} \parr (\dual{C} \parr (B \otimes C))\)}
            \alphaUpL\UnaryInfC{\((\dual{B} \parr \dual{C}) \parr (B \otimes C)\)}
          \end{prooftree}
        \item
          If \(A\) is a formula of the form \(B \parr C\),
          where we have a proof of the formula \(\dual{B} \parr B\) called
          \(\deriv{B}\) and a proof of \(\dual{C} \parr C\)
          of the form:
          \begin{prooftree}
            \AxiomC{}
            \aiDownL\UnaryInfC{\(\dual{c_0} \parr c_0\)}
            \arbitraryDIRules{\(\deriv{C}\)}
            \UnaryInfC{\(\dual{C} \parr C\)}
          \end{prooftree}
          then the formula to derive is \(\dual{(B \parr C)} \parr (B \parr C)
          \equiv (\dual{B} \otimes \dual{C}) \parr (B \parr C)\)
            with derivation:
          \begin{prooftree}
            \arbitraryDIAxioms{\(\deriv{B}\)}
            \UnaryInfC{\(\dual{B} \parr B\)}
            \aiDownL\UnaryInfC{\((\dual{B} \otimes (\dual{c_0} \parr c_0)) \parr B\)}
            \arbitraryDIRules{\((\dual{B} \otimes \deriv{C}) \parr B\)}
            \UnaryInfC{\((\dual{B} \otimes (\dual{C} \parr C)) \parr B\)}
            \alwaysSingleLine
            \switchL\UnaryInfC{\(((\dual{B} \otimes \dual{C}) \parr C) \parr B\)}
            \sigmaUpL\UnaryInfC{\(B \parr ((\dual{B} \otimes \dual{C}) \parr C)\)}
            \sigmaUpL\UnaryInfC{\(B \parr (C \parr (\dual{B} \otimes \dual{C}))\)}
            \alphaUpL\UnaryInfC{\((B \parr C) \parr (\dual{B} \otimes \dual{C})\)}
            \sigmaUpL\UnaryInfC{\((\dual{B} \otimes \dual{C}) \parr (B \parr C)\)}
          \end{prooftree}
    \end{enumerate}
  \end{proof}
\end{proposition}

It is well known (for example, see~\cite{proof_nets}) that $\mllm{SC}$
has the same derivable sequents as the related system whose axiom
rule only allows the creation of a pair consisting of an atom and its
negation. We explicitly give a version of this proof as we will refer
back to it below.

\begin{proposition}
  \label{scatomsonly}
  In \(\mllm{SC}\), a formula is derivable with the following axiom:
  \idD{}
  if and only if it is derivable with the following axiom:
  \axD{}
\begin{proof}
  The proof for the \emph{if}\/ direction is done by trivially only using \(A\) as an atom in
    \idT{}.
  The \emph{only if}\/ part is more interesting
    and is an inductive proof on the size of the formula \(A\).
  \begin{enumerate}
    \item
      If \(A\) is an atom \(\alpha\),
      then the derivation is simply the \axT{} rule.
    \item
      If \(A\) is the negation of some formula \(F\) such that \(A = \dual{F}\),
      where we have a proof of \(\ttile \dual{F}, F\),
      then the sequent to derive is \(\ttile \dual{\dual{F}}, \dual{F}\).
      Since \(\dual{\dual{F}} = F\),
        the sequent can be derived after a simple use of the \exchT{} rule.
    \item
      If \(A\) is a formula of the form \(B \otimes C\),
        where we have proofs of \(\ttile \dual{B}, B\) and \(\ttile \dual{C}, C\)
        called \(\deriv{B}\) and \(\deriv{C}\) respectively.
      Then the sequent to derive is \(\ttile \dual{(B \otimes C)}, B \otimes C\)
        with derivation:
      \begin{prooftree}
        \arbitrarySCAxioms{\(\deriv{B}\)}
        \UnaryInfC{\(\ttile \dual{B}, B\)}
        \arbitrarySCAxioms{\(\deriv{C}\)}
        \noLine\UnaryInfC{\(\ttile \dual{C}, C\)}
        \exchL\UnaryInfC{\(\ttile C, \dual{C}\)}
        \otimesL\BinaryInfC{\(\ttile \dual{B}, B \otimes C, \dual{C}\)}
        \exchL\UnaryInfC{\(\ttile \dual{B}, \dual{C}, B \otimes C\)}
        \parL\UnaryInfC{\(\ttile \dual{B} \parr \dual{C}, B \otimes C\)}
      \end{prooftree}
    \item
      If \(A\) is a formula of the form \(B \parr C\),
        where we have proofs of \(\ttile \dual{B}, B\) and \(\ttile \dual{C}, C\)
        called \(\deriv{B}\) and \(\deriv{C}\) respectively.
      Then the sequent to derive is \(\ttile \dual{(B \parr C)}, B \parr C\)
        with derivation:
      \begin{prooftree}
        \arbitrarySCAxioms{\(\deriv{B}\)}
        \UnaryInfC{\(\ttile \dual{B}, B\)}
        \exchL\UnaryInfC{\(\ttile B, \dual{B}\)}
        \arbitrarySCAxioms{\(\deriv{C}\)}
        \UnaryInfC{\(\ttile \dual{C}, C\)}
        \otimesL\BinaryInfC{\(\ttile B, \dual{B} \otimes \dual{C}, C\)}
        \exchL\UnaryInfC{\(\ttile \dual{B} \otimes \dual{C}, B, C\)}
        \parL\UnaryInfC{\(\ttile \dual{B} \otimes \dual{C}, B \parr C\)}
      \end{prooftree}
  \end{enumerate}
\end{proof}
\end{proposition}

We compare the two systems in \cref{connect}.

\subsection{Models of linear logic}\label{llmodels}

Models of linear logic are well studied, and even in the original
publications,  such as \cite{girard}, two models were
proposed. One based on \emph{phase spaces}\/ where the interpretation
takes place inside the powerset of some set, with the usual partial
order, and one based on \emph{coherence spaces}, sets equipped with a
binary `coherence relation'. The latter is more interesting as a
categorical model since it is non-degenerate as a category.

In the background research for the dissertation that the present work
is founded on we were surprised to discover that it
is quite difficult, however, to find an account of how exactly
derivations are interpreted in this latter model. We were unable to find
an account in a journal, and the only record that provides
significant detail that we were able to track down is the
hand-written~\cite{lazy}. This account does cover full classical
logic, and spends significant time on stable and linear functions,
which is not relevant to our investigation. The interpretation of
derivations as cliques is given informally which is sufficient to
look at results of interest to that account, but not for the purposes
of comparing models of different derivation systems. For this reason
we give a detailed account of this model in
\cref{sequentmodel}.

\subsubsection*{Coherence spaces}

We give a brief overview of coherence spaces as they appear
in~\cite{girard}, indicating how to model the connectives, but we
add here the generally accepted notion of morphism to obtain a
category. 

A \textbf{coherence space} is given by a set together with a
symmetric reflexive binary relation, written as $\coheq{}$, the
\emph{coherence relation}. 

The dual $\dual{(R,\coheq{})}$ of a coherence space $(R,\coheq{})$ is
given by the set $R$ together with the reflexive closure of the
complement of $\coheq{}$, and the result is typically written as
$\ncoheq{}$. We use `incoherent' (with respect to the original space)
when referring to this relation, but note that incoherent might mean
equal.

A \textbf{morphism of coherence spaces} from some $(R,\coheq{})$ to
some $(S,\coheq{})$ is given by a relation $F$ from $R$ to $S$ such
that if preserves the coherence relation going forward, and going in
the opposite direction it preserves incoherence:
\begin{itemize}
\item If $r$ and $r'$ are coherent and $r\mathrel{F} s$ as well as
  $r'\mathrel{F}s'$ then $s$ and $s'$ must be coherent; and
\item If $s$ and $s'$ are incoherent, and if $r\mathrel{F}s$ as well
  as $r'\mathrel{F}s'$ then $r$ and $r'$ must be incoherent.
\end{itemize}

The various connectives of multiplicative linear logic have
ready-made interpretations, and we reuse the corresponding symbols.
The tensor product of two coherence spaces $(R,\coheq{})$ and
$(S,\coheq{})$ has the Cartesian product $R\times S$ as its
underlying set and two pairs are coherent if they are coherent
component-wise. The par of two coherence spaces is given via the
DeMorgan rules from above: one forms the dual of the two spaces,
forms their tensor product, and takes the dual of the result. As a
consequence coherence spaces satisfy the DeMorgan rules from above as
equalities, and the double dual of a coherence space is the space
itself.

The category $\cohcat$ of coherence space is $*$-autonomous, with
unit $I$ consisting of the set $\{*\}$ with the only possible
coherence relation serving as a unit for tensor.

A \textbf{clique} in a coherence space is given by a subset of the
underlying set such that any two elements of this subset are
coherent. An \textbf{anticlique} is a clique for the dual space.
This corresponds to a clique in a graph and indeed one can construct a reflexive
graph with elements of a set as vertices and the coherence relation as a set of
edges.
In \cref{sequentmodel} we describe how to use cliques to model
derivations in the sequent calculus system.
Note that we consider the empty set
both a clique and an anticlique.

In \cref{concord} we provide an alternative description of coherence
spaces.

\subsubsection*{Identity of proofs}

It is generally accepted that two derivations in $\mllm{SC}$ starting
with no assumptions should be considered equivalent if they derive
the same sequent, and if they create the `same' pairs of atoms via the
axioms, where we have to take care to refer to instances of atoms in
the final sequent. This relation may be characterised syntactically
in the sense that one derivation may be transformed into any
equivalent one by applying a sequence of \emph{commuting conversions}\/
which allow certain instances of derivation rules to be moved past
each other.

We can see some of these ideas at work in the proofs below that show our
interpretations of sequent calculus and deep inference derivations are the same
under various transformations.

\subsection{Notation}

In what follows, most of the notation we use should be
self-explanatory or following standard conventions. For example, we
employ lower case letters for atoms, upper case for arbitrary
formulae and capital Greek letters for sequents. Note that we use a
line over a formula to indicate its negation, and that our logical
systems do not treat negation as a logical connective in its own
right as explained above.

We use the name of an object to also denote its identity morphism in
some category.

\section[Properties of MLL]{Properties of \texorpdfstring{$\mllm{}$}
  {MLL}}\label{props}

We start our exploration of the two systems introduced above, by
first of all establishing a necessary property for derivable sequents
or formulae. This seems to have been overlooked in the literature,
although we do not expect it to be a surprise to anybody familiar with
$\mllm{}$. Admittedly, it is a property much easier to spot in the
deep inference system since the commas between sequents act as a
distraction.

\subsection{A necessary condition for derivable formulae/sequents}
\normalsize

When manipulating formulae in either system one frequently writes
expressions which have to be rewritten according to the DeMorgan
rules to become well-formed formulae. The result we wish to establish
is a necessary condition for derivable formulae or sequents. It works
even for expressions which are not themselves well-formed formulae
and so we establish it at this level of generality.

For the purposes of this section we consider all formulae that may be
created using $\tensor$, $\llpar$ and negation, and we consider two
such formula equivalent if one can be rewritten into the other using
the DeMorgan laws.

A formula whose only negated parts are atoms (that is a well-formed
formula in our system) is considered the \emph{normal form}\/ of such
an expression.

\begin{definition}\label{posnegocc}
  We say that a connective in an $\mllm{}$ formula \textbf{occurs
    positively}\/ if the number of negation overlines above it is
  even; else we say the connective \textbf{occurs negatively}.
\end{definition}

We have reason to count the positive and negative occurrences of the
two connectives in a formula. Given a formula $A$ we use
$\ntp A$ for the number of positive occurrences of $\tensor$
in $A$, and we extend that notation in the expected way.

\begin{proposition}\label{equivnos}
  If two formulae $A$ and $B$ are equivalent then we have the
  following:
  \begin{align*}
\ntp A +\npm A &= \ntp B + \npm B\\
\ntm A  +\npp A &= \ntm B +\npp B.
  \end{align*}
\end{proposition}
\begin{proof}
  We can see that applying the DeMorgan rule on two formulae
  connected by a tensor means we decrease $\ntm$ by one and increase
  $\npp$ by one while the other two numbers stay the same. Similar
  rules apply to using this rule on the negation of two formulae connected by a
  par, or using the equivalent rules with pars and tensors swapped.
\end{proof}

\begin{proposition}\label{dualnos}
  Given a formula $A$ and its negation $\dual{A}$ we have that
  \begin{align*}
    \ntp A&=\ntm \dual{A} &\npp A&=\npm \dual{A}\\
    \ntm A&=\ntp \dual{A} &\npm A&=\npp \dual{A},
  \end{align*}
  and so 
  \begin{align*}
    \ntp A + \npm A &= \ntm\dual{A} +\npp\dual{A}\\
    \ntm A + \npp A &=\ntp\dual{A} +\npm\dual{A}.
  \end{align*}
\end{proposition}
\begin{proof}
  Forming the negation of $A$ adds one line over every connective in
  $A$, so positive ones become negative and vice versa.
\end{proof}

Surprisingly, we were unable to find \cref{derivnos,derivnosseq} in the
literature. %

\begin{theorem}\label{derivnos}
  If $A$ is a formula derivable  in $\mllm{DI}$ then
\[(\ntm A+ \npp A) - (\ntp A+\npm A)  = 1.\]
\end{theorem}
\begin{proof}
  The proof proceeds by induction. Clearly the result holds for a
  formula of the form $\dual{A}\llpar A$. It is trivially preserved
  by the symmetry, associativity and switch rules and also by
  DeMorgan equivalence, see \cref{equivnos}. That leaves the two
  non-axiomatic rules $i$. Considering \iDownT{}, assume that
  $S\{B\}$ satisfies the criterion. Then in
  $S\{B\tensor (\dual{A}\llpar A)\}$ the various counts of
  occurrences of connectives in $A$ balance out those in $\dual{A}$.
  There are two new connectives, $\tensor$ and $\llpar$, and they
  both occur either positively or negatively, so when counting we add
  one each to the number of positive $\tensor$s and positive
  $\llpar$s, or do the same for the number of negative $\tensor$s and
  $\llpar$s. This maintains the property. For \iUpT{} similar
  considerations apply, only here the counts are variously reduced.
\end{proof}

Note in particular that when we have a formula $A$ in normal form
then 
\[\ntm A=0\qquad\text{and}\qquad \npm A=0,\]
so we obtain the following corollary.

\begin{corollary}\label{derivnosnf}
  If a formula $A$ derivable in $\mllm{DI}$ is in normal form then
  its number of occurrence of $\llpar$ is one greater than the number
  of occurrences of~$\tensor$.
\end{corollary}

We extend our notion of number of negative or positive occurrences of
the connectives to sequents in the obvious way by writing, for
example, $\ntp\Gamma$ for the number of positive
occurrences of $\tensor$ in $\Gamma$, that is, if 
\[\Gamma\text{ is }  A_1,A_2,\ldots,A_n\]
then
\[\ntp\Gamma= \sum_{i=1}^n\ntp A_i.\]

We also need to count the number of commas that occur in $\Gamma$,
using~$\nc\Gamma$. Alternatively
we could use the rule~\parT{}  to 
turn $\Gamma$ into a formula by replacing all the commas
by~$\llpar$s.

This allows us to extend our result for the number of occurrences of
the connectives to the sequent calculus formulation of~$\mllm{}$.

\begin{theorem}\label{derivnosseq}
If $\Gamma$ is a sequent derivable in $\mllm{SC}$ then 
\[(\ntm \Gamma +\npp\Gamma +\nc\Gamma) - (\ntp \Gamma +\npm\Gamma)= 1.\]
\end{theorem}
\begin{proof}
  We could use the translation process from \cref{sctodi} but we
  prefer instead to reason about the sequent calculus
  system~$\mllm{SC}$. This is another proof by induction. Clearly
  the statement holds for the only axiom. The property is trivially
  maintained by the exchange and the $\llpar$ rules. Considering the
  $\tensor$ rules we can see that
  \begin{align*}
\MoveEqLeft[2]{ \ntm(\Gamma, A\tensor B,\Delta) + \npp(\Gamma, A\tensor B,\Delta)
    +\nc(\Gamma, A\tensor B,\Delta)}\\
 &= \ntm(\Gamma,A) +\ntm(B,\Delta) + \npp(\Gamma,A) + \npp(B,\Delta)
                                      +\nc(\Gamma,A) +\nc(B,\Delta)\\
 &=(\ntm(\Gamma,A)+\npp(\Gamma,A) +\nc(\Gamma,A)) + (\ntm(B,\Delta)
   +\npp(B,\Delta) +\nc(B,\Delta))\\
 &=(\ntp(\Gamma,A) +\npm(\Gamma,A) +1) + (\ntp(B,\Delta)
   +\npm(B,\Delta) +1)\\
&= (\ntp(\Gamma,A) +\ntp(B,\Delta) +1) + \npm(\Gamma,A) +
  \npm(B,\Delta) +1\\
&= \ntp(\Gamma, A\tensor B,\Delta) + \npm(\Gamma, A\tensor B,\Delta) +1.
 \end{align*}

It remains to look at the cut rule. We use the following derived
counters:
\begin{align*}
\nt&=\ntp+\npm&&\np=\ntm+\npp+\nc.
\end{align*}
We need to show that
\[\np(\Gamma,\Delta) -\nt(\Gamma,\Delta) =1\]
knowing that
\[\np(\Gamma,A) - \nt(\Gamma,A)=1\]
and
\[\np(\dual{A},\Delta) - \nt(\dual{A},\Delta) =1,\]
as well as, by \cref{dualnos},
\[\np A=\nt\dual{A}\qquad\text{and}\qquad\nt A=\nt\dual{A}.\]
\begin{align*}
\MoveEqLeft[3]{\np(\Gamma,\Delta) -\nt(\Gamma,\Delta)}\\
  &= (\np\Gamma + \np\Delta +1) -(\nt\Gamma + \nt\Delta)\\
  &= (\np\Gamma + \np A + \np\dual{A} + \np\Delta +1) - (\nt\Gamma +\nt
    A  +\nt\dual{A}+\nt\Delta)\\
  &= (\np(\Gamma,A) -1 +\np(\dual{A},\Delta) -1 +1) - (\nt(\Gamma,A)
    +\nt(\dual{A},\Gamma)) \\
  &= (\np(\Gamma,A) -\nt(\Gamma,A))
    +(\np(\dual{A},\Delta)-\nt(\dual{A},\Delta)) -1\\
  &= 1+1-1\\
  &= 1.
\end{align*}
This completes the proof.
\end{proof}

\begin{corollary}\label{derivnosseqnf}
  If $\Gamma$ is a sequent derivable in $\mllm{SC}$ consisting
  entirely of formulae in normal form then the number of $\llpar$s
  added to the number of commas is one greater than the number of
  $\tensor$s occurring in~$\Gamma$.
\end{corollary}

\section{Translations between the two systems}\label{connect}

Arguably, we owe a justification for the claim that the two systems
introduced in \cref{systems} implement `the same logic'. We do this
by providing processes of translations of derivations going in both
directions. The reader might want to know from the start that the two
translations do not undo each other's effect, no matter whether we
start in the deep inference or in the sequent calculus system.
Indeed, if we are not careful about how we define these translations
then the size of the derivation increases significantly each time we
apply one of them.

In \cite{proof_nets} Stra{\ss}burger gives a sketch of the proof
that the two systems $\mllm{SC}$ and $\mllm{DI}$ allow us to derive
essentially the same statements. We provide further details here,
as we require them below. We then look at the effects of these
translations, and suggest an alternative to one of them.

We begin by translating from the deep inference system to the sequent
calculus one.%

In what follows we use an inference rule labelled with
\involT\ to indicate that we are replacing some
formula with one equal as per the definition of negation.

We provide a translation $\transdisc$ for derivations from the deep
inference to the sequent calculus system.

\begin{proposition}\label{ditosc}
  If we have the derivation $\deriv{}$ of a formula $A$ in $\mllm{DI}$
  then there is a derivation $\transdisc\deriv{}$ of ${\ttile A}$ in
$\mllm{SC}$.
\end{proposition}
\begin{proof}
We provide a recursive procedure to define the translation.
  
We can describe any rule \(\rho\frac{S\{P\}}{S\{Q\}}\)
  called \(\rho\) as a transformation from some subformula \(P\) in \(S\)
  to \(Q\) in \(S\).
Then we can recursively execute the following three steps
  in order to find a derivation of \(S\{Q\}\).

\begin{enumerate}
  \item
    Find a sequent calculus derivation of \({\ttile \dual{P}, Q}\).
    These are detailed in
    \cref{fig:pq_proof_i,fig:pq_proof_switch,fig:pq_proof_assoc,fig:pq_proof_comm%
}. 

  \item
    The next step is to find a derivation of
      \({\ttile \dual{S\{P\}}, S\{Q\}}\)
      from the derivation of \({\ttile \dual{P}, Q}\).
    Consider the connective and formula that are immediately connected to \(Q\).
    For example, in the formula \[{S = (A \otimes B) \parr (C \parr Q)},\]
      the connective \parT{} and the formula \(C\) are immediately connected to
      \(Q\).
    Likewise, \(C \parr Q\) is then immediately connected to \(A \otimes B.\)
    Then simply use the derivations shown in \cref{fig:pq_to_spsq}
      until the entire context \(S\) is created,
      forming a derivation of \({\ttile \dual{S\{P\}}, S\{Q\}}\).

  \item
    Finally, use the \cutT{} rule to cut the derivation of \({\ttile S\{P\}}\),
      which is obtained by hypothesis,
      with \({\ttile \dual{S\{P\}}, S\{Q\}}\)
      to get a derivation of \({S\{Q\}}\).
\end{enumerate}
\end{proof}

\begin{figure}[!p]\centering
  \newcommand{\VERTICALSPACE}{\vspace{2em}}
  \begin{subfigure}[b]{0.25\linewidth}
    \begin{prooftree}
      \AxiomC{}
      \idL{}\UnaryInfC{\(\ttile{} \dual{ A },A\)}
      \parL{}\UnaryInfC{\(\ttile{} \dual{ A } \parr{} A \)}
    \end{prooftree}
    \caption{\iDownT}
  \end{subfigure}\qquad\qquad\qquad
  \begin{subfigure}[b]{0.35\linewidth}
    \begin{prooftree}
      \AxiomC{}
      \idL{}\UnaryInfC{\(\ttile{} \dual{B}, B\)}
      \AxiomC{}
      \idL{}\UnaryInfC{\(\ttile{} \dual{A}, A\)}
      \parL{}\UnaryInfC{\(\ttile{} \dual{A} \parr A\)}
      \otimesL{}\BinaryInfC{\(\ttile{} \dual{B}, B \otimes ( \dual{A} \parr A )\)}
    \end{prooftree}
    \caption{\iDownT}
  \end{subfigure}
  \VERTICALSPACE\\
  \begin{subfigure}[b]{0.4\linewidth}
    \begin{prooftree}
      \AxiomC{}
      \idL{}\UnaryInfC{\(\ttile{} \dual{A},A\)}
      \parL{}\UnaryInfC{\(\ttile{} \dual{A} \parr{} A\)}
      \AxiomC{}
      \idL{}\UnaryInfC{\(\ttile{} \dual{B},B\)}
      \otimesL{}\BinaryInfC{\(\ttile{} ( \dual{A} \parr{} A ) \otimes
        \dual{B},B\)} 
      \demorL{}\UnaryInfC{\(\ttile{} \dual{ (( A \otimes{} \dual{A} ) \parr B  ) },B\)}
    \end{prooftree}
    \caption{\iUpT}
  \end{subfigure}
  \caption{\label{fig:pq_proof_i}
    Proofs of the sequent \({\ttile \dual{P}, Q}\) for introduction and cut rules
  }
\end{figure}

\begin{figure}[!p]\centering
\scriptsize
  \newcommand{\SUBFIGUREWIDTH}{0.49\linewidth}
  \newcommand{\VERTICALSPACE}{\vspace{2em}}
  \begin{subfigure}[b]{\SUBFIGUREWIDTH}
        \begin{prooftree}
          \AxiomC{}
          \tidL{}\UnaryInfC{\(\ttile{} \dual{A},A\)}
          \texchL{}\UnaryInfC{\(\ttile{} A,\dual{A}\)}
          \AxiomC{}
          \tidL{}\UnaryInfC{\(\ttile{} \dual{B},B\)}
          \AxiomC{}
          \tidL{}\UnaryInfC{\(\ttile{} \dual{C},C\)}
          \texchL{}\UnaryInfC{\(\ttile{} C,\dual{C}\)}
          \totimesL{}\BinaryInfC{\(\ttile{} \dual{B},B \otimes{} C,\dual{C}\)}
          \totimesL{}\BinaryInfC{\(\ttile{} A,\dual{A} \otimes{} \dual{B},B \otimes{} C,\dual{C}\)}
          \texchL{}\UnaryInfC{\(\ttile{} \dual{A} \otimes{} \dual{B},A,B \otimes{} C,\dual{C}\)}
          \tparL{}\UnaryInfC{\(\ttile{} \dual{A} \otimes{}
            \dual{B},A \parr{} ( B \otimes{} C ) ,\dual{C}\)} 
          \texchL{}\UnaryInfC{\(\ttile{} \dual{A} \otimes{}
            \dual{B},\dual{C},A \parr{}  ( B \otimes{} C ) \)} 
          \tparL{}\UnaryInfC{\(\ttile{}  ( \dual{A} \otimes{} \dual{B} )  \parr{} \dual{C},A \parr{}  ( B \otimes{} C ) \)}
          \tdemorL{}\UnaryInfC{\(\ttile{} \dual{(( A \parr{} B )
            \otimes{} C ) },A \parr{}  ( B \otimes{} C ) \)} 
        \end{prooftree}
    \caption{Switch Left to Right}
  \end{subfigure}
  \begin{subfigure}[b]{\SUBFIGUREWIDTH}
        \begin{prooftree}
          \AxiomC{}
          \tidL{}\UnaryInfC{\(\ttile{} \dual{A},A\)}
          \AxiomC{}
          \tidL{}\UnaryInfC{\(\ttile{} \dual{B},B\)}
          \texchL{}\UnaryInfC{\(\ttile{} B,\dual{B}\)}
          \totimesL{}\BinaryInfC{\(\ttile{} \dual{A},A \otimes{} B,\dual{B}\)}
          \AxiomC{}
          \tidL{}\UnaryInfC{\(\ttile{} \dual{C},C\)}
          \totimesL{}\BinaryInfC{\(\ttile{} \dual{A},A \otimes{} B,\dual{B} \otimes{} \dual{C},C\)}
          \texchL{}\UnaryInfC{\(\ttile{} \dual{A},\dual{B} \otimes{} \dual{C},A \otimes{} B,C\)}
          \tparL{}\UnaryInfC{\(\ttile{} \dual{A} \parr{}  ( \dual{B} \otimes{} \dual{C} ) ,A \otimes{} B,C\)}
          \tparL{}\UnaryInfC{\(\ttile{} \dual{A} \parr{}  ( \dual{B} \otimes{} \dual{C} ) , ( A \otimes{} B )  \parr{} C\)}
          \tdemorL{}\UnaryInfC{\(\ttile{} \dual{ ( A \otimes{}  ( B \parr{} C )  ) }, ( A \otimes{} B )  \parr{} C\)}
        \end{prooftree}
    \caption{Switch Right to Left}
  \end{subfigure}
  \caption{\label{fig:pq_proof_switch}
    Proofs of the sequent \({\ttile \dual{P}, Q}\) for switch
  }
\end{figure}

 \begin{figure}[!p]\centering
\scriptsize
   \newcommand{\SUBFIGUREWIDTH}{0.49\linewidth}
   \newcommand{\VERTICALSPACE}{\vspace{2em}}
   \begin{subfigure}[b]{\SUBFIGUREWIDTH}
         \begin{prooftree}
           \AxiomC{}
           \idL{}\UnaryInfC{\(\ttile{} \dual{A},A\)}
           \AxiomC{}
           \idL{}\UnaryInfC{\(\ttile{} \dual{B},B\)}
           \AxiomC{}
           \idL{}\UnaryInfC{\(\ttile{} \dual{C},C\)}
           \exchL{}\UnaryInfC{\(\ttile{} C,\dual{C}\)}
           \otimesL{}\BinaryInfC{\(\ttile{} \dual{B},B \otimes{} C,\dual{C}\)}
           \exchL{}\UnaryInfC{\(\ttile{} B \otimes{} C,\dual{B},\dual{C}\)}
           \otimesL{}\BinaryInfC{\(\ttile{} \dual{A},A \otimes{}  ( B \otimes{} C ) ,\dual{B},\dual{C}\)}
           \exchL{}\UnaryInfC{\(\ttile{} \dual{A},\dual{B},A \otimes{}  ( B \otimes{} C ) ,\dual{C}\)}
           \exchL{}\UnaryInfC{\(\ttile{} \dual{A},\dual{B},\dual{C},A \otimes{}  ( B \otimes{} C ) \)}
           \parL{}\UnaryInfC{\(\ttile{} \dual{A} \parr{} \dual{B},\dual{C},A \otimes{}  ( B \otimes{} C ) \)}
           \parL{}\UnaryInfC{\(\ttile{}  ( \dual{A} \parr{} \dual{B} )  \parr{} \dual{C},A \otimes{}  ( B \otimes{} C ) \)}
           \demorL{}\UnaryInfC{\(\ttile{} \dual{ (  ( A \otimes{} B )  \otimes{} C ) },A \otimes{}  ( B \otimes{} C ) \)}
         \end{prooftree}
     \caption{\alphaDownT}
   \end{subfigure}
   \begin{subfigure}[b]{\SUBFIGUREWIDTH}
         \begin{prooftree}
           \AxiomC{}
           \idL{}\UnaryInfC{\(\ttile{} \dual{A},A\)}
           \exchL{}\UnaryInfC{\(\ttile{} A,\dual{A}\)}
           \AxiomC{}
           \idL{}\UnaryInfC{\(\ttile{} \dual{B},B\)}
           \exchL{}\UnaryInfC{\(\ttile{} B,\dual{B}\)}
           \AxiomC{}
           \idL{}\UnaryInfC{\(\ttile{} \dual{C},C\)}
           \otimesL{}\BinaryInfC{\(\ttile{} B,\dual{B} \otimes{} \dual{C},C\)}
           \exchL{}\UnaryInfC{\(\ttile{} \dual{B} \otimes{} \dual{C},B,C\)}
           \otimesL{}\BinaryInfC{\(\ttile{} A,\dual{A} \otimes{}  ( \dual{B} \otimes{} \dual{C} ) ,B,C\)}
           \exchL{}\UnaryInfC{\(\ttile{} \dual{A} \otimes{}  ( \dual{B} \otimes{} \dual{C} ) ,A,B,C\)}
           \parL{}\UnaryInfC{\(\ttile{} \dual{A} \otimes{}  ( \dual{B} \otimes{} \dual{C} ) ,A \parr{} B,C\)}
           \parL{}\UnaryInfC{\(\ttile{} \dual{A} \otimes{}  ( \dual{B} \otimes{} \dual{C} ) , ( A \parr{} B )  \parr{} C\)}
           \demorL{}\UnaryInfC{\(\ttile{} \dual{A} \otimes{} \dual{ ( B \parr{} C ) }, ( A \parr{} B )  \parr{} C\)}
           \demorL{}\UnaryInfC{\(\ttile{} \dual{ ( A \parr{}  ( B \parr{} C )  ) }, ( A \parr{} B )  \parr{} C\)}
         \end{prooftree}
     \caption{\alphaUpT}
   \end{subfigure}
   \caption{\label{fig:pq_proof_assoc}
     Proofs of the sequent \({\ttile \dual{P}, Q}\) for associativity
   }
 \end{figure}

 \begin{figure}[!p]\centering
   \newcommand{\ADJUSTBOXWIDTH}{0.42\linewidth}
   \newcommand{\SUBFIGUREWIDTH}{0.49\linewidth}
   \newcommand{\VERTICALSPACE}{\vspace{2em}}
   \begin{subfigure}[b]{\SUBFIGUREWIDTH}
     \begin{prooftree}
       \AxiomC{}
       \idL{}\UnaryInfC{\(\ttile{} \dual{A},A\)}
       \exchL{}\UnaryInfC{\(\ttile{} A,\dual{A}\)}
       \AxiomC{}
       \idL{}\UnaryInfC{\(\ttile{} \dual{B},B\)}
       \otimesL{}\BinaryInfC{\(\ttile{} A,\dual{A} \otimes{} \dual{B},B\)}
       \exchL{}\UnaryInfC{\(\ttile{} \dual{A} \otimes{} \dual{B},A,B\)}
       \exchL{}\UnaryInfC{\(\ttile{} \dual{A} \otimes{} \dual{B},B,A\)}
       \parL{}\UnaryInfC{\(\ttile{} \dual{A} \otimes{} \dual{B},B \parr{} A\)}
       \demorL{}\UnaryInfC{\(\ttile{} \dual{ ( A \parr{} B ) },B \parr{} A\)}
     \end{prooftree}
     \caption{\sigmaUpT}
   \end{subfigure}
   \begin{subfigure}[b]{\SUBFIGUREWIDTH}
     \begin{prooftree}
       \AxiomC{}
       \idL{}\UnaryInfC{\(\ttile{} \dual{B},B\)}
       \AxiomC{}
       \idL{}\UnaryInfC{\(\ttile{} \dual{A},A\)}
       \exchL{}\UnaryInfC{\(\ttile{} A,\dual{A}\)}
       \otimesL{}\BinaryInfC{\(\ttile{} \dual{B},B \otimes{} A,\dual{A}\)}
       \exchL{}\UnaryInfC{\(\ttile{} \dual{B},\dual{A},B \otimes{} A\)}
       \exchL{}\UnaryInfC{\(\ttile{} \dual{A},\dual{B},B \otimes{} A\)}
       \parL{}\UnaryInfC{\(\ttile{} \dual{A} \parr{} \dual{B},B \otimes{} A\)}
       \demorL{}\UnaryInfC{\(\ttile{} \dual{ ( A \otimes{} B ) },B \otimes{} A\)}
     \end{prooftree}
     \caption{\sigmaDownT}
   \end{subfigure}
   \caption{\label{fig:pq_proof_comm}
     Proofs of the sequent \({\ttile \dual{P}, Q}\) for commutativity
   }
 \end{figure}

\begin{figure}[!p]\centering
  \newcommand{\SUBFIGUREWIDTH}{0.45\linewidth}
  \newcommand{\VERTICALSPACE}{\vspace{2em}}
  \begin{subfigure}[b]{\SUBFIGUREWIDTH}
    \begin{prooftree}
      \AxiomC{}
      \idL{}\UnaryInfC{\(\ttile{} \dual{B},B\)}
      \arbitrarySCAxioms{\(\deriv{}\)}
      \UnaryInfC{\(\ttile{} \dual{P},Q\)}
      \exchL{}\UnaryInfC{\(\ttile{} Q,\dual{P}\)}
      \otimesL{}\BinaryInfC{\(\ttile{} \dual{B},B \otimes{} Q,\dual{P}\)}
      \exchL{}\UnaryInfC{\(\ttile{} \dual{B},\dual{P},B \otimes{} Q\)}
      \parL{}\UnaryInfC{\(\ttile{} \dual{B} \parr{} \dual{P},B \otimes{} Q\)}
      \demorL{}\UnaryInfC{\(\ttile{} \dual{ ( B \otimes{} P ) },B \otimes{} Q\)}
    \end{prooftree}
    \caption{Tensor appended to left}
  \end{subfigure}\qquad
  \begin{subfigure}[b]{\SUBFIGUREWIDTH}
    \begin{prooftree}
      \arbitrarySCAxioms{\(\deriv{}\)}
      \UnaryInfC{\(\ttile{} \dual{P},Q\)}
      \AxiomC{}
      \idL{}\UnaryInfC{\(\ttile{} \dual{B},B\)}
      \exchL{}\UnaryInfC{\(\ttile{} B,\dual{B}\)}
      \otimesL{}\BinaryInfC{\(\ttile{} \dual{P},Q \otimes{} B,\dual{B}\)}
      \exchL{}\UnaryInfC{\(\ttile{} \dual{P},\dual{B},Q \otimes{} B\)}
      \parL{}\UnaryInfC{\(\ttile{} \dual{P} \parr{} \dual{B},Q \otimes{} B\)}
      \demorL{}\UnaryInfC{\(\ttile{} \dual{ ( P \otimes{} B ) },Q \otimes{} B\)}
    \end{prooftree}
    \caption{Tensor appended to right}
  \end{subfigure}
  \VERTICALSPACE\\
  \begin{subfigure}[b]{\SUBFIGUREWIDTH}
    \begin{prooftree}
      \AxiomC{}
      \idL{}\UnaryInfC{\(\ttile{} \dual{B},B\)}
      \exchL{}\UnaryInfC{\(\ttile{} B,\dual{B}\)}
      \arbitrarySCAxioms{\(\deriv{}\)}
      \UnaryInfC{\(\ttile{} \dual{P},Q\)}
      \otimesL{}\BinaryInfC{\(\ttile{} B,\dual{B} \otimes{} \dual{P},Q\)}
      \exchL{}\UnaryInfC{\(\ttile{} \dual{B} \otimes{} \dual{P},B,Q\)}
      \parL{}\UnaryInfC{\(\ttile{} \dual{B} \otimes{} \dual{P},B \parr{} Q\)}
      \demorL{}\UnaryInfC{\(\ttile{} \dual{ ( B \parr{} P ) },B \parr{} Q\)}
    \end{prooftree}
    \caption{Par appended to left}
  \end{subfigure}\qquad
  \begin{subfigure}[b]{\SUBFIGUREWIDTH}
    \begin{prooftree}
      \arbitrarySCAxioms{\(\deriv{}\)}
      \UnaryInfC{\(\ttile{} \dual{P},Q\)}
      \exchL{}\UnaryInfC{\(\ttile{} Q,\dual{P}\)}
      \AxiomC{}
      \idL{}\UnaryInfC{\(\ttile{} \dual{B},B\)}
      \otimesL{}\BinaryInfC{\(\ttile{} Q,\dual{P} \otimes{} \dual{B},B\)}
      \exchL{}\UnaryInfC{\(\ttile{} \dual{P} \otimes{} \dual{B},Q,B\)}
      \parL{}\UnaryInfC{\(\ttile{} \dual{P} \otimes{} \dual{B},Q \parr{} B\)}
      \demorL{}\UnaryInfC{\(\ttile{} \dual{ ( P \parr{} B ) },Q \parr{} B\)}
    \end{prooftree}
    \caption{Par appended to right}
  \end{subfigure}
  \caption[Derivations of \(\ttile \dual{P}, Q\) into \(\ttile \dual{S\{P\}, S\{Q\}}\)]
    {\label{fig:pq_to_spsq}
    Derivations showing how to append a formula \(B\)
    with pars and tensors
    to the left and right of the formulae
    in the sequent \({\ttile \dual{P}, Q}\),
    obtained by a derivation \(\deriv{}\).
  }
\end{figure}
It is important to note that changing a derivation using the
transformation described here does \emph{not}\/ result in a derivation that
has a similar size. In fact, the size of the resulting sequent
calculus derivation grows rapidly with each new inference rule added to
the deep inference derivation. For example, the sequent calculus derivation
corresponding to the derivation in \cref{fig:disample} is given in \cref{fig:di2sc}.

\begin{figure}[!hp]
  \centering
\begin{prooftree}
\AxiomC{}
\aiDownL{}\UnaryInfC{\( \dual{b} \parr{} b \)}
\aiDownL{}\UnaryInfC{\( ( \dual{b} \otimes{}  ( \dual{a} \parr{} a )  )  \parr{} b \)}
\sigmaDownL{}\UnaryInfC{\( (  ( \dual{a} \parr{} a )  \otimes{} \dual{b} )  \parr{} b \)}
\sigmaSwitchL{}\UnaryInfC{\( ( \dual{a} \parr{}  ( a \otimes{} \dual{b} )  )  \parr{} b \)}
\end{prooftree}
  \caption{Sample deep inference derivation}
  \label{fig:disample}
\end{figure}
Notice that the left premise at each \cutT{} rule has the same
formula as a corresponding deep inference rule. Meanwhile, the right-hand
premises use many more rules to establish the next corresponding
rule's formula and the negation of the previous formula.

\begin{sidewaysfigure}[!hp]
\makebox[\textheight]{%
\noindent\begin{minipage}{\textheight}
\noindent\tiny
\begin{prooftree}
\AxiomC{}
\tidL{}\UnaryInfC{\(\ttile{} \dual{b},b\)}
\tparL{}\UnaryInfC{\(\ttile{} \dual{b} \parr{} b\)}
\AxiomC{}
\tidL{}\UnaryInfC{\(\ttile{} b,\dual{b}\)}
\AxiomC{}
\tidL{}\UnaryInfC{\(\ttile{} \dual{a},a\)}
\tparL{}\UnaryInfC{\(\ttile{} \dual{a} \parr{} a\)}
\totimesL{}\BinaryInfC{\(\ttile{} b,\dual{b} \otimes{}  ( \dual{a} \parr{} a ) \)}
\texchL{}\UnaryInfC{\(\ttile{} \dual{b} \otimes{}  ( \dual{a} \parr{} a ) ,b\)}
\AxiomC{}
\tidL{}\UnaryInfC{\(\ttile{} \dual{b},b\)}
\totimesL{}\BinaryInfC{\(\ttile{} \dual{b} \otimes{}  ( \dual{a} \parr{} a ) ,b \otimes{} \dual{b},b\)}
\texchL{}\UnaryInfC{\(\ttile{} b \otimes{} \dual{b},\dual{b} \otimes{}  ( \dual{a} \parr{} a ) ,b\)}
\tparL{}\UnaryInfC{\(\ttile{} b \otimes{} \dual{b}, ( \dual{b} \otimes{}  ( \dual{a} \parr{} a )  )  \parr{} b\)}
\tdemorL{}\UnaryInfC{\(\ttile{} \dual{ ( \dual{b} \parr{} b ) }, ( \dual{b} \otimes{}  ( \dual{a} \parr{} a )  )  \parr{} b\)}
\tcutL{}\BinaryInfC{\(\ttile{}  ( \dual{b} \otimes{}  ( \dual{a} \parr{} a )  )  \parr{} b\)}
\AxiomC{}
\tidL{}\UnaryInfC{\(\ttile{} \dual{ ( \dual{a} \parr{} a ) },\dual{a} \parr{} a\)}
\AxiomC{}
\tidL{}\UnaryInfC{\(\ttile{} \dual{b},b\)}
\totimesL{}\BinaryInfC{\(\ttile{} \dual{ ( \dual{a} \parr{} a ) }, ( \dual{a} \parr{} a )  \otimes{} \dual{b},b\)}
\texchL{}\UnaryInfC{\(\ttile{} \dual{ ( \dual{a} \parr{} a ) },b, ( \dual{a} \parr{} a )  \otimes{} \dual{b}\)}
\texchL{}\UnaryInfC{\(\ttile{} b,\dual{ ( \dual{a} \parr{} a ) }, ( \dual{a} \parr{} a )  \otimes{} \dual{b}\)}
\tparL{}\UnaryInfC{\(\ttile{} b \parr{} \dual{ ( \dual{a} \parr{} a ) }, ( \dual{a} \parr{} a )  \otimes{} \dual{b}\)}
\tdemorL{}\UnaryInfC{\(\ttile{} \dual{ ( \dual{b} \otimes{}  ( \dual{a} \parr{} a )  ) }, ( \dual{a} \parr{} a )  \otimes{} \dual{b}\)}
\texchL{}\UnaryInfC{\(\ttile{}  ( \dual{a} \parr{} a )  \otimes{} \dual{b},\dual{ ( \dual{b} \otimes{}  ( \dual{a} \parr{} a )  ) }\)}
\AxiomC{}
\tidL{}\UnaryInfC{\(\ttile{} \dual{b},b\)}
\totimesL{}\BinaryInfC{\(\ttile{}  ( \dual{a} \parr{} a )  \otimes{} \dual{b},\dual{ ( \dual{b} \otimes{}  ( \dual{a} \parr{} a )  ) } \otimes{} \dual{b},b\)}
\texchL{}\UnaryInfC{\(\ttile{} \dual{ ( \dual{b} \otimes{}  ( \dual{a} \parr{} a )  ) } \otimes{} \dual{b}, ( \dual{a} \parr{} a )  \otimes{} \dual{b},b\)}
\tparL{}\UnaryInfC{\(\ttile{} \dual{ ( \dual{b} \otimes{}  ( \dual{a} \parr{} a )  ) } \otimes{} \dual{b}, (  ( \dual{a} \parr{} a )  \otimes{} \dual{b} )  \parr{} b\)}
\tdemorL{}\UnaryInfC{\(\ttile{} \dual{ (  ( \dual{b} \otimes{}  ( \dual{a} \parr{} a )  )  \parr{} b ) }, (  ( \dual{a} \parr{} a )  \otimes{} \dual{b} )  \parr{} b\)}
\tcutL{}\BinaryInfC{\(\ttile{}  (  ( \dual{a} \parr{} a )  \otimes{} \dual{b} )  \parr{} b\)}
\AxiomC{}
\tidL{}\UnaryInfC{\(\ttile{} \dual{a},a\)}
\AxiomC{}
\tidL{}\UnaryInfC{\(\ttile{} \dual{a},a\)}
\AxiomC{}
\tidL{}\UnaryInfC{\(\ttile{} \dual{b},b\)}
\totimesL{}\BinaryInfC{\(\ttile{} \dual{a},a \otimes{} \dual{b},b\)}
\totimesL{}\BinaryInfC{\(\ttile{} \dual{a},a \otimes{} \dual{a},a \otimes{} \dual{b},b\)}
\texchL{}\UnaryInfC{\(\ttile{} a \otimes{} \dual{a},\dual{a},a \otimes{} \dual{b},b\)}
\tparL{}\UnaryInfC{\(\ttile{} a \otimes{} \dual{a},\dual{a} \parr{}  ( a \otimes{} \dual{b} ) ,b\)}
\texchL{}\UnaryInfC{\(\ttile{} a \otimes{} \dual{a},b,\dual{a} \parr{}  ( a \otimes{} \dual{b} ) \)}
\tparL{}\UnaryInfC{\(\ttile{}  ( a \otimes{} \dual{a} )  \parr{} b,\dual{a} \parr{}  ( a \otimes{} \dual{b} ) \)}
\tdemorL{}\UnaryInfC{\(\ttile{} \dual{ (  ( \dual{a} \parr{} a )  \otimes{} \dual{b} ) },\dual{a} \parr{}  ( a \otimes{} \dual{b} ) \)}
\texchL{}\UnaryInfC{\(\ttile{} \dual{a} \parr{}  ( a \otimes{} \dual{b} ) ,\dual{ (  ( \dual{a} \parr{} a )  \otimes{} \dual{b} ) }\)}
\AxiomC{}
\tidL{}\UnaryInfC{\(\ttile{} \dual{b},b\)}
\totimesL{}\BinaryInfC{\(\ttile{} \dual{a} \parr{}  ( a \otimes{} \dual{b} ) ,\dual{ (  ( \dual{a} \parr{} a )  \otimes{} \dual{b} ) } \otimes{} \dual{b},b\)}
\texchL{}\UnaryInfC{\(\ttile{} \dual{ (  ( \dual{a} \parr{} a )  \otimes{} \dual{b} ) } \otimes{} \dual{b},\dual{a} \parr{}  ( a \otimes{} \dual{b} ) ,b\)}
\tparL{}\UnaryInfC{\(\ttile{} \dual{ (  ( \dual{a} \parr{} a )  \otimes{} \dual{b} ) } \otimes{} \dual{b}, ( \dual{a} \parr{}  ( a \otimes{} \dual{b} )  )  \parr{} b\)}
\tdemorL{}\UnaryInfC{\(\ttile{} \dual{ (  (  ( \dual{a} \parr{} a )  \otimes{} \dual{b} )  \parr{} b ) }, ( \dual{a} \parr{}  ( a \otimes{} \dual{b} )  )  \parr{} b\)}
\tcutL{}\BinaryInfC{\(\ttile{}  ( \dual{a} \parr{}  ( a \otimes{} \dual{b} )  )  \parr{} b\)}
\end{prooftree}
\end{minipage}}
  \caption[Translation from a deep inference derivation to a sequent
  calculus one]{%
\label{fig:di2sc}Translation $\transdisc$ applied to the  deep
    inference derivation in \cref{fig:disample}}\vspace{-35pt}
\end{sidewaysfigure}

Similarly we may give a translation $\transscdi$ of derivations in the
system $\mllm{SC}$ to those in $\smllm{DI}$. First of all, we need to
translate sequents into formulae, and we do this by transferring the
sequent \({\ttile \Gamma \;=\; \ttile A_1, \dots, A_n}\) to the formula
\[{[\Gamma] = ((\dots(A_1 \parr A_2) \parr \dots) \parr A_n).}\]

\begin{proposition}\label{sctodi}
  For every derivation $\deriv{}$ of a sequent $\Gamma$ in $\mllm{SC}$
  there is a derivation $\transscdi\deriv{}$ in $\mllm{DI}$ of the formula consisting of
  all formulae in $\Gamma$ connected by $\llpar$ from left to right.
\end{proposition}
\begin{proof}
We give instructions how to transfer each instance of a rule,
and then a whole derivation may be transferred by applying these
rules recursively.

\begin{enumerate}
  \item For the base case, \idT{} (or \axT{}),
    this is simply the \iDownT{} rule or \aiDownT{} rule.
  \item If we have a derivation of the form:
    \begin{prooftree}
      \arbitrarySCAxioms{\(\deriv{0}\)}
      \UnaryInfC{\(\ttile \Gamma, A, B, \Delta\)}
      \parL{}\UnaryInfC{\(\ttile \Gamma, A \parr B, \Delta\)}
    \end{prooftree}
    then by hypothesis we have a derivation \(\transscdi\deriv{0}\) of
    \({S\{([\Gamma] \parr A) \parr B\}}\).
    Note that \({\ttile \Gamma, A \parr B, \Delta}\)
      is taken to be \({S\{[\Gamma] \parr (A \parr B)\}}\) in deep inference.
    Then the deep inference derivation is simply:
    \begin{prooftree}
      \arbitraryDIAxioms{\(\transscdi\deriv{0}\)}
      \UnaryInfC{\(S\{([\Gamma] \parr A) \parr B\}\)}
      \sigmaUpL{}\UnaryInfC{\(S\{B \parr ([\Gamma] \parr A)\}\)}
      \sigmaUpL{}\UnaryInfC{\(S\{B \parr (A \parr [\Gamma])\}\)}
      \alphaUpL{}\UnaryInfC{\(S\{(B \parr A) \parr [\Gamma]\}\)}
      \sigmaUpL{}\UnaryInfC{\(S\{[\Gamma] \parr (B \parr A)\}\)}
      \sigmaUpL{}\UnaryInfC{\(S\{[\Gamma] \parr (A \parr B)\}\)}
    \end{prooftree}
  \item If we have a derivation of the form:
    \begin{prooftree}
      \arbitrarySCAxioms{\(\deriv{0}\)}
      \UnaryInfC{\(\ttile \Gamma, A, B, \Delta\)}
      \exchL{}\UnaryInfC{\(\ttile \Gamma, B, A, \Delta\)}
    \end{prooftree}
    then by hypothesis we have a derivation \(\transscdi\deriv{0}\) of
    \(S\{(([\Gamma] \parr A) \parr B)\}\).
    Now, \({\ttile \Gamma, B, A, \Delta}\)
      is taken to be \({S\{(([\Gamma] \parr B) \parr A)\}}\) in deep inference.
    Then the deep inference derivation is simply:
    \begin{prooftree}
      \arbitraryDIAxioms{\(\transscdi\deriv{0}\)}
      \UnaryInfC{\(S\{([\Gamma] \parr A) \parr B\}\)}
      \sigmaUpL{}\UnaryInfC{\(S\{B \parr ([\Gamma] \parr A)\}\)}
      \alphaUpL{}\UnaryInfC{\(S\{(B \parr [\Gamma]) \parr A\}\)}
      \sigmaUpL{}\UnaryInfC{\(S\{([\Gamma] \parr B) \parr A\}\)}
    \end{prooftree}
  \item If we have a derivation of the form:
    \begin{prooftree}
      \arbitrarySCAxioms{\(\deriv{0}\)}
      \UnaryInfC{\(\ttile \Gamma, A\)}
      \arbitrarySCAxioms{\(\deriv{1}\)}
      \UnaryInfC{\(\ttile B, \Delta\)}
      \otimesL{}\BinaryInfC{\(\ttile \Gamma, A \otimes B, \Delta\)}
    \end{prooftree}
    then by hypothesis we a have derivation \(\transscdi\deriv{0}\) of \({[\Gamma] \parr A}\)
    of the form:
    \begin{prooftree}
      \AxiomC{}
      \ProofLabel{\aiDownT/\iDownT}\UnaryInfC{\(\dual{A_0} \parr A_0\)}
      \arbitraryDIRules{\(\deriv{0}'\)}
      \UnaryInfC{\([\Gamma] \parr A\)}
    \end{prooftree}
    and a derivation \(\transscdi\deriv{1}\) of \({S\{B\}}\).
    Then we want a deep inference derivation of
    \({S\{[\Gamma] \parr (A \otimes B)\}}\),
      which is given by:
    \begin{prooftree}
      \arbitraryDIAxioms{\(\transscdi\deriv{1}\)}
      \UnaryInfC{\(S\{B\}\)}
      \ProofLabel{\aiDownT/\iDownT}\UnaryInfC{\(S\{B \otimes (\dual{A_0} \parr A_0)\}\)}
      \arbitraryDIRules{\(S\{B \otimes \deriv{0}'\}\)}
      \UnaryInfC{\(S\{B \otimes ([\Gamma] \parr A)\}\)}
      \sigmaDownL{}\UnaryInfC{\(S\{([\Gamma] \parr A) \otimes B\}\)}
      \sigmaSwitchL{}\UnaryInfC{\(S\{[\Gamma] \parr (A \otimes B)\}\)}
    \end{prooftree}
  \item If we have a derivation of the form:
    \begin{prooftree}
      \arbitrarySCAxioms{\(\deriv{0}\)}
      \UnaryInfC{\(\ttile \Gamma, A\)}
      \arbitrarySCAxioms{\(\deriv{1}\)}
      \UnaryInfC{\(\ttile \dual{A}, \Delta\)}
      \cutL{}\BinaryInfC{\(\ttile \Gamma, \Delta\)}
    \end{prooftree}
    then by hypothesis we have a derivation \(\transscdi\deriv{0}\) of \({[\Gamma] \parr A}\)
    of the form:
    \begin{prooftree}
      \AxiomC{}
      \ProofLabel{\aiDownT/\iDownT}\UnaryInfC{\(\dual{A_0} \parr A_0\)}
      \arbitraryDIRules{\(\deriv{0}'\)}
      \UnaryInfC{\([\Gamma] \parr A\)}
    \end{prooftree}
    and a derivation \(\transscdi\deriv{1}\) of \({S\{\dual{A}\}}\).
    Then we want a deep inference derivation of
    \({S\{[\Gamma]\}}\), which is given by:
    \begin{prooftree}
      \arbitraryDIAxioms{\(\transscdi\deriv{1}\)}
      \UnaryInfC{\(S\{\dual{A}\}\)}
      \ProofLabel{\aiDownT/\iDownT}\UnaryInfC{\(S\{\dual{A} \otimes (\dual{A_0} \parr A_0)\}\)}
      \arbitraryDIRules{\(S\{\dual{A} \otimes \deriv{0}'\}\)}
      \UnaryInfC{\(S\{\dual{A} \otimes ([\Gamma] \parr A)\}\)}
      \sigmaDownL{}\UnaryInfC{\(S\{([\Gamma] \parr A) \otimes \dual{A}\}\)}
      \sigmaSwitchL{}\UnaryInfC{\(S\{[\Gamma] \parr (A \otimes \dual{A})\}\)}
      \sigmaUpL{}\UnaryInfC{\(S\{(A \otimes \dual{A}) \parr [\Gamma]\}\)}
      \iUpL{}\UnaryInfC{\(S\{[\Gamma]\}\)}
    \end{prooftree}
\end{enumerate}
\end{proof}

Notice that just as for the transformation going in the opposite
direction the resulting output derivation will usually be larger and
more complex than the input. For example, consider the following
sequent calculus derivation:
\begin{prooftree}
\AxiomC{}
\axL{}\UnaryInfC{\(\ttile{} \dual{a},a\)}
\AxiomC{}
\axL{}\UnaryInfC{\(\ttile{} \dual{b},b\)}
\otimesL{}\BinaryInfC{\(\ttile{} \dual{a},a \otimes{} \dual{b},b\)}
\AxiomC{}
\axL{}\UnaryInfC{\(\ttile{} \dual{b},b\)}
\AxiomC{}
\axL{}\UnaryInfC{\(\ttile{} \dual{c},c\)}
\otimesL{}\BinaryInfC{\(\ttile{} \dual{b},b \otimes{} \dual{c},c\)}
\cutL{}\BinaryInfC{\(\ttile{} \dual{a},a \otimes{} \dual{b},b \otimes{} \dual{c},c\)}
\end{prooftree}

Its translation into the deep inference system gives the following:
\begin{prooftree}
\AxiomC{}
\aiDownL{}\UnaryInfC{\( ( \dual{c} \parr{} c ) \)}
\aiDownL{}\UnaryInfC{\( (  ( \dual{c} \otimes{}  ( \dual{b} \parr{} b )  )  \parr{} c ) \)}
\sigmaDownL{}\UnaryInfC{\( (  (  ( \dual{b} \parr{} b )  \otimes{} \dual{c} )  \parr{} c ) \)}
\sigmaSwitchL{}\UnaryInfC{\( (  ( \dual{b} \parr{}  ( b \otimes{} \dual{c} )  )  \parr{} c ) \)}
\aiDownL{}\UnaryInfC{\( (  (  ( \dual{b} \otimes{}  ( \dual{b} \parr{} b )  )  \parr{}  ( b \otimes{} \dual{c} )  )  \parr{} c ) \)}
\aiDownL{}\UnaryInfC{\( (  (  ( \dual{b} \otimes{}  (  ( \dual{b} \otimes{}  ( \dual{a} \parr{} a )  )  \parr{} b )  )  \parr{}  ( b \otimes{} \dual{c} )  )  \parr{} c ) \)}
\sigmaDownL{}\UnaryInfC{\( (  (  ( \dual{b} \otimes{}  (  (  ( \dual{a} \parr{} a )  \otimes{} \dual{b} )  \parr{} b )  )  \parr{}  ( b \otimes{} \dual{c} )  )  \parr{} c ) \)}
\sigmaSwitchL{}\UnaryInfC{\( (  (  ( \dual{b} \otimes{}  (  ( \dual{a} \parr{}  ( a \otimes{} \dual{b} )  )  \parr{} b )  )  \parr{}  ( b \otimes{} \dual{c} )  )  \parr{} c ) \)}
\sigmaDownL{}\UnaryInfC{\( (  (  (  (  ( \dual{a} \parr{}  ( a \otimes{} \dual{b} )  )  \parr{} b )  \otimes{} \dual{b} )  \parr{}  ( b \otimes{} \dual{c} )  )  \parr{} c ) \)}
\sigmaSwitchL{}\UnaryInfC{\( (  (  (  ( \dual{a} \parr{}  ( a \otimes{} \dual{b} )  )  \parr{}  ( b \otimes{} \dual{b} )  )  \parr{}  ( b \otimes{} \dual{c} )  )  \parr{} c ) \)}
\sigmaUpL{}\UnaryInfC{\( (  (  (  ( b \otimes{} \dual{b} )  \parr{}  ( \dual{a} \parr{}  ( a \otimes{} \dual{b} )  )  )  \parr{}  ( b \otimes{} \dual{c} )  )  \parr{} c ) \)}
\aiUpL{}\UnaryInfC{\( (  (  ( \dual{a} \parr{}  ( a \otimes{} \dual{b} )  )  \parr{}  ( b \otimes{} \dual{c} )  )  \parr{} c ) \)}
\end{prooftree}

\begin{figure}[!htbp]

\tiny
\begin{prooftree}
\AxiomC{}
\taiDownL{}\UnaryInfC{\( ( \dual{b} \parr{} b ) \)}
\taiDownL{}\UnaryInfC{\( (  ( \dual{b} \otimes{}  ( b \parr{} \dual{b} )  )  \parr{} b ) \)}
\tsigmaUpL{}\UnaryInfC{\( (  ( \dual{b} \otimes{}  ( \dual{b} \parr{} b )  )  \parr{} b ) \)}
\taiDownL{}\UnaryInfC{\( (  ( \dual{b} \otimes{}  (  ( \dual{b} \otimes{}  ( \dual{a} \parr{} a )  )  \parr{} b )  )  \parr{} b ) \)}
\tsigmaDownL{}\UnaryInfC{\( (  ( \dual{b} \otimes{}  (  (  ( \dual{a} \parr{} a )  \otimes{} \dual{b} )  \parr{} b )  )  \parr{} b ) \)}
\tsigmaSwitchL{}\UnaryInfC{\( (  ( \dual{b} \otimes{}  (  ( \dual{a} \parr{}  ( a \otimes{} \dual{b} )  )  \parr{} b )  )  \parr{} b ) \)}
\taiDownL{}\UnaryInfC{\( (  ( \dual{b} \otimes{}  (  (  ( \dual{a} \otimes{}  ( a \parr{} \dual{a} )  )  \parr{}  ( a \otimes{} \dual{b} )  )  \parr{} b )  )  \parr{} b ) \)}
\tsigmaalphaL{}\UnaryInfC{\( (  ( \dual{b} \otimes{}  (  (  (  ( \dual{a} \parr{} a )  \otimes{} \dual{a} )  \parr{}  ( a \otimes{} \dual{b} )  )  \parr{} b )  )  \parr{} b ) \)}
\tsigmaSwitchL{}\UnaryInfC{\( (  ( \dual{b} \otimes{}  (  (  ( \dual{a} \parr{}  ( a \otimes{} \dual{a} )  )  \parr{}  ( a \otimes{} \dual{b} )  )  \parr{} b )  )  \parr{} b ) \)}
\tsigmaalphaL{}\UnaryInfC{\( (  ( \dual{b} \otimes{}  (  (  ( a \otimes{} \dual{a} )  \parr{} b )  \parr{}  ( \dual{a} \parr{}  ( a \otimes{} \dual{b} )  )  )  )  \parr{} b ) \)}
\tdemorL{}\UnaryInfC{\( (  ( \dual{b} \otimes{}  (  ( \dual{ ( \dual{a} \parr{} a ) } \parr{} b )  \parr{}  ( \dual{a} \parr{}  ( a \otimes{} \dual{b} )  )  )  )  \parr{} b ) \)}
\tdemorL{}\UnaryInfC{\( (  ( \dual{b} \otimes{}  ( \dual{ (  ( \dual{a} \parr{} a )  \otimes{} \dual{b} ) } \parr{}  ( \dual{a} \parr{}  ( a \otimes{} \dual{b} )  )  )  )  \parr{} b ) \)}
\tsigmaalphaL{}\UnaryInfC{\( (  (  (  ( \dual{a} \parr{}  ( a \otimes{} \dual{b} )  )  \parr{} \dual{ (  ( \dual{a} \parr{} a )  \otimes{} \dual{b} ) } )  \otimes{} \dual{b} )  \parr{} b ) \)}
\tsigmaSwitchL{}\UnaryInfC{\( (  (  ( \dual{a} \parr{}  ( a \otimes{} \dual{b} )  )  \parr{}  ( \dual{ (  ( \dual{a} \parr{} a )  \otimes{} \dual{b} ) } \otimes{} \dual{b} )  )  \parr{} b ) \)}
\tsigmaalphaL{}\UnaryInfC{\( (  ( \dual{ (  ( \dual{a} \parr{} a )  \otimes{} \dual{b} ) } \otimes{} \dual{b} )  \parr{}  (  ( \dual{a} \parr{}  ( a \otimes{} \dual{b} )  )  \parr{} b )  ) \)}
\tdemorL{}\UnaryInfC{\( ( \dual{ (  (  ( \dual{a} \parr{} a )  \otimes{} \dual{b} )  \parr{} b ) } \parr{}  (  ( \dual{a} \parr{}  ( a \otimes{} \dual{b} )  )  \parr{} b )  ) \)}
\taiDownL{}\UnaryInfC{\( (  ( \dual{ (  (  ( \dual{a} \parr{} a )  \otimes{} \dual{b} )  \parr{} b ) } \otimes{}  ( \dual{b} \parr{} b )  )  \parr{}  (  ( \dual{a} \parr{}  ( a \otimes{} \dual{b} )  )  \parr{} b )  ) \)}
\taiDownL{}\UnaryInfC{\( (  ( \dual{ (  (  ( \dual{a} \parr{} a )  \otimes{} \dual{b} )  \parr{} b ) } \otimes{}  (  ( \dual{b} \otimes{}  ( b \parr{} \dual{b} )  )  \parr{} b )  )  \parr{}  (  ( \dual{a} \parr{}  ( a \otimes{} \dual{b} )  )  \parr{} b )  ) \)}
\tsigmaUpL{}\UnaryInfC{\( (  ( \dual{ (  (  ( \dual{a} \parr{} a )  \otimes{} \dual{b} )  \parr{} b ) } \otimes{}  (  ( \dual{b} \otimes{}  ( \dual{b} \parr{} b )  )  \parr{} b )  )  \parr{}  (  ( \dual{a} \parr{}  ( a \otimes{} \dual{b} )  )  \parr{} b )  ) \)}
\taiDownL{}\UnaryInfC{\( (  ( \dual{ (  (  ( \dual{a} \parr{} a )  \otimes{} \dual{b} )  \parr{} b ) } \otimes{}  (  ( \dual{b} \otimes{}  (  ( \dual{b} \otimes{}  ( \dual{ ( \dual{a} \parr{} a ) } \parr{}  ( \dual{a} \parr{} a )  )  )  \parr{} b )  )  \parr{} b )  )  \parr{}  (  ( \dual{a} \parr{}  ( a \otimes{} \dual{b} )  )  \parr{} b )  ) \)}
\tsigmaDownL{}\UnaryInfC{\( (  ( \dual{ (  (  ( \dual{a} \parr{} a )  \otimes{} \dual{b} )  \parr{} b ) } \otimes{}  (  ( \dual{b} \otimes{}  (  (  ( \dual{ ( \dual{a} \parr{} a ) } \parr{}  ( \dual{a} \parr{} a )  )  \otimes{} \dual{b} )  \parr{} b )  )  \parr{} b )  )  \parr{}  (  ( \dual{a} \parr{}  ( a \otimes{} \dual{b} )  )  \parr{} b )  ) \)}
\tsigmaSwitchL{}\UnaryInfC{\( (  ( \dual{ (  (  ( \dual{a} \parr{} a )  \otimes{} \dual{b} )  \parr{} b ) } \otimes{}  (  ( \dual{b} \otimes{}  (  ( \dual{ ( \dual{a} \parr{} a ) } \parr{}  (  ( \dual{a} \parr{} a )  \otimes{} \dual{b} )  )  \parr{} b )  )  \parr{} b )  )  \parr{}  (  ( \dual{a} \parr{}  ( a \otimes{} \dual{b} )  )  \parr{} b )  ) \)}
\tsigmaalphaL{}\UnaryInfC{\( (  ( \dual{ (  (  ( \dual{a} \parr{} a )  \otimes{} \dual{b} )  \parr{} b ) } \otimes{}  (  ( \dual{b} \otimes{}  (  ( b \parr{} \dual{ ( \dual{a} \parr{} a ) } )  \parr{}  (  ( \dual{a} \parr{} a )  \otimes{} \dual{b} )  )  )  \parr{} b )  )  \parr{}  (  ( \dual{a} \parr{}  ( a \otimes{} \dual{b} )  )  \parr{} b )  ) \)}
\tdemorL{}\UnaryInfC{\( (  ( \dual{ (  (  ( \dual{a} \parr{} a )  \otimes{} \dual{b} )  \parr{} b ) } \otimes{}  (  ( \dual{b} \otimes{}  ( \dual{ ( \dual{b} \otimes{}  ( \dual{a} \parr{} a )  ) } \parr{}  (  ( \dual{a} \parr{} a )  \otimes{} \dual{b} )  )  )  \parr{} b )  )  \parr{}  (  ( \dual{a} \parr{}  ( a \otimes{} \dual{b} )  )  \parr{} b )  ) \)}
\tsigmaalphaL{}\UnaryInfC{\( (  ( \dual{ (  (  ( \dual{a} \parr{} a )  \otimes{} \dual{b} )  \parr{} b ) } \otimes{}  (  (  (  (  ( \dual{a} \parr{} a )  \otimes{} \dual{b} )  \parr{} \dual{ ( \dual{b} \otimes{}  ( \dual{a} \parr{} a )  ) } )  \otimes{} \dual{b} )  \parr{} b )  )  \parr{}  (  ( \dual{a} \parr{}  ( a \otimes{} \dual{b} )  )  \parr{} b )  ) \)}
\tsigmaSwitchL{}\UnaryInfC{\( (  ( \dual{ (  (  ( \dual{a} \parr{} a )  \otimes{} \dual{b} )  \parr{} b ) } \otimes{}  (  (  (  ( \dual{a} \parr{} a )  \otimes{} \dual{b} )  \parr{}  ( \dual{ ( \dual{b} \otimes{}  ( \dual{a} \parr{} a )  ) } \otimes{} \dual{b} )  )  \parr{} b )  )  \parr{}  (  ( \dual{a} \parr{}  ( a \otimes{} \dual{b} )  )  \parr{} b )  ) \)}
\tsigmaalphaL{}\UnaryInfC{\( (  ( \dual{ (  (  ( \dual{a} \parr{} a )  \otimes{} \dual{b} )  \parr{} b ) } \otimes{}  (  ( \dual{ ( \dual{b} \otimes{}  ( \dual{a} \parr{} a )  ) } \otimes{} \dual{b} )  \parr{}  (  (  ( \dual{a} \parr{} a )  \otimes{} \dual{b} )  \parr{} b )  )  )  \parr{}  (  ( \dual{a} \parr{}  ( a \otimes{} \dual{b} )  )  \parr{} b )  ) \)}
\tdemorL{}\UnaryInfC{\( (  ( \dual{ (  (  ( \dual{a} \parr{} a )  \otimes{} \dual{b} )  \parr{} b ) } \otimes{}  ( \dual{ (  ( \dual{b} \otimes{}  ( \dual{a} \parr{} a )  )  \parr{} b ) } \parr{}  (  (  ( \dual{a} \parr{} a )  \otimes{} \dual{b} )  \parr{} b )  )  )  \parr{}  (  ( \dual{a} \parr{}  ( a \otimes{} \dual{b} )  )  \parr{} b )  ) \)}
\taiDownL{}\UnaryInfC{\( (  ( \dual{ (  (  ( \dual{a} \parr{} a )  \otimes{} \dual{b} )  \parr{} b ) } \otimes{}  (  ( \dual{ (  ( \dual{b} \otimes{}  ( \dual{a} \parr{} a )  )  \parr{} b ) } \otimes{}  ( \dual{b} \parr{} b )  )  \parr{}  (  (  ( \dual{a} \parr{} a )  \otimes{} \dual{b} )  \parr{} b )  )  )  \parr{}  (  ( \dual{a} \parr{}  ( a \otimes{} \dual{b} )  )  \parr{} b )  ) \)}
\taiDownL{}\UnaryInfC{\( (  ( \dual{ (  (  ( \dual{a} \parr{} a )  \otimes{} \dual{b} )  \parr{} b ) } \otimes{}  (  ( \dual{ (  ( \dual{b} \otimes{}  ( \dual{a} \parr{} a )  )  \parr{} b ) } \otimes{}  (  ( \dual{b} \otimes{}  ( \dual{a} \parr{} a )  )  \parr{} b )  )  \parr{}  (  (  ( \dual{a} \parr{} a )  \otimes{} \dual{b} )  \parr{} b )  )  )  \parr{}  (  ( \dual{a} \parr{}  ( a \otimes{} \dual{b} )  )  \parr{} b )  ) \)}
\taiDownL{}\UnaryInfC{\( (  ( \dual{ (  (  ( \dual{a} \parr{} a )  \otimes{} \dual{b} )  \parr{} b ) } \otimes{}  (  ( \dual{ (  ( \dual{b} \otimes{}  ( \dual{a} \parr{} a )  )  \parr{} b ) } \otimes{}  (  ( \dual{b} \otimes{}  (  ( \dual{a} \parr{} a )  \otimes{}  ( b \parr{} \dual{b} )  )  )  \parr{} b )  )  \parr{}  (  (  ( \dual{a} \parr{} a )  \otimes{} \dual{b} )  \parr{} b )  )  )  \parr{}  (  ( \dual{a} \parr{}  ( a \otimes{} \dual{b} )  )  \parr{} b )  ) \)}
\tsigmaDownL{}\UnaryInfC{\( (  ( \dual{ (  (  ( \dual{a} \parr{} a )  \otimes{} \dual{b} )  \parr{} b ) } \otimes{}  (  ( \dual{ (  ( \dual{b} \otimes{}  ( \dual{a} \parr{} a )  )  \parr{} b ) } \otimes{}  (  ( \dual{b} \otimes{}  (  ( b \parr{} \dual{b} )  \otimes{}  ( \dual{a} \parr{} a )  )  )  \parr{} b )  )  \parr{}  (  (  ( \dual{a} \parr{} a )  \otimes{} \dual{b} )  \parr{} b )  )  )  \parr{}  (  ( \dual{a} \parr{}  ( a \otimes{} \dual{b} )  )  \parr{} b )  ) \)}
\tsigmaSwitchL{}\UnaryInfC{\( (  ( \dual{ (  (  ( \dual{a} \parr{} a )  \otimes{} \dual{b} )  \parr{} b ) } \otimes{}  (  ( \dual{ (  ( \dual{b} \otimes{}  ( \dual{a} \parr{} a )  )  \parr{} b ) } \otimes{}  (  ( \dual{b} \otimes{}  ( b \parr{}  ( \dual{b} \otimes{}  ( \dual{a} \parr{} a )  )  )  )  \parr{} b )  )  \parr{}  (  (  ( \dual{a} \parr{} a )  \otimes{} \dual{b} )  \parr{} b )  )  )  \parr{}  (  ( \dual{a} \parr{}  ( a \otimes{} \dual{b} )  )  \parr{} b )  ) \)}
\tsigmaalphaL{}\UnaryInfC{\( (  ( \dual{ (  (  ( \dual{a} \parr{} a )  \otimes{} \dual{b} )  \parr{} b ) } \otimes{}  (  ( \dual{ (  ( \dual{b} \otimes{}  ( \dual{a} \parr{} a )  )  \parr{} b ) } \otimes{}  (  (  (  ( \dual{b} \otimes{}  ( \dual{a} \parr{} a )  )  \parr{} b )  \otimes{} \dual{b} )  \parr{} b )  )  \parr{}  (  (  ( \dual{a} \parr{} a )  \otimes{} \dual{b} )  \parr{} b )  )  )  \parr{}  (  ( \dual{a} \parr{}  ( a \otimes{} \dual{b} )  )  \parr{} b )  ) \)}
\tsigmaSwitchL{}\UnaryInfC{\( (  ( \dual{ (  (  ( \dual{a} \parr{} a )  \otimes{} \dual{b} )  \parr{} b ) } \otimes{}  (  ( \dual{ (  ( \dual{b} \otimes{}  ( \dual{a} \parr{} a )  )  \parr{} b ) } \otimes{}  (  (  ( \dual{b} \otimes{}  ( \dual{a} \parr{} a )  )  \parr{}  ( b \otimes{} \dual{b} )  )  \parr{} b )  )  \parr{}  (  (  ( \dual{a} \parr{} a )  \otimes{} \dual{b} )  \parr{} b )  )  )  \parr{}  (  ( \dual{a} \parr{}  ( a \otimes{} \dual{b} )  )  \parr{} b )  ) \)}
\tsigmaalphaL{}\UnaryInfC{\( (  ( \dual{ (  (  ( \dual{a} \parr{} a )  \otimes{} \dual{b} )  \parr{} b ) } \otimes{}  (  ( \dual{ (  ( \dual{b} \otimes{}  ( \dual{a} \parr{} a )  )  \parr{} b ) } \otimes{}  (  ( b \otimes{} \dual{b} )  \parr{}  (  ( \dual{b} \otimes{}  ( \dual{a} \parr{} a )  )  \parr{} b )  )  )  \parr{}  (  (  ( \dual{a} \parr{} a )  \otimes{} \dual{b} )  \parr{} b )  )  )  \parr{}  (  ( \dual{a} \parr{}  ( a \otimes{} \dual{b} )  )  \parr{} b )  ) \)}
\tdemorL{}\UnaryInfC{\( (  ( \dual{ (  (  ( \dual{a} \parr{} a )  \otimes{} \dual{b} )  \parr{} b ) } \otimes{}  (  ( \dual{ (  ( \dual{b} \otimes{}  ( \dual{a} \parr{} a )  )  \parr{} b ) } \otimes{}  ( \dual{ ( \dual{b} \parr{} b ) } \parr{}  (  ( \dual{b} \otimes{}  ( \dual{a} \parr{} a )  )  \parr{} b )  )  )  \parr{}  (  (  ( \dual{a} \parr{} a )  \otimes{} \dual{b} )  \parr{} b )  )  )  \parr{}  (  ( \dual{a} \parr{}  ( a \otimes{} \dual{b} )  )  \parr{} b )  ) \)}
\taiDownL{}\UnaryInfC{\( (  ( \dual{ (  (  ( \dual{a} \parr{} a )  \otimes{} \dual{b} )  \parr{} b ) } \otimes{}  (  ( \dual{ (  ( \dual{b} \otimes{}  ( \dual{a} \parr{} a )  )  \parr{} b ) } \otimes{}  (  ( \dual{ ( \dual{b} \parr{} b ) } \otimes{}  ( \dual{b} \parr{} b )  )  \parr{}  (  ( \dual{b} \otimes{}  ( \dual{a} \parr{} a )  )  \parr{} b )  )  )  \parr{}  (  (  ( \dual{a} \parr{} a )  \otimes{} \dual{b} )  \parr{} b )  )  )  \parr{}  (  ( \dual{a} \parr{}  ( a \otimes{} \dual{b} )  )  \parr{} b )  ) \)}
\tsigmaDownL{}\UnaryInfC{\( (  ( \dual{ (  (  ( \dual{a} \parr{} a )  \otimes{} \dual{b} )  \parr{} b ) } \otimes{}  (  ( \dual{ (  ( \dual{b} \otimes{}  ( \dual{a} \parr{} a )  )  \parr{} b ) } \otimes{}  (  (  ( \dual{b} \parr{} b )  \otimes{} \dual{ ( \dual{b} \parr{} b ) } )  \parr{}  (  ( \dual{b} \otimes{}  ( \dual{a} \parr{} a )  )  \parr{} b )  )  )  \parr{}  (  (  ( \dual{a} \parr{} a )  \otimes{} \dual{b} )  \parr{} b )  )  )  \parr{}  (  ( \dual{a} \parr{}  ( a \otimes{} \dual{b} )  )  \parr{} b )  ) \)}
\taiUpL{}\UnaryInfC{\( (  ( \dual{ (  (  ( \dual{a} \parr{} a )  \otimes{} \dual{b} )  \parr{} b ) } \otimes{}  (  ( \dual{ (  ( \dual{b} \otimes{}  ( \dual{a} \parr{} a )  )  \parr{} b ) } \otimes{}  (  ( \dual{b} \otimes{}  ( \dual{a} \parr{} a )  )  \parr{} b )  )  \parr{}  (  (  ( \dual{a} \parr{} a )  \otimes{} \dual{b} )  \parr{} b )  )  )  \parr{}  (  ( \dual{a} \parr{}  ( a \otimes{} \dual{b} )  )  \parr{} b )  ) \)}
\tsigmaDownL{}\UnaryInfC{\( (  ( \dual{ (  (  ( \dual{a} \parr{} a )  \otimes{} \dual{b} )  \parr{} b ) } \otimes{}  (  (  (  ( \dual{b} \otimes{}  ( \dual{a} \parr{} a )  )  \parr{} b )  \otimes{} \dual{ (  ( \dual{b} \otimes{}  ( \dual{a} \parr{} a )  )  \parr{} b ) } )  \parr{}  (  (  ( \dual{a} \parr{} a )  \otimes{} \dual{b} )  \parr{} b )  )  )  \parr{}  (  ( \dual{a} \parr{}  ( a \otimes{} \dual{b} )  )  \parr{} b )  ) \)}
\taiUpL{}\UnaryInfC{\( (  ( \dual{ (  (  ( \dual{a} \parr{} a )  \otimes{} \dual{b} )  \parr{} b ) } \otimes{}  (  (  ( \dual{a} \parr{} a )  \otimes{} \dual{b} )  \parr{} b )  )  \parr{}  (  ( \dual{a} \parr{}  ( a \otimes{} \dual{b} )  )  \parr{} b )  ) \)}
\tsigmaDownL{}\UnaryInfC{\( (  (  (  (  ( \dual{a} \parr{} a )  \otimes{} \dual{b} )  \parr{} b )  \otimes{} \dual{ (  (  ( \dual{a} \parr{} a )  \otimes{} \dual{b} )  \parr{} b ) } )  \parr{}  (  ( \dual{a} \parr{}  ( a \otimes{} \dual{b} )  )  \parr{} b )  ) \)}
\taiUpL{}\UnaryInfC{\( (  ( \dual{a} \parr{}  ( a \otimes{} \dual{b} )  )  \parr{} b ) \)}
\end{prooftree}
\caption[Transformation of derivation example]{%
  \label{fig:sc2di}Translation $\transscdi$ applied to the derivation
  in \protect\cref{fig:di2sc} back into $\protect\mllm{DI}$ (symmetry
  and associativity rules contracted and labelled
  $\protect\sigma,\alpha$)}
\end{figure}

We wish to analyse the relative complexities of the proofs before and after a
transformation is used and we therefore introduce the following notation.
We write, for example, \(\proofsize{\proofcut{\deriv{}}}\) for the number of \cutT{}
rules in the proof \(\deriv{}\), or, where the relevant proof is clear, we just write
\(\proofsize{\proofcut{}}\).

Most of the complexity in the transformation \(\transscdi\) comes from the syntactic
associativity and symmetry rules \(\alpha, \sigma\). It is easy to see from the
proof above that
\begin{align*}
  \proofsize{\proofid{\deriv{}}} &= \proofsize{\proofiDown{\transscdi\deriv{}}},
  \\
  \proofsize{\proofax{\deriv{}}} &= \proofsize{\proofaiDown{\transscdi\deriv{}}},
  \\
  \proofsize{\proofcut{\deriv{}}} &= \proofsize{\proofiUp{\transscdi\deriv{}}}
\end{align*}
and
\[ \proofsize{\proofcut{\deriv{}}} + \proofsize{\proofotimes{\deriv{}}} = \proofsize{\proofswitch{\transscdi\deriv{}}}. \]
It is trivial to show that for any $\mllm{SC}$ proof
\[ \proofsize{\proofcut{\deriv{}}} + \proofsize{\proofotimes{\deriv{}}} =
\proofsize{\proofax{\deriv{}}} + \proofsize{\proofid{\deriv{}}} - 1 \]
so the number of switch rules in a translated proof additionally obeys the following equality:
\[
  \proofsize{\proofaiDown{\transscdi\deriv{}}} +
  \proofsize{\proofiDown{\transscdi\deriv{}}} - 1 =
  \proofsize{\proofswitch{\transscdi\deriv{}}}.
\]

Now for the transformation \(\transdisc\) it is clear that these
equalities above become rather large inequalities.
Furthermore, consider the sequent calculus derivation in
\cref{fig:di2sc}. The deep inference derivation resulting from the
above transformation is given in \cref{fig:sc2di}.
We use this as justification that
\(\transdisc\) greatly increases the complexity of proofs it transforms.

\subsection{An alternative translation}

The translation $\transdisc$ in particular makes use of the cut rule
in a way that introduces a large number of cuts into the resulting
derivation. It turns out that most of these are unnecessary and we
provide an alternative translation $\transdiscalt$ here.

\begin{proposition}\label{ditosc_direct}
  If $\deriv{}$ is a derivation of a formula $A$ in $\mllm{DI}$ then
  there is a derivation $\transdiscalt\deriv{}$ of ${\ttile A}$ in $\mllm{SC}$.
\end{proposition}

\begin{proof}
  \newcommand{\SCRSIZE}{0.5}
  We again describe a recursive procedure to define the translation.
  To begin, use the procedure in \cref{diatomsonly}
  to convert \iDownT{} to \aiDownT{}. The transformation is then explained below.

  For the base case, use \axT{} followed by \parT{} for \aiDownT{}.

  Given a deep inference derivation $\deriv{}$ that transforms into a sequent
  calculus derivation $\transdiscalt\deriv{}$, consider what the
  resulting transformed proof $\transdiscalt\deriv{}$ looks like when the following rule is
  inductively appended to $\deriv{}$:

  \begin{enumerate}

    \item \aiDownT{}:
      It is easy to see from the base case and by hypothesis that
      $\transdiscalt\deriv{}$ has
      the form:
      \begin{prooftree}
        \tiny
        \arbitrarySCAxioms[\SCRSIZE]{$\deriv{0}$}
        \UnaryInfC{$\ttile \Gamma, B, \Gamma'$}
        \arbitrarySCRules[\SCRSIZE]{$\deriv{1}$}
        \UnaryInfC{$\ttile S\{B\}$}
      \end{prooftree}
      where there exists a location in the proof where B is a distinct element
      of a sequent. This can be done by traversing up the proof tree from
      $\ttile S\{B\}$. It should be relatively easy to see that this property is
      true in the base case because we use the procedure to remove \iDownT{},
      and this property is maintained throughout the proof.

      Transform $\transdiscalt\deriv{}$ into:
      \begin{prooftree}
        \tiny
        \arbitrarySCAxioms[\SCRSIZE]{$\deriv{0}$}
        \UnaryInfC{$\ttile \Gamma, B, \Gamma'$}
        \arbitrarySCRules[\SCRSIZE]{$\deriv{e}$}
        \UnaryInfC{$\ttile \Gamma, \Gamma', B$}
        \AxiomC{}
        \taxL{}\UnaryInfC{$\ttile \dual{a}, a$}
        \tparL{}\UnaryInfC{$\ttile \dual{a} \parr a$}
        \totimesL{}\BinaryInfC{$\ttile \Gamma, \Gamma', B \otimes (\dual{a} \parr a)$}
        \arbitrarySCRules[\SCRSIZE]{$\deriv{e}$}
        \UnaryInfC{$\ttile \Gamma, B \otimes (\dual{a} \parr a), \Gamma'$}
        \arbitrarySCRules[\SCRSIZE]{$\deriv{1}$}
        \UnaryInfC{$\ttile S\{B \otimes (\dual{a} \parr a)\}$}
      \end{prooftree}
      where the $\deriv{e}$ append as many \exchT{} rules as appropriate as many
      times as there are formulae in $\Gamma'$.

      One might think that altering the premise to $\deriv{1}$ may lead to an
      invalid proof when the \cutT{} rule is used. However, because we were able
      to find the location where $B$ was its own distinct element, we know that
      it was not cut out.

    \item \dontbreak{ \switchT{}:
      \begin{center}\begin{tikzcd}[column sep = large, row sep = large]
        \tiny
        \begin{bprooftree}
          \arbitrarySCAxioms[\SCRSIZE]{$\deriv{0A}$}
          \UnaryInfC{$\ttile \Gamma, A$}
          \arbitrarySCAxioms[\SCRSIZE]{$\deriv{0B}$}
          \UnaryInfC{$\ttile \Gamma'_0, B, C, \Gamma'_1$}
          \tparL{}\UnaryInfC{$\ttile \Gamma'_0, B \parr C, \Gamma'_1$}
          \arbitrarySCRules[\SCRSIZE]{$\deriv{1}$}
          \UnaryInfC{$\ttile B \parr C, \Gamma'$}
          \totimesL{}\BinaryInfC{$\ttile \Gamma, A \otimes (B \parr C), \Gamma'$}
          \arbitrarySCRules[\SCRSIZE]{$\deriv{2}$}
          \UnaryInfC{$\ttile S\{A \otimes (B \parr C)\}$}
        \end{bprooftree}
      \arrow[r, "\text{becomes}"] &
        \tiny
        \begin{bprooftree}
          \arbitrarySCAxioms[\SCRSIZE]{$\deriv{0A}$}
          \UnaryInfC{$\ttile \Gamma, A$}
          \arbitrarySCAxioms[\SCRSIZE]{$\deriv{0B}$}
          \UnaryInfC{$\ttile \Gamma'_0, B, C, \Gamma'_1$}
          \arbitrarySCRules[\SCRSIZE]{$\deriv{1}'$}
          \UnaryInfC{$\ttile B, C, \Gamma'$}
          \totimesL{}\BinaryInfC{$\ttile \Gamma, A \otimes B, C, \Gamma'$}
          \tparL{}\UnaryInfC{$\ttile \Gamma, (A \otimes B) \parr C, \Gamma'$}
          \arbitrarySCRules[\SCRSIZE]{$\deriv{2}$}
          \UnaryInfC{$\ttile S\{(A \otimes B) \parr C\}$}
        \end{bprooftree}
      \end{tikzcd}\end{center}
      }
      where $\deriv{1}'$ is $\deriv{1}$ with \exchT{} rules on $B \parr C$
      becoming two exchange rules on $B$ and $C$ respectively.
      Notice that, as one would expect, the number of cut and tensor rules do
      not change when we transform a switch rule.

    \item \dontbreak{ \sigmaUpT{}:
      \begin{center}\begin{tikzcd}[column sep = large, row sep = large]
        \tiny
        \begin{bprooftree}
          \arbitrarySCAxioms[\SCRSIZE]{$\deriv{0}$}
          \UnaryInfC{$\ttile \Gamma, A, B, \Gamma'$}
          \tparL{}\UnaryInfC{$\ttile \Gamma, A \parr B, \Gamma'$}
          \arbitrarySCRules[\SCRSIZE]{$\deriv{1}$}
          \UnaryInfC{$\ttile S\{A \parr B\}$}
        \end{bprooftree}
      \arrow[r, "\text{becomes}"] &
        \tiny
        \begin{bprooftree}
          \arbitrarySCAxioms[\SCRSIZE]{$\deriv{0}$}
          \UnaryInfC{$\ttile \Gamma, A, B, \Gamma'$}
          \texchL{}\UnaryInfC{$\ttile \Gamma, B, A, \Gamma'$}
          \tparL{}\UnaryInfC{$\ttile \Gamma, B \parr A, \Gamma'$}
          \arbitrarySCRules[\SCRSIZE]{$\deriv{1}$}
          \UnaryInfC{$\ttile S\{B \parr A\}$}
        \end{bprooftree}
      \end{tikzcd}\end{center}
    }

    \item \dontbreak{ \sigmaDownT{}:
      \begin{center}\begin{tikzcd}[column sep = large, row sep = large]
        \tiny
        \begin{bprooftree}
          \arbitrarySCAxioms[\SCRSIZE]{$\deriv{0A}$}
          \UnaryInfC{$\ttile \Gamma, A$}
          \arbitrarySCAxioms[\SCRSIZE]{$\deriv{0B}$}
          \UnaryInfC{$\ttile B, \Gamma'$}
          \totimesL{}\BinaryInfC{$\ttile \Gamma, A \otimes B, \Gamma'$}
          \arbitrarySCRules[\SCRSIZE]{$\deriv{1}$}
          \UnaryInfC{$\ttile S\{A \otimes B\}$}
        \end{bprooftree}
      \arrow[r, "\text{becomes}"] &
        \tiny
        \begin{bprooftree}
          \arbitrarySCAxioms[\SCRSIZE]{$\deriv{0B}$}
          \UnaryInfC{$\ttile B, \Gamma'$}
          \arbitrarySCRules[\SCRSIZE]{$\deriv{e}$}
          \UnaryInfC{$\ttile \Gamma', B$}
          \arbitrarySCAxioms[\SCRSIZE]{$\deriv{0A}$}
          \UnaryInfC{$\ttile \Gamma, A$}
          \arbitrarySCRules[\SCRSIZE]{$\deriv{e}$}
          \UnaryInfC{$\ttile A, \Gamma$}
          \totimesL{}\BinaryInfC{$\ttile \Gamma', B \otimes A, \Gamma$}
          \arbitrarySCRules[\SCRSIZE]{$\deriv{e}$}
          \UnaryInfC{$\ttile \Gamma, B \otimes A, \Gamma'$}
          \arbitrarySCRules[\SCRSIZE]{$\deriv{1}$}
          \UnaryInfC{$\ttile S\{B \otimes A\}$}
        \end{bprooftree}
      \end{tikzcd}\end{center}
    }

    \item \dontbreak{ \alphaUpT{}:
      \begin{center}\begin{tikzcd}[column sep = large, row sep = large]
        \tiny
        \begin{bprooftree}
          \arbitrarySCAxioms[\SCRSIZE]{$\deriv{0}$}
          \UnaryInfC{$\ttile \Gamma_B, B, C, \Gamma_C$}
          \tparL{}\UnaryInfC{$\ttile \Gamma_B, B \parr C, \Gamma_C$}
          \arbitrarySCRules[\SCRSIZE]{$\deriv{1}$}
          \UnaryInfC{$\ttile \Gamma, A, B \parr C, \Gamma'$}
          \tparL{}\UnaryInfC{$\ttile \Gamma, A \parr (B \parr C), \Gamma'$}
          \arbitrarySCRules[\SCRSIZE]{$\deriv{2}$}
          \UnaryInfC{$\ttile S\{A \parr (B \parr C)\}$}
        \end{bprooftree}
      \arrow[r, "\text{becomes}"] &
        \tiny
        \begin{bprooftree}
          \arbitrarySCAxioms[\SCRSIZE]{$\deriv{0}$}
          \UnaryInfC{$\ttile \Gamma_B, B, C, \Gamma_C$}
          \arbitrarySCRules[\SCRSIZE]{$\deriv{1}'$}
          \UnaryInfC{$\ttile \Gamma, A, B, C, \Gamma'$}
          \tparL{}\UnaryInfC{$\ttile \Gamma, A \parr B, C, \Gamma'$}
          \tparL{}\UnaryInfC{$\ttile \Gamma, (A \parr B) \parr C, \Gamma'$}
          \arbitrarySCRules[\SCRSIZE]{$\deriv{2}$}
          \UnaryInfC{$\ttile S\{(A \parr B) \parr C\}$}
        \end{bprooftree}
      \end{tikzcd}\end{center}
    }
      where $\deriv{1}'$ is $\deriv{1}$ with \exchT{} rules on $B \parr C$
      becoming two exchange rules respectively on $B$ and $C$.

    \item \dontbreak{ \alphaDownT{}:
      \begin{center}\begin{tikzcd}[column sep = large, row sep = large]
        \tiny
        \begin{bprooftree}
          \arbitrarySCAxioms[\SCRSIZE]{$\deriv{0A}$}
          \UnaryInfC{$\ttile \Gamma_A, A$}
          \arbitrarySCAxioms[\SCRSIZE]{$\deriv{0B}$}
          \UnaryInfC{$\ttile B, \Gamma_B$}
          \totimesL{}\BinaryInfC{$\ttile \Gamma_A, A \otimes B, \Gamma_B$}
          \arbitrarySCRules[\SCRSIZE]{$\deriv{1}$}
          \UnaryInfC{$\ttile \Gamma, A \otimes B$}
          \arbitrarySCAxioms[\SCRSIZE]{$\deriv{0C}$}
          \UnaryInfC{$\ttile C, \Gamma'$}
          \totimesL{}\BinaryInfC{$\ttile \Gamma, (A \otimes B) \otimes C, \Gamma'$}
          \arbitrarySCRules[\SCRSIZE]{$\deriv{2}$}
          \UnaryInfC{$\ttile S\{(A \otimes B) \otimes C\}$}
        \end{bprooftree}
      \arrow[r, "\text{becomes}"] &%
        \tiny
        \begin{bprooftree}
          \arbitrarySCAxioms[\SCRSIZE]{$\deriv{0A}$}
          \UnaryInfC{$\ttile \Gamma_A, A$}
          \arbitrarySCAxioms[\SCRSIZE]{$\deriv{0B}$}
          \UnaryInfC{$\ttile B, \Gamma_B$}
          \arbitrarySCRules[\SCRSIZE]{$\deriv{e}$}
          \UnaryInfC{$\ttile \Gamma_B, B$}
          \arbitrarySCAxioms[\SCRSIZE]{$\deriv{0C}$}
          \UnaryInfC{$\ttile C, \Gamma'$}
          \totimesL{}\BinaryInfC{$\ttile \Gamma_B, B \otimes C, \Gamma'$}
          \arbitrarySCRules[\SCRSIZE]{$\deriv{e}$}
          \UnaryInfC{$\ttile B \otimes C, \Gamma', \Gamma_B$}
          \totimesL{}\BinaryInfC{$\ttile \Gamma_A, A \otimes (B \otimes C), \Gamma', \Gamma_B$}
          \arbitrarySCRules[\SCRSIZE]{$\deriv{1}'$}
          \UnaryInfC{$\ttile \Gamma, A \otimes (B \otimes C), \Gamma'$}
          \arbitrarySCRules[\SCRSIZE]{$\deriv{2}$}
          \UnaryInfC{$\ttile S\{A \otimes (B \otimes C)\}$}
        \end{bprooftree}
      \end{tikzcd}\end{center}
      }
      where $\deriv{1}'$ is $\deriv{1}$ with \exchT{} rules on $A \otimes B$
      becoming as many exchange rules as necessary on $A \otimes (B \otimes C)$ and
      $\Gamma'$.

    \item \dontbreak{ \iUpT{}:
      \begin{center}\begin{tikzcd}[column sep = large, row sep = large]
        \tiny
        \begin{bprooftree}
          \arbitrarySCAxioms[\SCRSIZE]{$\deriv{0A}$}
          \UnaryInfC{$\ttile \Gamma_A, A$}
          \arbitrarySCAxioms[\SCRSIZE]{$\deriv{0\dual{A}}$}
          \UnaryInfC{$\ttile \dual{A}, \Gamma_{\dual{A}}$}
          \totimesL{}\BinaryInfC{$\ttile \Gamma_A, A \otimes \dual{A}, \Gamma_{\dual{A}}$}
          \arbitrarySCRules[\SCRSIZE]{$\deriv{1}$}
          \UnaryInfC{$\ttile \Gamma, A \otimes \dual{A}, B, \Gamma'$}
          \tparL{}\UnaryInfC{$\ttile \Gamma, (A \otimes \dual{A}) \parr B, \Gamma'$}
          \arbitrarySCRules[\SCRSIZE]{$\deriv{2}$}
          \UnaryInfC{$\ttile S\{(A \otimes \dual{A}) \parr B\}$}
        \end{bprooftree}
      \arrow[r, "\text{becomes}"] &
        \tiny
        \begin{bprooftree}
          \arbitrarySCAxioms[\SCRSIZE]{$\deriv{0A}$}
          \UnaryInfC{$\ttile \Gamma_A, A$}
          \arbitrarySCAxioms[\SCRSIZE]{$\deriv{0\dual{A}}$}
          \UnaryInfC{$\ttile \dual{A}, \Gamma_{\dual{A}}$}
          \tcutL{}\BinaryInfC{$\ttile \Gamma_A, \Gamma_{\dual{A}}$}
          \arbitrarySCRules[\SCRSIZE]{$\deriv{1}'$}
          \UnaryInfC{$\ttile \Gamma, B, \Gamma'$}
          \arbitrarySCRules[\SCRSIZE]{$\deriv{2}$}
          \UnaryInfC{$\ttile S\{B\}$}
        \end{bprooftree}
      \end{tikzcd}\end{center}
      }
      where $\deriv{1}'$ is $\deriv{1}$ with any \exchT{} rules affecting
      $A \otimes \dual{A}$ omitted.
      The only rules that can affect $A \otimes \dual{A}$ is \exchT{}.
      There cannot be a \cutT{} rule used on this formula that replaces
      it with an identical one, or with the atoms required to construct
      a new $A \otimes \dual{A}$ or else there would be another choice
      for the location of $\deriv{1}$ where the formula is not cut out.

      Notice that this transformation means that there is a bijection from the \iUpT{} rules to
      the corresponding \cutT{} rules.
      Additionally, this means there is never a point where the choice of the
      location of any $\ttile \Gamma, A, \Gamma'$ in the proof tree is ambiguous.
      This is what maintains the correctness of this transformation.
  \end{enumerate}
\end{proof}

It should be clear that $\transdiscalt$ has the following equalities:
\begin{align*}
  \proofsize{\proofaiDown{\deriv{}}} &=
  \proofsize{\proofax{\transdiscalt\deriv{}}},
  \\
  \proofsize{\proofiUp{\deriv{}}} &= \proofsize{\proofcut{\transdiscalt\deriv{}}} \text{ and}
  \\
  \proofsize{\transdiscalt\proofax{\deriv{}}} - 1 =
  \proofsize{\proofaiDown{\deriv{}}} - 1 &=
  \proofsize{\proofotimes{\transdiscalt\deriv{}}} + \proofsize{\proofcut{\transdiscalt\deriv{}}}
  .
\end{align*}
Therefore, $\transdiscalt$ is a more faithful transformation than $\transdisc$
as it preserves the number of axiom and cut rules.
Note that the transformations are still not inverses. We can arbitrarily
construct a deep inference proof where the number of switch rules is zero
while the number of axiomatic rules are large. Transforming the proof
to sequent calculus and back would result in a large number of switch rules.
However, proofs that start in sequent calculus, or deep inference proofs
that were translated from sequent calculus first, are in a kind of
``normal form'' where the only changes are the bureaucratic \exchT{},
\alphaDownT{}, \alphaUpT{},
\sigmaDownT{} and \sigmaUpT{} rules.
So it could informally be said that $\transscdi{}$ and the composition
${\transdiscalt{} \after \transscdi{} \after \transdiscalt{}}$
are ``almost inverse''.

In sequent calculus, formulae are internally immutable and can only be
externally combined with other formulae to create larger ones.
In computer science terms, the formulae and proof tree are persistent data
structures, and the inference rules are akin to pure functions.
Meanwhile, deep inference works with a single internally-mutable formula,
corresponding to how functions work in ``normal'' imperative programming languages.
This appears to be why the $\mllm{SC}$ to $\mllm{DI}$ transformation proof is
shorter and easier to reason about.
Considering how a sequent calculus derivation tree can be mutated into another
one is what allows us to see the connection between the two systems.

\subsection{Cut Elimination}

We provide one final transformation for a cut elimination procedure in sequent
calculus called $\transcut{}$. We do not go into full detail here as it is
rather laborious; full details for this procedure for full linear logic can be
found in \cite{pfenningcut_elim, mellies}, while a sketch specific to
$\mllm{SC}$ can be found in \cite{proof_nets}.

\begin{proposition}\label{equiv_cut_elimination}
  If $\deriv{}$ is a derivation that contains the \cutT{} rule in
  $\mllm{SC}$ then so is $\transcut\deriv{}$.
\end{proposition}
\begin{proof}
  \newcommand{\AxiomSize}{0.25}
  \newcommand{\AX}{\arbitrarySCAxioms[\AxiomSize]{}}
  (Sketch)
  \\
  Cut elimination is the process of continuously moving a \cutT{} rule upwards
  in a proof until the \cutT{} disappears. For $\mllm{SC}$ there are five main cases to cover.
  Without loss of generality we only consider the left parent of the \cutT{} rule.
  For simplicity in this sketch we ignore any \exchT{} rules in the parents,
  as they will be ignored anyway in our categorical model.
  Additionally, following the \cutT{} rule we abbreviate multiple \exchT{} rules into a single \exchT{}.
  Consider the cases where the left parent rule ignoring any \exchT{} rules is:
  \begin{itemize}
    \item \dontbreak{ \axT{} or \idT{}:
      \\
      \begin{center} \begin{tikzcd}[column sep = large, row sep = large]
      \tiny
      \begin{bprooftree}
        \AxiomC{}
          \tidL{}\UnaryInfC{$\ttile \dual{A}, A$}
          \AX\UnaryInfC{$\ttile \dual{A}, \Gamma$}
        \tcutL\BinaryInfC{$\ttile \dual{A}, \Gamma$}
      \end{bprooftree}
        \arrow[r, "\text{becomes}"] &
      \tiny
      \begin{bprooftree}
        \AX\UnaryInfC{$\ttile \dual{A}, \Gamma$}
      \end{bprooftree}
      \end{tikzcd} \end{center}
    }
    \item \dontbreak { \parT{} on the formula we do not cut:
      \\
      \begin{center} \begin{tikzcd}[column sep = large, row sep = large]
      \tiny
      \begin{bprooftree}
        \AX
          \UnaryInfC{$\ttile \Gamma'_0, A$}
          \tparL{}\UnaryInfC{$\ttile \Gamma_0, A$}
        \AX
          \UnaryInfC{$\ttile \dual{A}, \Gamma_1$}
        \tcutL{}\BinaryInfC{$\Gamma_0, \Gamma_1$}
      \end{bprooftree}
      \arrow[r, "\text{becomes}"] &
      \tiny
      \begin{bprooftree}
        \AX
          \UnaryInfC{$\Gamma'_0, A$}
        \AX
          \UnaryInfC{$\dual{A}, \Gamma_1$}
        \tcutL{}\BinaryInfC{$\Gamma'_0, \Gamma_1$}
        \tparL{}\UnaryInfC{$\Gamma_0, \Gamma_1$}
      \end{bprooftree}
      \end{tikzcd} \end{center}
    }
    \item \dontbreak{ \otimesT{} on the formula we do not cut:
      \\
      \begin{center} \begin{tikzcd}[column sep = large, row sep = large]
        \tiny
        \begin{bprooftree}
          \AX
            \UnaryInfC{$\ttile \Gamma_0$}
          \AX
            \UnaryInfC{$\ttile \Gamma_1, A$}
          \totimesL{}\BinaryInfC{$\ttile \Gamma_2, A$}
          \AX
            \UnaryInfC{$\ttile \dual{A}, \Gamma_3$}
          \tcutL{}\BinaryInfC{$\ttile \Gamma_2, \Gamma_3$}
        \end{bprooftree}
        \arrow[r, "\text{becomes}"] &
        \tiny
        \begin{bprooftree}
          \AX
            \UnaryInfC{$\ttile \Gamma_0$}
          \AX
            \UnaryInfC{$\ttile \Gamma_1, A$}
          \AX
            \UnaryInfC{$\ttile \dual{A}, \Gamma_3$}
          \tcutL{}\BinaryInfC{$\ttile \Gamma_1, \Gamma_3$}
          \totimesL{}\BinaryInfC{$\ttile \Gamma_2, \Gamma_3$}
        \end{bprooftree}
      \end{tikzcd} \end{center}
    }
    \item \dontbreak{ \parT{} on the formula we cut:
      \\
      \begin{center} \begin{tikzcd}[column sep = large, row sep = large]
        \tiny
        \begin{bprooftree}
          \AX
            \UnaryInfC{$\ttile \Gamma_0, A, B$}
            \tparL{}\UnaryInfC{$\ttile \Gamma_0, A \parr B$}
          \AX
            \UnaryInfC{$\ttile \Gamma_1, \dual{A}$}
          \AX
            \UnaryInfC{$\ttile \dual{B}, \Gamma_2$}
          \totimesL{}\BinaryInfC{$\ttile \Gamma_1, \dual{A} \otimes \dual{B}, \Gamma_2$}
          \arbitrarySCRules[\AxiomSize]{}
          \UnaryInfC{$\ttile \dual{A} \otimes \dual{B}, \Gamma_3$}
          \tcutL{}\BinaryInfC{$\ttile \Gamma_0, \Gamma_3$}
        \end{bprooftree}
        \arrow[r, "\text{becomes}"] &
        \tiny
        \begin{bprooftree}
          \AX
            \UnaryInfC{$\ttile \Gamma_1, \dual{A}$}
          \AX
            \UnaryInfC{$\ttile \Gamma_0, A, B$}
          \AX
            \UnaryInfC{$\ttile \dual{B}, \Gamma_2$}
          \tcutL{}\BinaryInfC{$\ttile \Gamma_0, A, \Gamma_2$}
          \texchL{}\UnaryInfC{$\ttile A, \Gamma_0, \Gamma_2$}
          \tcutL{}\BinaryInfC{$\ttile \Gamma_1, \Gamma_0, \Gamma_2$}
          \arbitrarySCRules[\AxiomSize]{}
          \UnaryInfC{$\ttile \Gamma_0, \Gamma_3$}
        \end{bprooftree}
      \end{tikzcd} \end{center}
    }
    \item \dontbreak{ \otimesT{} on the formula we cut:
      \\
      This is the same as for the \parT{} rule but with negations on $A$ and $B$
      and the branches swapped.
    }
  \end{itemize}
\end{proof}

The main property of proofs after this transformation
resulting from the obvious fact that
\[ \proofsize{\proofcut{\transcut{}\deriv{}}} = 0 \]
is that
\[ \proofsize{\proofax{\transcut{}\deriv{}}} +
\proofsize{\proofid{\transcut{}\deriv{}}} =
\proofsize{\proofotimes{\transcut{}\deriv{}}} + 1. \]

\section{Modelling derivations}\label{models}

We describe how to model derivations as cliques in some coherence
space. We begin by looking at an alternative description of coherence
spaces before giving details for modelling both systems. We conclude
by showing that our interpretations are preserved by the translation
processes defined in the previous section.

\subsection{Concordance spaces}\label{concord}

In \cite{doubleglueing}, in particular Example~5.15~(3), an
alternative description of coherence spaces is proposed. Instead of
presenting coherence spaces via a binary relation we provide an
axiomatization via the sets of cliques and anticliques.

We find it less confusing to introduce a new name for this
axiomatization and then to show that there is a suitable isomorphism
to provide a connection with coherence spaces.

\begin{definition}
  Let $R$ be a set. If $u$ and $x$ are subsets of $R$ we say that
  \textbf{$u$ is orthogonal to $x$}, $u\perp x$, if and only if their
  intersection is at most a singleton, that is
\[r,r'\in u\cap x\qquad\text{implies}\qquad r=r'.\]
\end{definition}

This is an orthogonality relation in the sense of \cite{doubleglueing}.
Given a subset $U$ of $\pow R$ we may generate another such set
as\footnote{We avoid the obvious notation $U^\perp$ for fear of
  confusion with logical negation.}
\[\orth{U}=\{x\subseteq R\suchthat \text{for all }u\in U.\, u\perp x\}.\]
We note that this provides an operation on $\pow R$ which is
order-reversing, that is, if $U\subseteq U'$ then
$\orth{U}\supseteq\orth{U'}$. Applying the operator twice gives a
closure operator on the ordered set~$\pow R$, that is for subsets $U$
and $V$ of $\pow R$ we have
\begin{itemize}
\item $U\subseteq\dorth{U}$,
\item $U\subseteq V$ implies $\dorth{U}\subseteq\dorth{V}$ and
\item $\dorth{\dorth{U}} =\dorth{U}$.
\end{itemize}

We define spaces whose structure is given by a set together with two
subsets of its powerset, linked by the above notion of orthogonality
extended to sets. We show below that this gives us a category that is
isomorphic to the category of coherence spaces as described in
\cref{llmodels}.

\begin{definition}
  A \textbf{concordance space} $(R,U,X)$ consists of a set $R$ and
  subsets $U$, $X$ of the powerset $\pow R$ such that 
\[U=\orth{X}\qquad\text{and}\qquad X=\orth{U}.\]
\end{definition}

Given a coherence space we obtain a concordance space by taking $R$
to be the underlying set, $U$ the set of cliques, and $X$ the set of
anticliques. To see that this results in a concordance space note
that cliques and anticliques intersect in at most one element, and so
we clearly have $U\subseteq\orth{X}$ and $X\subseteq\orth{U}$. To see
that we also have the converse inclusions note that if $u$ is a
subset of $R$ with the property that it intersects every anticlique
in at most one element then it has to be a clique---any two of its
elements which are distinct are therefore not incoherent, and so they
must be coherent. Consequently we use \emph{clique}\/ to refer to the
elements of $U$ and \emph{anticlique}\/ to refer to elements of~$X$.

We aim to model derivations in our systems as cliques,
so having a presentation of coherence spaces which puts cliques at
the centre is helpful.

\begin{proposition}\label{cohprop}
  If $(R,U,X)$ is a concordance space then the following properties
  hold:
  \begin{enumerate}
  \item For all $r\in R$ we have $\{r\}\in U$ and $\{r\}\in X$, that
    is, every singleton is a clique as well as an anticlique.
  \item Both $U$ and $X$ are downward closed as subsets of~$\pow R$,
    that is, for example, if $u\in U$ and $u'\subseteq u$ then $u'\in
    U$. In other words, every subset of a clique is a clique, and the
    corresponding statement holds for anticliques.
  \item We have $U=\dorth{U}$ and $X=\dorth{X}$.
  \item For all $r,r'\in R$ we have $\{r,r'\}\in U$ or $\{r,r'\}\in
    X$, and if $r\not= r'$ this is an exclusive or. This means that
    every two element subset is either a clique or an anticlique.
  \item We have that $u\in U$ if and only if for all $r,r'\in u$ we
    know $\{r,r'\}\in U$, and similarly for~$X$.
  \end{enumerate}
\end{proposition}
\begin{proof}[Proof sketch]
  Clearly a singleton has the property that the intersection with
  every other set is at most a singleton, so for any subset $S$ of
  $\pow R$ we have that every singleton is an element of $\orth{S}$.
  Hence every singleton is an element of both, $U=\orth{X}$ and $X=\orth{U}$.
  Given $u\in U$ and a subset $u'$ of $u$ we can see that the
  intersection of $u'$ with every $x\in X$ is a subset of $u\cap x$
  which is at most a singleton, so $u'\in\orth{X}=U$. It is easy to
  show that applying the operator $\orth{(-)}$ once is the same as
  applying it three times, which implies~(iii). For (iv), if
  $\{r,r'\}$ is not in $U$ then none of its supersets are in $U$
  either by (ii) and thus its intersection with every element
  of $U$ is at most a singleton, so $\{r,r'\}\in\orth{U}=X$. For (v)
  the only if direction follows from (ii), so assume we have a subset
  $u$ of $R$ such that for all $r,r'\in u$ we know $\{r,r'\}\in U$.
  Let $x$ be an arbitrary element of~$X$. Assume that
  $r,r'\in u\cap x$, so $r,r'\in \{r,r'\}\cap x$, and since
  $x\in X=\orth{U}$ this must imply $r=r'$ as required.
\end{proof}

The previous proposition indicates how to take one of our concordance
spaces and translate it into a coherence space: take the relation
that relates $r$ to $r'$ if and only if $\{r,r'\}\in U$. This is a
symmetric reflexive relation by construction. It is not hard to check
that our assignments between concordance and coherence spaces are
mutually inverse.

The morphisms in our categorical model are relations and we need to establish
some notation for these. If we have a relation $f$ from a set $R$ to
a set $S$, then for a subset $u$ of $R$ we use the following notation
for its `direct image' under $f$:
\[[u]f=\{s\in S\suchthat\exists r\in u.\,r\mathrel{f}s\}.\]
Similarly, for a subset $y$ of $S$, we use the following:
\[f[y]=\{r\in R\suchthat\exists s\in y.\, \,r\mathrel{f}s\}\]
for its `inverse image' under~$f$.

Morphisms between concordance spaces have to be well-behaved with
respect to cliques and anticliques. They have to map cliques from the
source to cliques in the target, and anticliques from the target to
anticliques for the source.

\begin{definition}
  A \textbf{morphism  of concordance spaces} from $(R,U,X)$ to
  $(S,V,Y)$ is given by a relation from $R$ to $S$ such that 
  \begin{itemize}
  \item for all $u\in U$ we have $[u]F\in V$ and
  \item for all $y\in Y$ we have $F[y]\in X$.
  \end{itemize}
\end{definition}
This gives us a category $\concat$, with composition and identities
inherited from the category of sets and relations.

If we have a morphism of coherence spaces it is easy to check that it
will map cliques of the source to cliques of the target, and
anticliques of the target to anticliques of the source, and so be a
morphism of the corresponding concordance spaces.

If, on the other hand, we have a morphism of concordance spaces then
it has to map two-element cliques of the source to cliques of the
target, which means it preserves the corresponding coherence
relation. It also has to map two-element anticliques of the target to
anticliques of the source, which means it preserves the incoherence
relation of the corresponding coherence spaces. Hence we have
the following result.

\begin{theorem}
  The categories $\cohcat$ and $\concat$ are isomorphic.
\end{theorem}

Hence the category $\concat$ is also $*$-autonomous, and below we
describe the action of the various connectives.

For this purpose, given $U\subseteq \pow R$ and $V\subseteq\pow S$ we
set
\[U\tensor V=\{u\times v\subseteq R\times S\suchthat u\in U\text{ and }v\in V\}.\]

Assume that $(R,U,X)$ and $(S,V,Y)$ are concordance spaces. Then we
have the following constructions:
\begin{itemize}
\item $\dual{(R,U,X)} = (R,X,U)$.
\item $(R,U,X)\tensor (S,V,Y)=(R\times S, \dorth{(U\tensor
    V)},\orth{(U\tensor V)})$.
\item $(R,U,X)\llpar (S,V,Y)=(R\times S, \orth{(X\tensor
    Y)},\dorth{(X\tensor Y)})$.
\item $(R,U,X)\limp (S,V,Y)=(R\times S, \orth{(U\tensor
    Y)},\dorth{(U\tensor Y)})$.
\end{itemize}
These constructions satisfy the same DeMorgan laws as our formulae,
and they do so as equalities rather than isomorphisms.

We note in particular that we no longer rely on the DeMorgan dual to
define $\llpar$ of two spaces. indeed, the above gives us some idea
how to construct `basic anticliques', which we may then add to by
applying the closure operator $\dorth{(-)}$.

There is only one concordance space on a singleton set, and any such
provides a unit for the monoidal structures given by $\tensor$
or~$\llpar$. We use $I$ for the concordance space
$(\{*\},\pow\{*\},\pow\{*\})$.

In particular the constructions $\tensor$ and $\llpar$ may be
extended to morphisms in such a way that they give symmetric monoidal
structures on~$\concat$. We do this by mapping two relations to their
Cartesian product for both constructions. The structural isomorphisms for symmetry,
associativity and unit are given by the underlying structural
isomorphisms for product on the category of sets and
functions.\footnote{More formally we may state that there is a
  forgetful functor to the category of sets and relations that maps
  the monoidal structure given by $\tensor$ and $\llpar$ to the
  monoidal structure given by Cartesian product, but we do not expect
  the reader to be familiar with the category of sets and relations.}

We give a few helpful results about morphisms without proof. These
are straightforward to establish.

\begin{proposition}
  A relation $F$ from $R$ to $S$ is a morphism from $(R,U,X)$ to
  $(S,V,Y)$ if and only if, viewed as a subset of $R\times S$, it is
  an element of $\orth{(U\tensor Y)}$.
\end{proposition}

\begin{proposition}\label{dualmorphism}
  If $f:(R,U,X) \rightarrow (S,V,Y)$ is a morphism
  then there exists a dual morphism
  $\dual{f}:\dual{(S,V,Y)} = (S,Y,V) \rightarrow (R,X,U) = \dual{(R,U,X)}$.
\end{proposition}

\begin{corollary}
  If $f:(R,U,X) \rightarrow (S,V,Y)$ is a morphism
  then $\dual{\opcat{f}} = \opcat{(\dual{f})}$.
\end{corollary}

\subsection{Deep inference derivations as morphisms}\label{deepmodel}

We explain how to interpret a deep inference derivation as a clique
in a concordance space based on the derived formula.

For these purposes we assume that there is a countable set of atoms,
$\alpha_i$, where $i\in \N$, and that each of these has a dual,
$\dual{\alpha_i}$. These are the primitives of our language. We refer
to the $\alpha_i$ as positive and the $\dual{\alpha_i}$ as negative
as before.

Given a well-formed formula $A$---so negation is applied to atoms
only---we define a concordance space, and we do so \emph{relative}\/
to having picked, for each atom, a concordance space to represent it.
The concordance space for each atom is arbitrary but fixed,
and the properties of our cliques depend on on whether any
concordance space of some atom $\alpha_i$ is equal to or isomorphic to another
instance of the same atom $\alpha_i$ in the same formula or proof.

We may think of such an object as a functor from the terminal
category $\One$ to $\concat$, and we define this recursively.

\begin{itemize}
\item For each $i$ we have a functor $F_{\alpha_i}$ from $\One$ to $\concat$ which
  maps the only object of $\One$ to the concordance space
  representing it.\footnote{This means that our interpretation is
    relative to this valuation.}

\item For each $i$ we also have a functor $F_{\dual{\alpha_i}}$, which is
  constructed as follows:
\[\begin{tikzcd}
\One=\opcat{\One}\arrow{r}{\opcat{F_{\alpha_i}}}&\opcat{\concat}\arrow{r}{\dual{(-)}}
&\concat.
\end{tikzcd}
\]
\item If $F_A$ is the functor corresponding to $A$ and $F_B$ is the
  functor corresponding to $B$ then  the functor
  $F_{A\tensor B}$ corresponding to $A\tensor B$ is the composite
\[
  \begin{tikzcd}
  \One\isom\One\times\One \arrow{r}{F_A\times F_B}
  &\concat\times\concat\arrow{r}{\tensor}&\concat
  \end{tikzcd}
\]
  and the functor $F_{A\llpar B}$ is the composite
\[
  \begin{tikzcd}
  \One\isom\One\times\One \arrow{r}{F_A\times F_B}
  &\concat\times\concat\arrow{r}{\llpar}&\concat.
  \end{tikzcd}
\]
\end{itemize}

For each derivation rule in $\mllm{DI}$ we provide a morphism from
the concordance space interpreting the premise to the concordance space
interpreting the conclusion. The empty formula is interpreted by~$I$.

Hence a derivation is interpreted by a morphism, given by the
composite of morphisms corresponding to the derivation rules used. If
we have a derivation that starts from the empty premise then we may
associate it with a clique, namely the image of the clique $\{*\}$ of
$I$ under that morphism.

We provide a result that allows us to interpret an axiom, by
providing a clique in the space that interprets the $\llpar$ of a
space and its dual.

\begin{lemma}\label{i_clique}
  For a concordance space $A=(R,U,X)$ the following set is a clique in
  $\dual{A}\llpar A$:
\[\Delta_A=\{(r,r)\suchthat r\in R\}.\]
\end{lemma}
\begin{proof}
  We need to show that for all anticliques $x$ and cliques $u$ of $A$
  we have that if $(r,r),(r',r')\in \Delta_A\cap (u\times x)$ 
  then $r=r'$. But for such $r$ and $r'$ we have that $r,r'\in u$ as
  well as $r,r'\in x$, and so $r,r'\in u\cap x$, which means $r=r'$
  by the definition of a concordance space.
\end{proof}

For a concordance space $A=(R,U,X)$ let \fDownT{A} be the
morphism from $I$ to $\dual{A}\llpar A$ which relates $*$ to the pair
$(r,r')$ if and only if $r=r'$. To see that this is indeed a morphism
note that it sends the only non-trivial clique $\{*\}$ on $I$ to
$\Delta_A$. If, on the other hand, we have an anticlique $z$ in
$\dual{A}\llpar A$ then it is mapped to the empty anticlique if it
contains none of the elements of $\Delta_A$, and to the anticlique
$\{*\}$ otherwise. These morphisms interpret the rule~\iDownT{},
depending on the formula generated.

Similarly we may define a morphism \fUpT{A} from some concordance
space $A\tensor\dual{A}$ to $I$ which relates some pair $(r,r')$ to
$*$ if and only if $r=r'$ to interpret the rule~\iUpT{}.

We may think of these cliques as encoding which pairs of atoms have
been created together, and effectively giving the same information as
the axiom links of a proof net.

The remaining rules all involve contexts. In a manner similar to
using a formula to define a functor we may define a functor from a
context.

We define a grammar that generates contexts. Let $A$ be an
arbitrary formula.
\[S\coloneqq \{\,\}%
  \prodor S\tensor A\prodor A\tensor
  S\prodor S\llpar A\prodor A\llpar S.\]

We may interpret each context as a functor from $\concat$ to $\concat$:
\begin{itemize}
\item The empty context $\{\,\,\}$ is interpreted by the identity
  functor on $\concat$.
\item Given the functor $F_S$ which interprets $S$, the functor to
  interpret say, $S\tensor A$, is given by
\[
  \begin{tikzcd}
    \concat\arrow{r}{F_S}&\concat\arrow{r}{-\tensor A} &\concat,
  \end{tikzcd}
\]
and functors for the other construction rules are defined correspondingly.
\end{itemize}

When we `fill the hole' in $S$ with some formula $A$ then, relative
to the given valuation, we obtain a concordance space by applying the
functor $F_S$ to the interpretation of $A$ to obtain the coherence
space that interprets the resulting formula.

Since we have constructed functors it is the case that if we have a
morphism $f$ from the interpretation $\intp{ A}$ of the formula $A$
to the interpretation $\intp{B}$ of the formula $B$ we obtain a
morphism from the interpretation $F_S\intp{A}$ of $S\{A\}$ to the
interpretation $F_S\intp{B}$ of $S\{B\}$ by applying the functor
$F_S$ to~$f$. Hence in order to interpret a non-axiomatic
derivation rule all we need to do is to provide a morphism from the
interpretation of the formula that appears in the hole in the
premise to the interpretation of the formula that appears in that
position in the conclusion.

Since $\llpar$ and $\tensor$ provide monoidal structures on $\concat$
we have structural isomorphisms
  for unit, symmetry and
associativity, and we use the latter two to interpret the
corresponding derivation rules, providing us with morphisms
\fSigmaUpT{A,B}, \fSigmaDownT{A,B}, \fAlphaUpT{A,B,C} and
\fAlphaDownT{A,B,C}.
These isomorphisms are given
  in a very obvious way by bijective functions provided by the set
  theoretic properties of the Cartesian product; for example, the
  associativity isomorphism is that which maps some $(a,(b,c))$ to
  $((a,b),c)$, see also~\cref{lemswitch} where this particular
  bijection is also used to interpret switch.
We find it useful to distinguish these isomorphisms by referring to them as \textbf{\permutations{}}.
 For the non-axiomatic rules for \iDownT{} and
\iUpT{}, observe the results below.

We have the unit structural isomorphism \rhoTensorT{} for the
symmetric monoidal structure given by $\tensor$ which allows us to
form
\[
  \begin{tikzcd}
    B\arrow{r}{\rhoTensor{B}} &B\tensor I\arrow{r}{B\tensor
      \fDown{A}} &B\tensor (\dual{A}\llpar A).
  \end{tikzcd}\]
Let \fiDownT{B, A} be this morphism and applying $F_S$ to \fiDownT{}
allows us to interpret the \iDownT{} rule:
\[\fromto{F_S\intp{B}}{F_S\intp{B\tensor(\dual{A}\llpar A)} =
    F_S(\intp{B}\tensor(\dual{\intp{A}}\llpar\intp{A}))}.\]
Similarly, since $I$ is also the unit for $\llpar$ in $\concat$, we may
use the unit structural isomorphism to obtain
\[
  \begin{tikzcd}
    (A\tensor\dual{A})\llpar B\arrow{r}{\fUp{A}\llpar B}
    & I\llpar B\arrow{r}{\lambdaPar{B}} &B,
  \end{tikzcd}
\]
which gives the desired composite \fiUpT{A, B} that interprets the rule~\iUpT{}.

As we will show, all cliques of the same formula will be equal by these definitions
regardless of whether or not each atom pair come from the same introduction
rules. For example, the clique for the formula
\[ ((\textbf{a} \parr \dual{\textbf{a}}) \otimes (a \parr \dual{a})) \]
will be equal to the clique for
\[ ((\textbf{a} \parr \dual{a}) \otimes (a \parr \dual{\textbf{a}})) \]
where we have highlighted typographically which two atoms came from the same
axiom rule.
If we desire to keep track of where instances of atoms originated together
we can assign different instances of the same axiomatic rules
of the same atoms to cliques that are isomorphic rather than equal.
So for example if there are two \aiDownT{} rules for the same atom $\alpha_i$
then we would assign two isomorphic coherence spaces \({A_i \isom A_i'}\)
that differ only in the labelling of the elements of the sets \({R \in (R, U, X) = A_i, R' \in
(R', U', X') = A_i'.}\)
Then we would need to fix a bijection for all elements in the underlying sets.
Then for \fUpT{A} we would relate a pair \((r, r')\) to \(*\) if and only if
there is a fixed bijection mapping \(r \mapsto r'\) rather than an equality.

\begin{remark}\label{tensor_of_cliques} We make explicit how we get
  from a pair of cliques to a clique in the tensor product of the
  underlying concordance spaces. In general we have for morphisms $f$
  and $g$ that their tensor product has the following effect:
  \[[u\times v](f\tensor g) =[u]f\tensor [v]g\]
  
  Now let $f : I \to X$ be a morphism which maps the clique $\{*\}$
  to $x$ and let $g : I \to Y$ be a morphism mapping $\{*\}$ to $y$.
  Then $f \otimes g : I \otimes I \to X \otimes Y$ maps the clique
  $\{(*, *)\}$ to $x \times y$.
\end{remark}

\begin{lemma}\label{tensor_inside_outside_isomorphism}
  Let $f_B: I \to B$ be some morphism. Then
  the clique resulting from
  $(B \otimes \fDown{A}) \after \rhoTensor{B} \after f_B$ is equal to
  the clique resulting from $f_B \otimes \fDown{A}$.
\end{lemma}
\begin{proof}
  Let $b$ be the clique resulting from~$f_B$.
  For the first composition of morphisms,
    $\rhoTensor{B}$ gives us the clique $b \times \{*\}$.
  Then, $B \otimes \fDown{A}$ results in $b \times \Delta_A$.

  For the second, $\fDown{A}$ still yields us $\Delta_A$. The
  bifunctor $\otimes: \concat\times\concat \to \concat$ maps the
  product of morphisms componentwise. Therefore, by
  \cref{tensor_of_cliques}, given two cliques $b$ and $\Delta_A$,
  their product is also $b \times \Delta_A$.
\end{proof}

It remains to interpret the switch rule, that is, for concordance
spaces $A$, $B$ and $C$, to give a morphism
\[\fctl{\fSwitch{A,B,C}}{A\tensor(B\llpar C)}{(A\tensor B)\llpar C.}\]

\begin{lemma}\label{lemswitch}
  Let $A$, $B$ and $C$ be concordance spaces with underlying sets
  $R$, $S$ and $T$ respectively.
  The canonical isomorphism \[\fromto{R\times (S\times T)}{(R\times
    S)\times T,}\] is a morphism in $\concat$ from
   $A\tensor(B\llpar C)$ to $(A\tensor B)\llpar C$.
\end{lemma}
\begin{proof}
  We have to show that this relation maps cliques in
  $A\tensor(B\llpar C)$ to cliques in $(A\tensor B)\llpar C$. This
  means we have to show that cliques of the source and anticliques of
  the target space are treated appropriately. To reduce the number of
  brackets we suppress the bijection and treat elements of both
  concordance spaces as triples in $A\times B\times C$, and think of
  the morphism in question as the identity, which requires being
  somewhat liberal with bracketing along the way.

Assume that $A=(R,U,X)$, $B=(S,V,Y)$ and $C=(T,W,Z)$.

We have to show that 
\[\dorth{(U\tensor \orth{(Y\tensor Z)})}\subseteq
  \orth{(\orth{(U\tensor V)}\tensor Z)}.\] 
Since the operation $\orth{( \,)}$ is order-reversing it is sufficient
to show
\[\orth{(U\tensor \orth{(Y\tensor Z)})}\supseteq
  \orth{(U\tensor V)}\tensor Z.\] 
For this purpose assume that we have $p\in\orth{(U\tensor V)}$ and
$z\in Z$. We have to show that given $u\in U$ and $n\in
\orth{(Y\tensor Z)}$
\[{u\times n}\perp p\times z,\]
that is that
\[(r,s,t),(r',s',t')\in (p\times z)\cap (u\times n)\]
implies
\[r=r',\qquad s=s',\qquad t=t'.\]
There are two cases.
\begin{itemize}
\item $\{s,s'\}\in V$ and so $u\times \{s,s'\}\in U\tensor V$. In
  this case we have that
\[(r,s), (r',s')\in (u\times\{s,s'\})\cap p,\]
but $p\in\orth{(U\tensor V)}$ and so $r=r'$ and $s=s'$. Further we have
that
\[(s,t),(s,t')=(s',t')\in (\{s\}\times z)\cap n\]
and since $\{s\}\times z\in Y\times Z$ and $n\in\orth{(Y\times Z)}$
we get that $t=t'$.
\item $\{s,s'\}\in Y$, and so $\{s,s'\}\times z\in Y\tensor Z$. In
  this case we have
\[(s,t),(s',t')\in (\{s,s'\}\times z)\cap n\]
and since $n\in \orth{(Y\tensor Z)}$ this implies $s=s'$ and $t=t'$.
Accordingly we now obtain
\[(r,s),(r',s')=(r',s)\in (u\times \{s\})\cap p,\]
and since $p\in \orth{(U\tensor V)}$ and also $u\times \{s\}\in
U\tensor V$ we may conclude that $r=r'$ as required.
\end{itemize}
The proof that the given relation preserves anticliques in the
opposite direction is much the same.
\end{proof}

It is convenient to have a name for the analogous morphism from
\[(A\llpar B)\tensor C\qquad\text{to}\qquad A\llpar (B\tensor C),\]
and we use \fSigmaSwitchT{A,B,C} for this purpose.

Note that while the underlying relations for \fSwitchT{A,B,C} and
\fSigmaSwitchT{A,B,C} are isomorphisms, their inverses do
\emph{not}\/ give morphisms in the opposite direction in $\concat$.
This is desired as there is no inverse to the switch rule in deep inference.

For every $\mllm{DI}$ derivation and suitable valuation we obtain a
morphism in $\concat$, and if the derivation starts from no
assumption then this is a morphism from $I$ to the interpretation
$\intp{A}$ of the concluding formula~$A$. This morphism is uniquely
determined by the image of the clique $\{*\}$ in $I$, which is itself
a clique of $\intp{A}$, so we could alternatively consider the
interpretation of a derivation to be a clique in the concluding
formula.

\begin{proposition}\label{same_cliques_with_contexts}
  Suppose we have cliques $b, c$ for respective formulae $B, C$ and two
  morphisms $f_{DI}$ and $f_{SC}$ that both map $b$ to $c$.

  Then the morphisms $F_Sf_{DI}$ and $F_Sf_{SC}$ map $b'$ in $F_S\intp{B}$ to
  $c'$ in $F_S\intp{C}$ for any context $S$.
\end{proposition}
\begin{proof}
The construction of the functor $F_S$ means that all elements of the surrounding
context get mapped as identity morphisms.
For example, in the context
  \[S\{\,\} = \{\,\} \tensor A, \]
$F_S$ is the functor mapping $B$ to $B \otimes A$ and so
morphisms
  \[ f: \intp{B} \to \intp{C} \]
get mapped to
  \[ (f, \intp{A}): \intp{B}\tensor\intp{A} \to \intp{C}\tensor\intp{A}, \]
  where the morphism $\intp{A}$ is the identity morphism for the concordance space~$\intp{A}$.
\end{proof}

\begin{corollary}\label{idown_tensor_morphism}
  Let $c$ be the clique for a derivation
  such that $c$ is, for some $n$, a \permutation{} of some
  cross product of $n$ cliques of the form $\Delta_{A_i}$:
  \[ \Delta_{A_1} \times \dots \times \Delta_{A_n}. \]
  Then the clique after \fiDownT{} is also a \permutation{} of some cross
  product of $n+1$
  cliques of the form $\Delta_{A_i}$.
\end{corollary}
\begin{proof}
  This is a consequence of \cref{tensor_inside_outside_isomorphism}.
\end{proof}

\begin{lemma}\label{diatomsonly_cliques}
  Let \(A\) be some formula. Then the clique \(\Delta_A\) given in
  \cref{i_clique} is the clique for the formula \({\dual{A} \parr A}\) for both the
  \iDownT{} rule and the resulting clique after using the translation in
  \cref{diatomsonly} on \iDownT{}.
\end{lemma}
\begin{proof}
  It should be easy to see by \cref{idown_tensor_morphism} that all of the cases in \cref{diatomsonly}
  result in a clique of the form in \cref{i_clique}
  because the only rules used are \fiDownT{} or rules that have morphisms that are \permutations{}.
\end{proof}

\begin{proposition}\label{iup_same_cliques}
  Let $c$ be the clique for a derivation of ${\intp{(A \tensor \dual{A}) \parr B}}$
  such that $c$ is, for some $n$, a \permutation{} of some
  cross product of $n$ cliques of the form $\Delta_{A_i}$:
  \[ \Delta_{A_1} \times \dots \times \Delta_{A_n}. \]
  Then the clique after \fiUpT{} is also a \permutation{} of some cross product
  of $m < n$ cliques of the form $\Delta_{A_i}$.
\end{proposition}
\begin{proof}
  Clearly $c$ has elements of the form ${((a, a'), b)}$.
  If $a$ and $a'$ come from the same $\Delta_{A_i}$
  then $a = a'$ and the new clique is \(b\) and is a \permutation{} of
  \[ \Delta_{A_1} \times \dots \times \Delta_{A_{i-1}} \times \Delta_{A_{i+1}} \times \dots \times \Delta_{A_n}. \]
  Otherwise, they come from two different instances $\Delta_{A_i}$ and $\Delta_{A_j}$
  of concordance spaces of the same formula.
  We know that $b$ must contain the other (negative or positive) halves of $a$ and $a'$ inside of it.
  Additionally, we know from \cref{diatomsonly_cliques} that even if one of $\Delta_{A_i}$ and $\Delta_{A_j}$
  comes from \axT{} and one from \idT{} or otherwise by hypothesis a combination of non-\cutT{} rules, that
  they will have the same form.
  Therefore, the new clique after \fiUpT{} will be
  \[ \{ b \suchthat \exists a :\, ((a, a), b) \in c \text{ and } (a, a) \in \Delta_{A_i} \text{ and } (a, a) \in \Delta_{A_j} \}. \]
  or in other words
  \[ \{ b \suchthat \exists a:\, ((a, a), b) \in c \} \]
  because
  \[ \{ (a, a) \suchthat (a, a) \in \Delta_{A_i}, (a, a) \in \Delta_{A_j} \} = \{ (a, a) \suchthat a \in \Delta_{A_i} \}. \]
  This operation relating \(\Delta_{A_i} \times \Delta_{A_j} \times B\) to
  \(B\) could be described as a Cartesian division.
\end{proof}

\begin{corollary}\label{all_di_cliques_permutations}
  All $\mllm{DI}$ cliques are a \permutation{} of
  \[ \Delta_{A_1} \times \dots \times \Delta_{A_n} \]
  where $n$ is the number of atom pairs in the formula.
\end{corollary}
\begin{proof}
  For derivations without \fiUpT{} this is a result of \cref{idown_tensor_morphism}.
  Combining \cref{same_cliques_with_contexts} and \cref{iup_same_cliques}
  gives this as a corollary.
\end{proof}

\subsection{Sequent calculus derivations as cliques}\label{sequentmodel}

We may apply the ideas from the previous section to take derivations
in the sequent calculus and map them to cliques in~$\concat$.

Once again the process is recursive over the inference rules used,
and so we have to describe how to interpret each of those. The empty
sequent is again modelled by the unit~$I$ for $\llpar$. 

Above we explain how to interpret formulae, and we may extend this to
sequents by replacing comma by $\llpar$. Sequents do not come with a
bracketing structure, so we have to make a choice here, and we pick
bracketing from left to right once again as our interpretation. We expect every derivation
containing no assumptions to be interpreted by a clique in the
interpretation of the concluding sequent. Note that the non-linear
nature of derivations makes it less obvious how one might type a
derivation as a morphism. 

For the axiom rules \axT{} and \idT{} we choose the clique that is the image of $\{*\}$
under the morphism~\fDownT{\intp{A}}. For the exchange rule
note that we have a unique morphism consisting of pars of identities,
associativity and symmetry isomorphisms
such that \fExchT{A, B} is the result of:
\[
\begin{tikzcd}[column sep = tiny]
  \intp{\Gamma,A,B,\Delta} = F_S(\intp{\Gamma} \parr \intp{A}) \parr \intp{B}) \arrow[d,"F_S\fSigmaUp{}"]
  \\F_S(\intp{B} \parr (\intp{\Gamma} \parr \intp{A})) \arrow[d,"F_S\fAlphaUp{}"]
  \\F_S((\intp{B} \parr \intp{\Gamma}) \parr \intp{A}) \arrow[d,"F_S(\fSigmaUp{} \parr \intp{A})"]
  \\F_S((\intp{\Gamma} \parr \intp{B}) \parr \intp{A}) = \intp{\Gamma, B, A, \Delta}
\end{tikzcd}
\]
and we use the image of the clique of the derivation so far under
that \permutation{} to interpret the derivation resulting from adding
this rule.

In a similar way, for the $\llpar$ rule we note the existence of a unique \permutation{}
\fParT{A,B}:
\[
  \begin{tikzcd}[column sep = tiny]
    \intp{\Gamma,A,B,\Delta} = F_S(\intp{\Gamma} \parr \intp{A}) \parr \intp{B}) \arrow[d,"F_S\fSigmaUp{}"]
    \\F_S(\intp{B} \parr (\intp{\Gamma} \parr \intp{A})) \arrow[d,"F_S(\intp{B} \parr \fSigmaUp{})"]
    \\F_S(\intp{B} \parr (\intp{A} \parr \intp{\Gamma})) \arrow[d,"F_S\fAlphaUp{}"]
    \\F_S(\intp{B \parr A} \parr \intp{\Gamma}) \arrow[d,"F_S\fSigmaUp{}"]
    \\F_S(\intp{\Gamma} \parr \intp{B \parr A}) \arrow[d,"F_S(\intp{\Gamma} \parr \fSigmaUp{})"]
    \\F_S(\intp{\Gamma} \parr (\intp{A} \parr \intp{B})) = \intp{\Gamma, B, A, \Delta}
  \end{tikzcd}
\]
and we use the image of the clique interpreting the derivation so far
to interpret the derivation resulting from adding this rule.

For the $\tensor$ rule we may assume that the derivation so far has
given us cliques in $\intp{\Gamma,A}$ and ${\intp{B,\Delta} = F_S\intp{B}}$. From
these we need to create a clique in $\intp{\Gamma,A\tensor
  B,\Delta}$.

We know that the morphism for $\intp{\Gamma,A}$ can be expressed as the
composition of morphisms called~\( f_{\intp{\Gamma, A}}. \)
So we may use
\[ \intp{B} \otimes f_{\intp{\Gamma, A}} \after \rhoTensor{\intp{B}} \]
with the appropriate context functors to obtain a clique in
\[ F_S(\intp{B} \otimes (\intp{\Gamma} \parr \intp{A})) \]
to which we may apply
\[ \fSigmaSwitch{\intp{\Gamma}, \intp{A}, \intp{B}} \after \fSigmaDown{\intp{B}, \intp{\Gamma}, \intp{A}} \]
that obtains a clique in
\[ \intp{\Gamma, A \tensor B, \Delta} \]
as required:

\[
\begin{tikzcd}[column sep = tiny]
  \intp{B, \Delta} = F_S\intp{B} \arrow[d,"F_S\rhoTensor{}"]
  \\F_S(\intp{B}\otimes \intp{I}) \arrow[d,"F_S(\intp{B} \otimes f_{\intp{\Gamma, A}})"]
  \\F_S(\intp{B}\otimes (\intp{\Gamma} \parr \intp{A})) \arrow[d,"F_S\fSigmaDown{}"]
  \\F_S((\intp{\Gamma} \parr \intp{A}) \otimes \intp{B}) \arrow[d,"F_S\fSigmaSwitch{}"]
  \\F_S(\intp{\Gamma} \parr (\intp{A} \otimes \intp{B})) = \intp{\Gamma, A \otimes B, \Delta}
\end{tikzcd}
\]

\begin{corollary}\label{tensor_morphism}
  Let $a$ be the clique for a derivation of $\intp{\Gamma, A}$ and $b$ be the clique for $\intp{B, \Delta}.$
  Then the clique for \fTensorT{\intp{\Gamma, A}, \intp{B, \Delta}} is a \permutation{}
  of $a \times b$ in the obvious way.
\end{corollary}
\begin{proof}
  This is a natural consequence of \cref{tensor_inside_outside_isomorphism}.
\end{proof}

It remains to interpret the cut rule.
From the $\otimes$ rule we know that given cliques in $\intp{\Gamma, A}$ and
$\intp{\dual{A}, \Delta}$, we can have a clique in
\[\intp{\Gamma, A \otimes \dual{A}, \Delta} = F_S\intp{\Gamma, A \otimes \dual{A}}.\]
Applying the appropriate context to \fSigmaUpT{\intp{\Gamma}, \intp{A \otimes \dual {A}}}
and \fiUpT{} results in a clique in
\[ F_S\intp\Gamma = \intp{\Gamma, \Delta}\]
as required:

\[
\begin{tikzcd}[column sep = tiny]
  \intp{A, \Delta} = F_S\intp{\dual{A}} \arrow[d,"F_S\rhoTensor{}"]
  \\F_S(\intp{\dual{A}}\otimes \intp{I}) \arrow[d,"F_S(\intp{\dual{A}} \otimes f_{\intp{\Gamma, A}})"]
  \\F_S(\intp{\dual{A}}\otimes (\intp{\Gamma} \parr \intp{A})) \arrow[d,"F_S\fSigmaDown{}"]
  \\F_S((\intp{\Gamma} \parr \intp{A}) \otimes \intp{\dual{A}}) \arrow[d,"F_S\fSigmaSwitch{}"]
  \\F_S(\intp{\Gamma} \parr (\intp{A} \otimes \intp{\dual{A}})) \arrow[d,"F_S\fSigmaUp{}"]
  \\F_S((\intp{A} \otimes \intp{\dual{A}}) \parr \intp{\Gamma}) \arrow[d,"F_S\fiUp{}"]
  \\F_S\intp{\Gamma} = \intp{\Gamma, \Delta}
\end{tikzcd}
\]

The need to operate `inside' the interpretation of a sequent is
reminiscent of operating inside a context, and this is no
coincidence.

\begin{lemma}\label{scatomsonly_cliques}
  Let \(A\) be some formula. Then the clique \(\Delta_A\) given in
  \cref{i_clique} is the clique for the formula \({\dual{A} \parr A}\) for both the
  \idT{} rule and the resulting clique after using the translation in
  \cref{scatomsonly} on \idT{}.
\end{lemma}
\begin{proof}
  It should be easy to see by \cref{tensor_morphism} that all of the cases in \cref{scatomsonly}
  result in a clique of the form in \cref{i_clique}
  because the only rules used are \axT{}, \otimesT{} or rules that have morphisms that are only \permutations{}.
\end{proof}

\begin{corollary}\label{all_sc_cliques_permutations}
  All $\mllm{SC}$ cliques are a \permutation{} of
  \[ \Delta_{A_1} \times \dots \times \Delta_{A_n} \]
  where $n$ is the number of atom pairs introduced.
\end{corollary}
\begin{proof}
  For derivations without \cutT{} this is a result of \cref{tensor_morphism}.
  For the \cutT{} morphism we can apply the exact same reasoning from \cref{all_di_cliques_permutations}.
\end{proof}

\subsection{Equivalence of Transformations}

Note that almost all the morphisms chosen for the various inference
rules map cliques as simple \permutations{}. The only morphisms
that are not \permutations{} are \fiUpT{}, \fiDownT{}, \fCutT{} and
\fTensorT{}. This is a property that we take advantage
of in the following proofs.

We tackle the proof that the interpretations are preserved by the
translations given in \cref{sctodi,ditosc_direct,ditosc,equiv_cut_elimination}.

\begin{theorem}\label{i_dont_care_what_your_transformation_is}
  Any two $\mllm{}$ derivations of the same formula will produce the same cliques.
\end{theorem}
\begin{proof}
  From \cref{all_di_cliques_permutations,all_sc_cliques_permutations}
  we can see that any clique will be a \permutation{} of 
  \[ \Delta_{A_1} \times \dots \times \Delta_{A_n} \]
  where $n$ corresponds to half the number of atoms in the formula.
  Just as each \axT{} of the same atom is indistinguishable from any other \axT{},
  each $\Delta_{A_i}$ of the same atom, or even formula, are equal.
  The bracketing of any \permutation{} of $c_1$ and $c_2$ depends entirely on how the formula itself is bracketed
  and so both cliques must be equal.
\end{proof}

\begin{remark}
  If one chooses to keep track of which atoms were created together, then
  the cliques are isomorphic instead. Note that if the coherence space for an
  atom \(\alpha\) is isomorphic to that of the coherence space for an atom \(\beta\)
  then cliques where the only difference between them are the labelling of each
  atom will be isomorphic too.
  To prevent this, different atoms must be assigned to cliques that are not
  isomorphic: their sets must have different cardinalities.
\end{remark}

\begin{proposition}\label{equiv_sc}
  Let $\deriv{}$ be a sequent calculus derivation and
  let $c_{SC}$ be the clique corresponding to $\deriv{}$. Then the
  clique $c_{DI}$ corresponding to $\transscdi\deriv{}$ is
  equal to $c_{SC}$.
\end{proposition}
\begin{proof}
  We can of course use \cref{i_dont_care_what_your_transformation_is},
  but we could otherwise notice that
  all the sequent calculus morphisms are defined the same as the transformation $\transscdi$ in \cref{sctodi}.
\end{proof}

\begin{proposition}\label{equiv_di_direct}
  If $c$ is the clique corresponding to a deep inference
  derivation $\deriv{}$ then it is also the clique corresponding to
  the sequent calculus derivation $\transdiscalt\deriv{}$.
\end{proposition}
\begin{proof}
  We will give the following proof without \cref{i_dont_care_what_your_transformation_is}
  as it is interesting in its own right.

  We do this inductively on the number of rules in $\deriv{}$.

  The base case is trivial, as both cliques are defined as being equal
  for both deep inference and sequent calculus.

  For the inductive step, consider that we already have derivations $\deriv{0}$ and $\transdiscalt\deriv{0}$
    with clique~$c_0$.
  Notice that the deep inference clique $c$ resulting from any of the inference rule morphisms
    \fAlphaParT{}, \fAlphaTensorT{}, \fSigmaParT{}, \fSigmaTensorT{} and \fSwitchT{}
    is a \permutation{} of~$c_0$.
  Since \fExchT{} and \fParT{} are compositions of the above morphisms, they too are \permutations{}.
  Additionally, notice that
    the corresponding transformations to $\transdiscalt\deriv{0}$ in \cref{ditosc_direct}
    only restructure the derivation branches, reorder \parT{} rules and add \exchT{} rules.
  All of these changes to the derivation tree are, or correspond to, \permutations{}.
  Therefore, the cliques are equal.

  Thus, we only have to consider the following rules in closer detail,
  and, by \cref{same_cliques_with_contexts},
    we only need to consider the cases without contexts.
  \begin{itemize}
    \item \iDownT{}:
      By hypothesis, we have a clique $b$ which becomes $b \times \Delta_A$ after the morphism for \iDownT{}.
      By \cref{tensor_inside_outside_isomorphism},
        this is exactly the clique we get in \cref{ditosc_direct}
        because $\Gamma,$ $\Gamma',$ $\deriv{e}$ and $\deriv{1}$ are empty when we have no contexts.
      Note that with contexts,
        changing the sequent directly above, for example, $\deriv{1}$ just changes the arguments to each morphism.
      For example, if $\deriv{1}$ has an \exchT{} rule inside of it,
        then there will be a corresponding
        \[ \fExch{\intp{B}, \intp{C}} \] morphism.
      This morphism will be composed from \permutation{} morphisms including an
      associativity \permutation{} so
        \[ \fAlphaPar{\intp{A}, \intp{B}, \intp{C}} \]
        changes its arguments to
        \[ \fAlphaPar{\intp{A}, \intp{B \otimes (\dual{A} \parr A)}, \intp{C}}. \]
    \item \iUpT{}:
      By hypothesis, we have some clique $c$ as a result of both $\deriv{0}$ and $\transdiscalt\deriv{0}$.
      In deep inference, we then apply \fiUpT{}.
      In sequent calculus, we replace a \otimesT{} with \cutT{},
        which has the same effect of applying 
        a permutation followed by \fiUpT{}.
      Sequent calculus also applies and removes any \permutations{} necessary
        to get the formula, and therefore clique, into the right form.
      As above, transforming the input sequents to $\deriv{1}$ and $\deriv{2}$ just changes the arguments to each morphism.
  \end{itemize}
\end{proof}

\begin{proposition}\label{equiv_di}
If there is a clique corresponding to a deep inference derivation $\deriv{}$
  then it is also the clique corresponding to the sequent calculus
  derivation $\transdisc\deriv{}$.
\end{proposition}
\begin{proof}
  We can and will use \cref{i_dont_care_what_your_transformation_is} for this
  as proving this directly is a long and repetitive process. We give a sample of
  how to prove this directly on the \iUpT{} rule for the curious reader that wishes
  to understand how this may be done.

  For each inference rule $\rho\frac{S\{P\}}{S\{Q\}}$ from the formula $S\{P\}$
  to $S\{Q\}$, we do this inductively on the size of the context $S$.
  Informally, what is needed is to show the cliques for the proofs in
  this cube~commute:
   
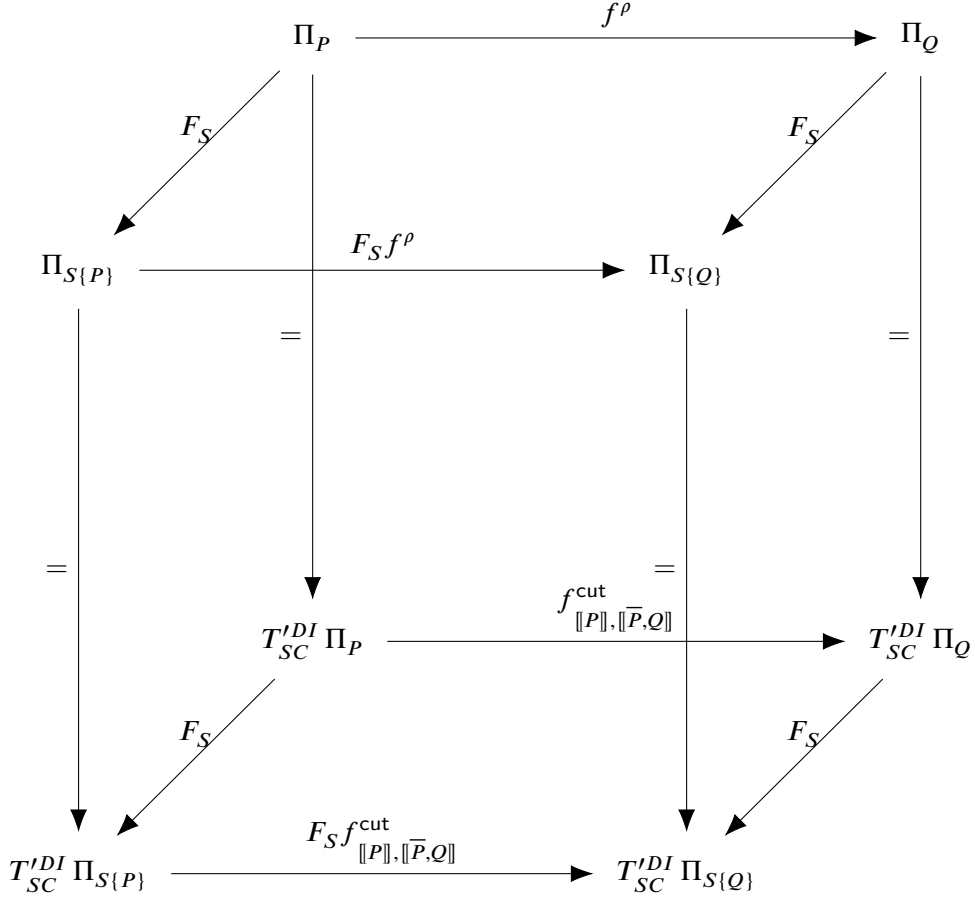
\[
    \begin{tikzpicture}
      \begin{scope}
        [scale=0.8]
        \node (PiSP)     at ( 0, 10,  0) {\(\deriv{P}\)};
        \node (DeltaSP)  at ( 0,  0,  0) {\(\transdiscalt\deriv{P}\)};
        \node (PiSPB)    at ( 0, 10, 10) {\(\deriv{S\{P\}}\)};
        \node (DeltaSPB) at ( 0,  0, 10) {\(\transdiscalt\deriv{S\{P\}}\)};
        \node (PiSQ)     at (10, 10,  0) {\(\deriv{Q}\)};
        \node (DeltaSQ)  at (10,  0,  0) {\(\transdiscalt\deriv{Q}\)};
        \node (PiSQB)    at (10, 10, 10) {\(\deriv{S\{Q\}}\)};
        \node (DeltaSQB) at (10,  0, 10) {\(\transdiscalt\deriv{S\{Q\}}\)};
        \foreach \side/\from/\to/\label in%
        {above/PiSPB/PiSQB/\(F_Sf^\rho\),
         above/PiSP/PiSPB/\(F_S\),
         above/PiSP/PiSQ/\(f^\rho\),
         above/PiSQ/PiSQB/\(F_S\),
         above/DeltaSPB/DeltaSQB/\(F_S\fCut{\intp{P}, \intp{\dual{P}, Q}}\),
         above/DeltaSP/DeltaSPB/\(F_S\),
         above/DeltaSP/DeltaSQ/\(\fCut{\intp{P}, \intp{\dual{P}, Q}}\),
         above/DeltaSQ/DeltaSQB/\(F_S\),
         left/PiSPB/DeltaSPB/=,
         left/PiSP/DeltaSP/=,
         left/PiSQ/DeltaSQ/=,
         left/PiSQB/DeltaSQB/=
        }
        {\draw[-{Latex[length=3mm]}, shorten >= 5pt, shorten <= 5pt]
        (\from)->(\to) node[midway,\side]{\label}; }
      \end{scope}
    \end{tikzpicture}
  \]
  Because the inductive step follows from \cref{same_cliques_with_contexts},
    we only need to consider the base case when $S\{P\} = P$ and $S\{Q\} = Q$.

  Consider the proofs in \cref{fig:pq_proof_i,fig:pq_proof_switch}.
  Notice that they are all created from the \idT{}, \otimesT{}, \exchT{} and
  \parT{} rules.
  Therefore, all of the cliques for these proofs will be isomorphic to either
    $\Delta_A \times \Delta_B$ or $\Delta_A \times \Delta_B \times \Delta_C$.

  Now, taking $\rho$ to be \iUpT{} we show that the cliques are the same.
      Let $c$ be the clique for \[P = (A \otimes \dual{A}) \parr B.\]
      Then in deep inference the clique for $Q = B$
        will be \[ \{b \suchthat \exists\ a \in A:\ ((a, a), b) \in c\}. \]

      For sequent calculus,
        we use the \cutT{} morphism on the clique $c$ and the clique, call it
        \(d\), that is isomorphic to $\Delta_A \times \Delta_B$.
      After the \cutT{} morphism we will get the clique
        \begin{align*}
          \{b \suchthat \exists\ ((a, a), b) \in (A \otimes \dual{A}) \parr B:\ &((a, a),
        b) \in c,
          \\&((a, a), (b, b)) \in d \}
        \end{align*}
        equal to
        \[ \{ b \suchthat \exists\ a, b:\ ((a, a), b) \in c, ((a, a), (b, b)) \in d \} \]
        which is the same as
          \[ \{ b \suchthat \exists\ a \in A:\ ((a, a), b) \in c \}. \]
\end{proof}

\begin{proposition}\label{equiv_cut}
  If $c$ is the clique corresponding to a sequent calculus
  derivation $\deriv{}$ then it is also the clique corresponding to
  the sequent calculus derivation $\transcut\deriv{}$.
\end{proposition}
\begin{proof}
  This is a result of \cref{i_dont_care_what_your_transformation_is}.
\end{proof}

\section{Conclusions}

In this paper we consider two formal systems describing
multiplicative linear logic. We have provided a new translation from
the deep inference system to the sequent calculus one, and we have
illustrated that such translations cannot be expected to be mutually
inverse due to the use of the cut rule. Our translation reduces the
number of cuts introduced and it produces smaller derivations such that
the number of axiomatic and cut rules are equal.

We formulate a necessary property for sequents or formulae to be
derivable.

We finally describe for both systems how to interpret derivations as
cliques in coherence (or concordance) spaces, based on a valuation of
atoms. These interpretations are invariant under the translations.
We may think of this as being summarised informally by the following
diagram. The triangles we may form do commute, but as argued above,
the translation processes are not mutually inverse. We have shown that
this model of linear logic is a representation of proofs that maintains
equality under the bureaucracy of both sequent calculus and deep inference.
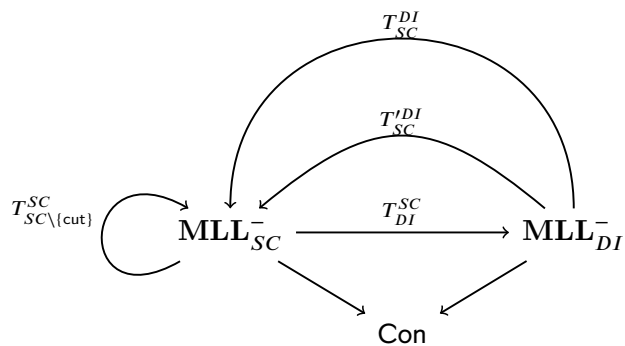
\[
\begin{tikzcd}
  \mllm{SC}
    \arrow[loop left, "\transcut"]
    \arrow{rd} %
    \arrow{rr}{\transscdi}
  & &
  \mllm{DI}
    \arrow[ll, bend left=-40, distance=2cm, outer sep=-0.45cm, "\transdiscalt"]
    \arrow[ll, bend left=-90, distance=3cm, outer sep=-0.45cm, "\transdisc"]
    \arrow[ld] %
  \\
    & \concat &
\end{tikzcd}
\]

There are of course other models of linear logic, such as those
that use a more complete or full system of linear logic, or those that make use of alternative models
that are quite different to ours,
such as handsome proof nets (based on cographs) that uses disjoint unions instead of cross products \cite{retore2003}.
Some questions that are not covered are if this model maintains equality of all
cliques of the same formula under a full system of linear logic, or if other models
do the same under transformation procedures. We leave these as future research questions.

\makeatletter
\interlinepenalty=10000
\makeatother
\bibliographystyle{alphaurl}
\bibliography{bibl}

\end{document}